\documentclass[onecolumn,10pt]{IEEEtran}
\makeatletter
\def\ps@headings{%
\def\@oddhead{\mbox{}\scriptsize\rightmark \hfil \thepage}%
\def\@evenhead{\scriptsize\thepage \hfil \leftmark\mbox{}}%
\def\@oddfoot{}%
\def\@evenfoot{}}
\makeatother
\usepackage{graphicx}
\usepackage{amsmath}
\IEEEoverridecommandlockouts
\usepackage{url}

\usepackage{enumitem}
\usepackage{mathtools,amssymb,lipsum,caption}
\usepackage{array}
\usepackage{fixltx2e}
\usepackage{amsfonts}
\usepackage{amsmath}
\usepackage{amsthm}

\newtheorem{lemma}{Lemma}
\newtheorem{definition}{Definition}
\newtheorem{proposition}{Proposition}
\newtheorem{assumption}{Assumption}
\newtheorem{remark}{Remark}
\newtheorem{corollary}{Corollary}
\newtheorem{theorem}{Theorem}

\DeclareMathOperator*{\argmin}{arg\,min}
\DeclareMathOperator*{\argmax}{arg\,max}

\begin{document}
\title{Asymptotically optimal pilot allocation over Markovian fading channels\thanks{This work has been partly funded by Huawei Technologies France SASU. A shorter version of this paper was published in the proceedings of IEEE ITW 2016, \cite{LADP16}.}}

\author{\IEEEauthorblockN{Maialen Larra\~naga\IEEEauthorrefmark{1}, 
Mohamad Assaad\IEEEauthorrefmark{1}, 
Apostolos Destounis\IEEEauthorrefmark{2},
Georgios S. Paschos \IEEEauthorrefmark{2} \\
 }
 \IEEEauthorblockA{\IEEEauthorrefmark{1}
Laboratoire des Signaux et Systemes (L2S, CNRS),  CentraleSupelec, Gif-sur-Yvette, France.\\
\IEEEauthorrefmark{2}Huawei Technologies \& Co., Mathematical and Algorithmic Sciences Lab, Boulogne Billancourt, France.}}

\maketitle

\begin{abstract}

We investigate a pilot allocation problem in wireless networks over Markovian fading channels. In wireless systems,  the Channel State Information (CSI)  is collected at the Base Station (BS),
in particular, this paper considers
a pilot-aided channel estimation method (TDD mode). 
Typically, there are less available pilots than users, hence at each slot the scheduler needs to decide an allocation of pilots to users with the goal of maximizing the long-term average throughput. 
There is an inherent tradeoff in how the limited pilots are used: assign a pilot to a user with up-to-date CSI and good channel condition for \emph{exploitation}, or assign a pilot to a user with outdated CSI for \emph{exploration}.
As we show, the arising pilot allocation problem is a restless bandit problem and thus its optimal solution is out of reach. In this paper, we propose an approximation that, through the Lagrangian relaxation approach, provides a low-complexity heuristic, the Whittle index policy. We prove this policy to be asymptotically optimal in the \emph{many users} regime (when the number of users in the system and the available pilots for channel sensing grow large). We evaluate the performance of Whittle's index policy in various scenarios and illustrate its remarkably good performance. 
\end{abstract}
\section{Introduction}

In order to support applications with large data traffic rates in the downlink, future generations of communication networks will support technologies such as multiple input multiple output (MIMO) possibly with massive antenna installations, e.g., \cite{Marzetta}. The performance of these techniques critically depends on acquiring accurate channel state information (CSI) at the transmitter, which is then used to encode the transmitting signals and null the interference at the receivers \cite{Marzetta}. 

In practice wireless channels are highly volatile, and CSI needs to be acquired very frequently. Furthermore, in both FDD (Frequency Division Duplex) and TDD (Time Division Duplex) systems only a minority of the users can be selected to provide CSI to the base station at each given time, since the resources used for CSI acquisition reduce the system efficiency. In this paper, we focus on pilot-aided CSI acquisition proposed for TDD systems. However, we mention that our framework can be applied directly to the CSI feedback context (i.e. FDD) as well.

For TDD systems downlink CSI is inferred by the uplink training symbols and the use of the reciprocity property of the channel; the process is as follows.  The BS  allocates the $M$ available pilot sequences to $M$ users out of the total $N$ users in the system. The chosen users  transmit the training  symbols to the BS which provides uplink CSI information. Last, the base station estimates  the downlink CSI exploiting the channel reciprocity. For the estimation to be successful, $M$ needs to be small to avoid the pilot contamination issue. Hence in systems with a large number of users it is expected that $M<N$. 

It has been observed that once a channel is measured and its CSI is acquired, the channel coefficients remain the same for some period of time termed \emph{channel coherence time}. In fact, sophisticated transmission schemes can exploit this channel property to avoid requesting CSI constantly. 

The problem under study in this paper, is to exploit the channel memory to optimize the allocation of pilots for CSI acquisition. To model the channel memory we consider channels that evolve according to a Markovian stochastic process and we study the pilot allocation problem over these channels. Markovian modeling of the wireless channel is commonly used in the literature to incorporate memory, e.g., to model the shadowing phenomenon, \cite{LiuZhao10}, \cite{OuyangMuruEryilShroff15}, \cite{koole2001optimal}, and \cite{cecchi2016nearly}. 

 The pilot allocation problem introduced above, with channels evolving in a Markovian fashion, can be formulated as a restless bandit problem  (RBP). RBPs are a generalization of multi-armed bandit problems (MABPs) \cite{GGW11}, sequential decision-making problems that can be seen as a particular case of Markov decision processes (MDPs). In a MABP, at each decision epoch, a scheduler  chooses which bandit\footnote{The notion of the bandit historically refers to a slot machine with an unknown reward distribution.}  to play, and a reward is obtained accordingly.  The objective is to design a bandit selection policy that maximizes the average expected reward. In MABPs the bandits that have not been played remain at the same state and provide no reward. Gittins \cite{GGW11} proved that the optimal solution of a MABPs is characterized by a simple index, known today as Gittins’ index. In the more general framework of RBPs, the statistics of all bandits evolve even in slots that are not chosen, hence the term restless. As a result, obtaining an optimal solution is typically out of reach. In \cite{Whi88}, Whittle, based on the Lagrangian relaxation approach, proposed a scheduling algorithm, the so-called Whittle's index policy, as a heuristic for solving RBPs. This has been the approach considered in this paper. 

Previous papers that are related to our work (\cite{LiuZhao10},  \cite{OuyangMuruEryilShroff15}, \cite{koole2001optimal}, \cite{JackoVillar12}, and \cite{LiuZhaoKrish}) study the Gilbert-Elliot channel model, the simplest Markovian channel model having two states, where the channel  is either in a GOOD or in a BAD state. The limitation of such binary models is that they fail to capture the complex nature of the wireless channel. Instead, here we consider a multi-dimensional Markov process, 
where each dimension corresponds to a different channel quality level representing the modulation and coding techniques used in practice to interact with the wireless channels. 
Thus, we have considered here a more challenging problem where channels are modeled by $K-$state Markov Chains, with $K$  arbitrarily large. This represents a generalization of prior binary Markovian models.

The pilot allocation problem over Markovian channels with $K>2$, can be cast as a Partially Observable Markov Decision Process (POMDP), and is an extremely challenging problem. Even the Lagrangian relaxation technique, which yields a simple index type of policy (i.e., Whittle's index policy), turns out to be very difficult to solve. One of the reasons for that is that, proving structural properties, such as threshold type of policies (the more outdated the CSI the more attractive it becomes to allocate a pilot), for an optimal POMDP allocation policy is, to the best of our knowledge, an unsolved problem, see Albright et al. \cite{albright1979structural} and Lovejoy~\cite{lovejoy1987some}. Moreover,  Cecchi et al. \cite{cecchi2016nearly} show for a similar downlink scheduling problem that threshold policies are not necessarily optimal for $K>2$. To overcome this difficulty we develop an approximation. The latter simplifies the analysis, allowing the Lagrangian relaxation technique to be applied. 

The objective of this paper is therefore to provide well performing policies for the notoriously difficult problem of pilot allocation over channels that follow Markovian laws. The main contributions of the paper are the following.
\begin{itemize}
\item We develop an approximation of the POMDP introduced above. We apply the Lagrangian relaxation technique and prove the optimality of threshold type of policies for the relaxed problem. 
\item We prove the indexability property (required for the existence of Whittle's index) and we obtain an explicit expression for Whittle's index. 
\item We derive a simple suboptimal policy for the approximation based on Whittle's index, i.e., Whittle's index policy ($WIP$). This is to the best of our knowledge the first work that provides  an explicit index for $K-$state Markov Chain channels for arbitrary $K$.   
\item We prove $WIP$ for the approximation to be asymptotically optimal in the many users setting (i.e., as the number of users and the number of available pilots grow large). The latter is an extension of the optimality results derived in \cite{OuyangEryilShroff12} for a downlink scheduling problem with Gilbert-Elliot wireless channels. 
\end{itemize}

The remainder of the paper is organized as follows. In Section~\ref{sec:description} we describe the wireless downlink scheduling problem that has been considered. In Section~\ref{sec:relaxation} we introduce an approximation  that can be solved using a Lagrangian relaxation approach. We derive a closed-form expression for the Whittle index and we define a heuristic for the original problem based on this index.  In Section~\ref{sec:error} we obtain a bound on the error introduced by the approximation. The latter serves as performance measure. In Section~\ref{asymptotic} we prove $WIP$ to be  asymptotically optimal in the many users setting. Finally, in Section~\ref{sec:numerical} we evaluate the performance of Whittle's index policy, comparing it to the performance of a myopic policy and a randomized policy, and we observe that $WIP$ captures closely the structure of the optimal policy. Most of the proofs can be found in the Appendix.
\section{Model Description}\label{sec:description}
We consider a wireless downlink scheduling problem with a single base station (BS) and $N$ users.  The channel between a user and the BS is modeled as a $K$-state Markov chain. Time is  slotted and users are synchronized.
We denote by $X_n(t)$ the channel state of user $n$ at time slot~$t$. Then $X_n(t)\in\{h_1, h_2, \ldots, h_K\}$. The state of the channel remains the same during a time slot and evolves according to the probability transition matrix $P_n=(p_{n,ij})_{i,j\in\{1,\ldots,K\}},$
where 
$p_{n,ij}=\mathbb{P}(X_n(t+1)=h_j|X_n(t)=h_i).$ 
 Channels are assumed to be independent and non-identical across users, i.e., two different users may have different probability transition matrices. The BS can not directly observe the states of the channels in the beginning of each time slot. However, this information can be acquired using pilot sequences for channel sensing. The objective is therefore to find an optimal pilot allocation policy.

We adopt the following scheduling model.
    We assume  $M$ different pilot sequences to be available to the BS for channel sensing.  In the beginning of each time slot, the BS chooses $M$ users out of $N$ (typically, $M<N$). The selected users use the allocated pilots to send the uplink training symbols.  After the training phase, the BS transmits data to all users in the system (selected for pilot allocation or not).
This mechanism allows the BS to have perfect CSI during downlink data transmission of the  selected users.   Users that have not been selected cannot provide their current CSI. Instead, the BS infers their channel state from past observations (the deduction of the belief state is explained below). 
We highlight that the results in this paper can easily be adapted for different problems such as,  downlink scheduling with ARQ feedback or scheduling in radio cognitive networks.


Next we explain the belief channel state update for the pilot allocation problem introduced above.
 Let us define $\vec b_{n}^\phi(t)$ the belief state of user $n$ during the $t^{\text{th}}$ time slot under policy $\phi$. The element $b_{n,j}^\phi(t)$ is the probability that user $n$ is in state $h_j$ in slot $t$ given all the past channel state information.
Let us denote by $a_n^\phi(\vec b_1^\phi(t),\ldots,\vec b_N^\phi(t))\in\{0,1\}$, the decision of the BS with respect to user $n$, and define for ease of notation $a_n^\phi(t):=a_n^\phi(\vec b_1^\phi(t),\ldots,\vec b_N^\phi(t))$, where $a_n^\phi(\cdot)=0$ if no pilot has been allocated to user $n$, and $a_n^\phi(\cdot)=1$ if a pilot has been allocated to user $n$ in slot $t$.
  Since at most $M$ pilots can be allocated we have
$$\sum_{n=1}^Na_n^\phi(t)\leq M.$$
 Let us denote by $S^\phi(t)=\{n\in\{1,\ldots,N\}:a_n^\phi(t)=1\}$ the set of users that have been selected in time slot $t$ under policy $\phi$.  We then define 
\begin{align*}
&\vec b_{n}^\phi(t+1):=\begin{cases}
\vec b_n^\phi(t)P_n &\text{if } n\notin S^\phi(t),\\
\vec \pi_{n,j}^1 & \text{if } n\in S^\phi(t), X_n(t)=h_j, 
\end{cases} 
\end{align*}
to be the evolution of the belief states. In the latter equation $\vec \pi_{n,j}^1=(p_{n,j1},\ldots,p_{n,jK})$ and $\vec b_n^\phi(t)$ take values in the countable state space 
$$\Pi_n=\{\vec\pi_{n,j}^{\tau}:\vec\pi_{n,j}^{\tau}=\vec e_jP_n^{\tau},\tau\in\mathbb{N}, \hbox{ and } j\in\{1,\ldots,K\}\},$$ where $\vec e_j$ is the vector with all entries 0 except the $j^{th}$ entry which equals 1. We will use the notation $\vec\pi_{n,j}^{\tau}=(p_{n,j1}^{(\tau)},\ldots,p_{n,jK}^{(\tau)})$ throughout the paper, where obviously $p_{n,ji}^{(1)}=p_{n,ji}$ for all $n,i,j$. Belief state $\vec b_n^\phi(t)=\vec\pi_{n,j}^{\tau}$ implies that  user $n$ has last been selected in slot $t-\tau$ and the observed channel state has been~$h_j$.
We note that $\vec b_n^\phi(t)$ is a sufficient statistic for the scheduling decisions and channel state information in the past, see the proof  in Smallwood et al.~\cite{SmallSondik1973}. The scheduling and the belief state updates procedure are depicted in Figure~\ref{fig:scheduling_mechanism}.
\begin{figure}
\centering
\includegraphics[scale=0.7]{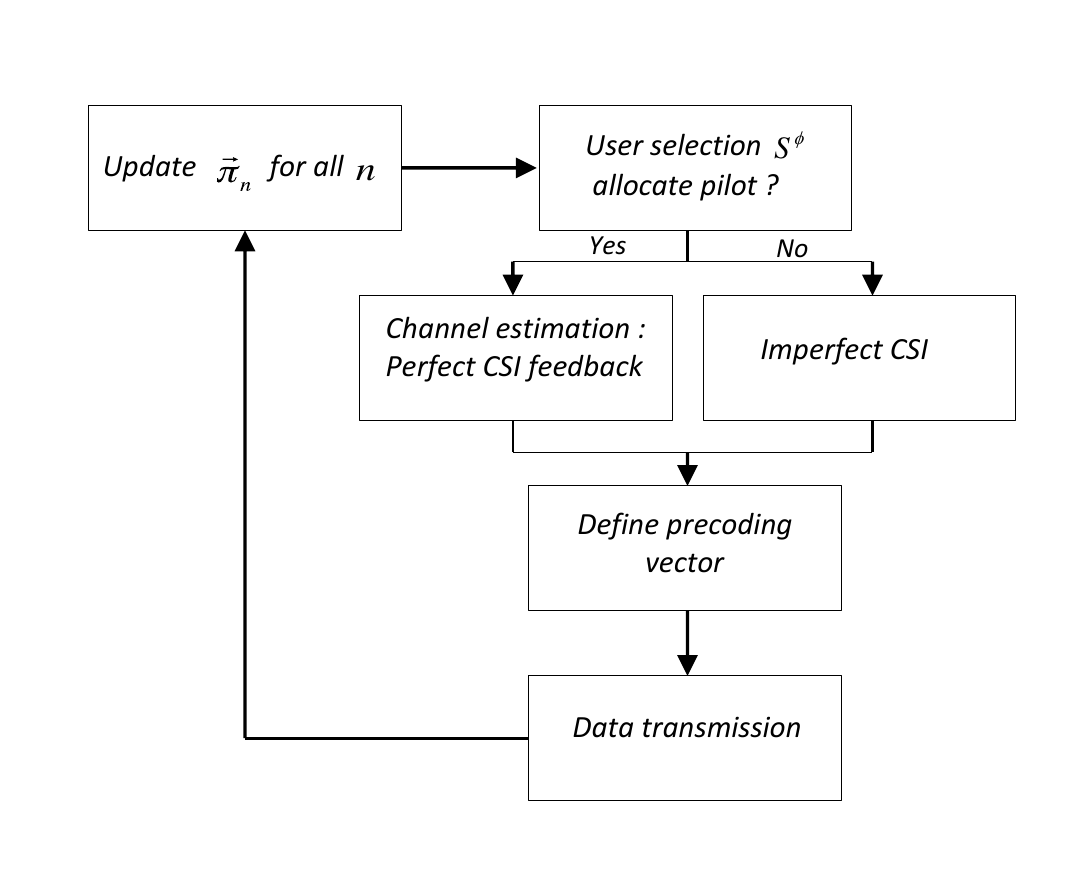}
\caption{Scheduling with pilot-aided estimation.}\label{fig:scheduling_mechanism}
\end{figure}

Next we make an assumption on $\vec\pi_{n,j}^{\tau}$ and we provide a sufficient condition for this assumption to hold.
\begin{assumption}[A1]\label{assump:belief} Let $P_n=(p_{n,ij})_{i,j\in\{1,\ldots,K\}},$ and
 $\vec\pi_{n,j}^{\tau}$ and $\vec\pi_{n,j}^{\tau'}\in \Pi_n$. We assume that, if $\tau\leq \tau'$, then
$\max_i p_{n,ji}^{(\tau)}\geq\max_i p_{n,ji}^{(\tau')},$ for all $j$.

\end{assumption}
\begin{remark}
If $P_n$ is doubly stochastic then Assumption~\ref{assump:belief} holds.
\end{remark}
Note that if the Markov chain is irreducible, and $P_n$ doubly stochastic, the belief channel vector approaches the uniform distribution as $\tau$ increases.

\subsection{Throughput maximization problem}
The objective of the present work is to efficiently allocate the available pilots to the users in the system in order to maximize the {\it long-run expected average throughput}. That is, find $\phi$ such that
\begin{align}\label{eq:obj_func_average}
\liminf_{T\to\infty}\frac{1}{T}\mathbb{E}\left(\sum_{n=1}^N\sum_{t=1}^T R_n( X_n(t),\vec b_n^\phi(t),a^\phi_n(t))\right),
\end{align}
is maximized, where $R_n(X_n(t),\vec b_n^\phi(t),a^\phi_n(t))$ is the throughput obtained by user $n$ in channel $X_n(t)$, belief vector $\vec b_n^\phi(t)$ and action $a^\phi_n(t)$. We have assumed that if a pilot has been allocated to a user, then the BS obtains full CSI of that particular user before transmitting the data. Therefore, the reward that corresponds to that user, accrued at the end of the time slot, is independent of the belief state (since the actual channel state $X_n(t)$ is revealed in the training phase). For that reason, we define $R_n(h,1):=R_n(h,\vec\pi^\tau_{n,j},1)$ to be the immediate reward obtained by user $n$ in channel state $h\in\{h_1,\ldots,h_K\}$. This is not the case for the users to whom a pilot has not been allocated. The channel state of a non-selected user is unknown even after the training phase and therefore, the reward, accrued at the end of the time slot, depends on the mismatch between the belief channel state and the real channel state. We make the following natural assumption on the  reward for not selected users, which is motivated by A1.
\begin{assumption}[A2] Let $R^1_n$ and $R_n(\vec\pi_j^\tau,0)$ be the \emph{average immediate rewards} of user $n$ under active and  passive actions, respectively. Let $R^1_n<\infty$. Then, we assume $R_n^1\geq R_n(\vec\pi_{n,j}^\tau,0)\geq R_n(\vec\pi_{n,j}^{\tau'},0),$
for all $\tau'\geq\tau$.
\end{assumption}
 The latter implies that the more outdated the CSI of a user is, the less the average reward accrued by that user will be. A trade-off emerges between exploiting users with up-to-date CSI, which provide high immediate rewards, and exploring users with outdated CSI, with potentially higher future rewards.

Although in this paper we are interested in maximizing the throughput, we note that the reward function $R_n(\cdot,\cdot,\cdot)$ could represent any function of the actual channel state and belief channel state of user $n$, and the action (allocate a pilot or not) taken on user $n$. The results provided in this paper hold for any function $R$ that satisfies Assumption A2.


While~\eqref{eq:obj_func_average} is a typical performance measure, it is not obvious at all how to deal with it.  In many existing works, e.g., \cite{LiuZhao10}, a discounted reward function is used.  In this work, we deal with ~\eqref{eq:obj_func_average} as follows. We first  consider the
 {\it discounted reward over the infinite horizon:} find $\phi$ such that
\begin{align}\label{eq:obj_func_discounted}
\liminf_{T\to\infty}\frac{1}{\sum_{t=1}^T\beta^{t-1}}\mathbb{E}\left(\sum_{n=1}^N\sum_{t=1}^T\beta^{t-1}R_n(X_n(t),\vec b_n^\phi(t),a^\phi_n(t))\right),
\end{align}
is maximized,
with $0\leq\beta<1$ the discount factor. We then retrieve the solution of~\eqref{eq:obj_func_average} as a limit of the discounted reward model (i.e., letting the discount factor $\beta\to 1$). This limit is not straightforward since certain conditions on Equation~\eqref{eq:obj_func_discounted}, \cite[Chap. 8.10]{Puterman2005} must be verified. The proof can be found in Appendix~\ref{append:discountedLimit}.

\section{Lagrangian relaxation and Whittle's index}\label{sec:relaxation}
The model introduced above falls in the framework of RBP problems.
Each user $n\in\{1,\ldots,N\}$ present in the system can be seen as bandit or arm. The state of each arm represents the belief channel state of the user.  RBPs have been shown to be PSPACE-hard, see Papadimitriou et al. \cite{papadimitriou1999complexity}. A well established method for solving RBPs is the Lagrangian relaxation introduced by Whittle in~\cite{Whi88}.

The Lagrangian relaxation technique consists in relaxing the constraint on the available resources, by letting it be satisfied on average and not in every time slot, that is,
\begin{align}\label{relaxation_constraint}
\sum_{n=1}^Na_n^\phi(t)\leq M \Rightarrow \lim_{T\to \infty}\frac{1}{T}\mathbb{E}\left(\sum_{t=1}^T\sum_{n=1}^Na_n^\phi(t)\right)\leq M,
\end{align}
in the expected average reward model, and 
\begin{align}\label{relaxation_constraint_discounted}
&\sum_{n=1}^Na_n^\phi(t)\leq M \Rightarrow \lim_{T\to \infty}\frac{1}{\sum_{t=1}^T\beta^{t-1}}\mathbb{E}\left(\sum_{t=1}^T\sum_{n=1}^N\beta^{t-1} a_n^\phi(t)\right)\leq M,
\end{align}
in the discounted model with $0\leq \beta <1$.  The Objective function~\eqref{eq:obj_func_discounted} together with the relaxed constraint~\eqref{relaxation_constraint_discounted} constitute a  Partially Observable Markov Decision Process (POMDP), and we will refer to it as the \emph{$\beta$-discounted relaxed POMDP} throughout the paper. The particular case of $\beta=1$ applies to the expected long-run average reward model in Equation~\eqref{eq:obj_func_average} and Constraint~\eqref{relaxation_constraint}. We will refer to the latter simply as the \emph{relaxed POMDP}. The solution for the $\beta$-discounted relaxed POMDP can be derived as follows: find a policy $\phi$ such that
\begin{align}\label{eq:beta-discounted-relaxed}
\liminf_{T\to\infty}\frac{1}{\sum_{t=1}^T\beta^{t-1}}\mathbb{E}\bigg(&\sum_{t=1}^T\beta^{t-1}\bigg(\sum_{n=1}^NR_n(X_n(t),\vec b_n^\phi,a_n^\phi(t))+W(M-N+\sum_{n=1}^N(1-a_n^\phi(t)))\bigg)\bigg),
\end{align}
is maximized, where $W$ is a Lagrange multiplier and can be seen as a \emph{subsidy for passivity} (or equivalently, penalty for activity).
Observe that, in problem~\eqref{eq:beta-discounted-relaxed}, users become independent from each other and the $\beta$-discounted relaxed POMDP can be decomposed into $N$ uni-dimensional optimization problems, that is, the problem is to find a policy $\phi$ such that
\begin{align}\label{eq:beta-discounted-relaxed-unidimensinal}
\liminf_{T\to\infty}\frac{1}{\sum_{t=1}^T\beta^{t-1}}\mathbb{E}\bigg(&\sum_{t=1}^T\beta^{t-1}\bigg(R_n(X_n(t),\vec b^\phi_n,a^\phi_n(t))    -W(1-a^\phi_n(t))\bigg)\bigg),
\end{align}
is maximized
for all $n\in\{1,\ldots,N\}$. The solution of the $\beta$-discounted relaxed POMDP is an index type of policy, and can be obtained by combining the solution  of problem~\eqref{eq:beta-discounted-relaxed-unidimensinal} for all $n$. More specifically, the solution is characterized by the Whittle index (see Section~\ref{sec:WIP} for a formal definition of Whittle's index, and Whittle \cite{Whi88} for the first results on Whittle's index theory). An {\it index} can be seen as a value, that is assigned to a user in a given state, that measures the gain obtained by activating the user in that particular state. The index depends only on the parameters of that user. An index policy, is simply a policy that is characterized by those indices. An example of a simple index policy is a myopic policy, where the index reduces to the immediate reward gained by each user in the current state. Index policies, in particular Whittle's index, have become extremely popular in recent years due to their simplicity, see Liu et al.~\cite{LiuZhao10}, Ouyang et al.~\cite{OuyangMuruEryilShroff15}, and  Cecchi et al.~\cite{cecchi2016nearly} for a few examples related to the present work.
 
Next we will explain how to obtain Whittle's index for problem~\eqref{eq:beta-discounted-relaxed-unidimensinal} for all $n$.
We  drop the user index from the notation since we will focus on one dimensional problems. A general recipe to compute Whittle's index is to: (i) prove some structure on the solution of problem~\eqref{eq:beta-discounted-relaxed-unidimensinal} (usually optimality of threshold policies), (ii) show that the indexability property holds (which ensures Whittle's index to exist), (iii) derive an explicit expression for Whittle's index and (iv) define Whittle's index policy.
For this particular problem, proving threshold type of policies to be optimal has shown to be extremely challenging, except in the 2-state Markov channel systems (Gilbert-Elliot model), see Albright \cite{albright1979structural} and Lovejoy~\cite{lovejoy1987some}. To the best of our knowledge, all the research work done in this area has focused on either i.i.d. channel models or the Gilbert-Elliot channel model. In the more general case of $K$-state Markov channel models, with arbitrary $K$, no results are known.

In the present work, we have considered an approximation that allows to obtain Whittle's indices for arbitrarily large Markov channel models. To define this approximation recall the POMDP under study. The action space is defined by $\{0,1\}$, the set of belief states is given by $\Pi$ and the channel state transitions are characterized by the transition matrix $P=(p_{ij})_{i,j\in\{1,\ldots,K\}}$. Let us define $q^a(\vec\pi^\tau_i,\vec\pi^{\tau'}_j)$ to be the transition probability from belief state $\vec\pi^\tau_i$ to belief state $\vec\pi^{\tau'}_j$ conditioned on action $a\in\{0,1\}$. The transition probabilities that characterize the original POMDP are given as follows:
\begin{equation}\label{transitionsPOMDP_passive}
q^0(\vec\pi^\tau_i,\vec\pi^{\tau'}_j)=\begin{cases}
1, &\hbox{if } j=i \hbox{ and } \tau'=\tau+1,\\
0, & \hbox{otherwise},
\end{cases} 
\end{equation}
 and 
\begin{equation}\label{transitionsPOMDP_active}
q^1(\vec\pi^\tau_i,\vec\pi^{\tau'}_j)=
\begin{cases}
p^{(\tau)}_{ij} &\hbox{if } \tau'=1,\\
0, & \hbox{otherwise}.
\end{cases}  
\end{equation}
We next define the approximation, for which a complete analysis of Whittle's index policy can be performed.

\noindent
{\it Approximation:} We assume a POMDP with action space $\{0,1\}$, belief state space $\Pi$ and transition probabilities
\begin{equation}\label{equation:app}
q^1(\vec\pi^\tau_i,\vec\pi^{\tau'}_j)=
\begin{cases}
p^{s}_{j} &\hbox{if } \tau'=1,\\
0, & \hbox{otherwise},
\end{cases}  
\end{equation}
where $p_j^s$ is the steady-state probability of channel $h_j$, and $q^0(\cdot,\cdot)$ as defined in Equation~\eqref{transitionsPOMDP_passive}. That is, we assume that under passive action the transition probabilities are identical to that of the original POMDP, and that under active action, the transitions are governed by the steady-state probabilities.

A priori this approximation looks suitable for problems in which $N$ is much larger than $M$, since we expect users not to be selected for long time frames (and therefore the belief vector is closer to $(p_1^s,\ldots,p_K^s)$). We will observe in Section~\ref{sec:index} (Remark~\ref{remark:omega}) however, that if instead of taking $q^1(\vec\pi^\tau_i,\vec\pi^{\tau'}_j)=p_j^s$ we had taken  $q^1(\vec\pi^\tau_i,\vec\pi^{1}_j)=p_{ij}^{(r)}$ with $r$ independent of $\tau$ the heuristic we obtain is the same. In Section~\ref{sec:approxnumerics} we numerically evaluate the accuracy of this approximation.

\subsection{Threshold policies}
As mentioned in the previous section a possible first step into obtaining Whittle's index is to prove threshold type of policies to be optimal for the one dimensional optimization problem in Equation~\eqref{eq:beta-discounted-relaxed-unidimensinal}. A threshold policy can be described as follows. Let $\vec\Gamma$ be a vector of positive values. Then the action regarding a user in belief state $\vec\pi_{j}^{\tau}$ is $a=1$ (active action) if $\tau> \Gamma_j$ and $a=0$ (passive action) otherwise.
However, for the downlink problem with $K>2$  threshold policies are not necessarily optimal, Cecchi~\cite{cecchi2016nearly}. In this section, we prove threshold type of policies to be optimal for the approximation introduced above. 

We next give a formal definition of threshold policies.  
\begin{definition}\label{def:threshold}
We say that $\phi$ is a threshold type of policy if it prescribes action $a\in\{0,1\}$ in all states $\vec\pi_j^\tau$ such that $\tau\leq\Gamma_j$ and prescribes action $a'\in\{0,1\}$ with $a'\neq a$ for all $\vec\pi_j^\tau$ where $\tau>\Gamma_j$, $j\in\{1,\ldots,K\}$ and $\vec\Gamma=(\Gamma_1,\ldots,\Gamma_K)$. Such a threshold policy  will be referred to as policy $\vec\Gamma$.
\end{definition} 

We will focus on the discounted reward model in~\eqref{eq:beta-discounted-relaxed-unidimensinal}. The Bellman optimality equation writes
\begin{align}\label{eq:discounted_optimality}
 V_\beta^{app}(\vec\pi_{j}^\tau)=\max\{&R(\vec\pi_{j}^\tau,0)+ W +\beta V_\beta^{app}(\vec\pi_{j}^{\tau+1});     R^1  + \beta\sum_{k=1}^Kp_k^sV_\beta^{app}(\vec\pi_{k}^{1}))\},
\end{align}
 where $W$ is the subsidy for passivity.
 In the latter equation the function $V_\beta^{app}$ is the value function that corresponds to the discounted one dimensional problem given in Equation~\eqref{eq:beta-discounted-relaxed-unidimensinal}, and although not made explicit in the notation it also depends on~$W$.

In the next theorem we prove that  threshold type of policies are an optimal solution for~\eqref{eq:beta-discounted-relaxed-unidimensinal}. The proof can be found in Appendix~\ref{app:proofthreshold}.
\begin{theorem}[Discounted reward threshold]\label{prop:threshold}
Assume that A\ref{assump:belief} and A2 hold and let $W$ be fixed. Then there exist $\Gamma_1,\ldots,\Gamma_K\in\{0,1,\ldots\}$ such that the threshold policy $\vec\Gamma=(\Gamma_1,\ldots,\Gamma_K)$ is an optimal solution for problem~\eqref{eq:beta-discounted-relaxed-unidimensinal} for all $0\leq\beta<1$.
\end{theorem}

Having proven the structure of the optimal policy, the explicit expression of $V_\beta^{app}$ can be obtained. The latter enables to prove conditions  8.10.1- 8.10.4' in Puterman~\cite{Puterman2005}, see Appendix~\ref{append:discountedLimit}.  It then can be shown that the one-dimensional long-run expected average reward,  equals $\lim_{\beta\to1}(1- \beta)V_\beta^{app}$, see~\cite[Th. 8.10.7]{Puterman2005}. Moreover, these conditions imply that (i) an optimal stationary policy exists, and (ii) the optimality equation for the average reward model, i.e.,
\begin{align}\label{eq:average_optimality}
 V^{app}(\vec\pi_{j}^\tau)+g(W)=\max\{&R(\vec\pi_{j}^\tau,0)+ W + V^{app}(\vec\pi_{j}^{\tau+1});     R^1  + \sum_{k=1}^Kp_k^sV^{app}(\vec\pi_{k}^{1}))\},
\end{align}
 has a solution. In the latter equation $g(W)$ refers to the average reward which can be obtained by $\lim_{\beta\to1}(1- \beta)V_\beta^{app}$.
In the following theorem we show that threshold type of policies are an optimal solution of the average reward model too. 
\begin{theorem}[Average reward threshold]\label{prop:threshold_average}
Assume that A\ref{assump:belief} and A2 hold and let $W$ be fixed. Then there exist $\Gamma_1,\ldots,\Gamma_K\in\{0,1,\ldots\}$ such that the threshold policy $\vec\Gamma=(\Gamma_1,\ldots,\Gamma_K)$ is an optimal solution for problem~\eqref{eq:beta-discounted-relaxed-unidimensinal} for $\beta=1$.
\end{theorem}

\begin{proof}
 For ease of notation we drop the superscript $app$.
We want to prove that  if it is optimal to select the user  in state $\vec \pi_j^{\tau}$ then it is also optimal to select the user in state $\vec \pi_j^{\tau+1}$. From Equation~\eqref{eq:average_optimality}, the latter statement translates to showing that 
\begin{align*}
R^1+\sum_{k=1}^Kp_k^sV(\vec\pi_k^1)\geq R(\vec\pi_j^\tau,0)+W+V(\vec\pi_j^{\tau+1}),
\end{align*}
implies
\begin{align*}
R^1+\sum_{k=1}^Kp_k^sV(\vec\pi_k^1)\geq R(\vec\pi_j^{\tau+1},0)+W+V(\vec\pi_j^{\tau+2}).
\end{align*}
To prove this implication it suffices to show that 
\begin{align}\label{eq-another_equation}
R(\vec\pi_j^\tau,0)+W+V(\vec\pi_j^{\tau+1})\geq R(\vec\pi_j^{\tau+1},0)+W+V(\vec\pi_j^{\tau+2}).
\end{align}
Due to A2 (i.e., $R(\vec\pi_j^\tau,0)\geq R(\vec\pi_j^{\tau+1},0)$ for all $\tau>0$), to show~\eqref{eq-another_equation}, it suffices to show $V(\vec\pi_j^{\tau+1})\geq V(\vec\pi_j^{\tau+2})$ for all $j$ and all $\tau>0$. That is, $V(\cdot)$ being non-increasing. In order to prove the latter, we will use the value iteration approach Puterman~\cite[Chap. 8]{Puterman2005}. Define $V_{0}(\vec\pi_j^\tau)=0$ for all $j\in\{1,\ldots,K\}$ and $\tau>0$ and 
\begin{align*}
V_{r+1}(\vec\pi_j^\tau)=\max\{&R(\vec\pi_j^\tau,0)+W+V_{r}(\vec\pi_j^{\tau+1}),   R^1+\sum_{k=1}^Kp_{k}^{s}V_{r}(\vec\pi_k^1)\},
\end{align*}
with $g(W)=V_{r+1}(\vec\pi_j^\tau)-V_{r}(\vec\pi_j^\tau)$.
Observe that $V_{0}(\vec\pi_j^\tau)=0$ satisfies the non-increasing property. We assume that $V_{r}(\vec\pi_j^\tau)$ satisfies it for all $j\in\{1,\ldots,K\}$ and all $\tau>0$, and we prove that $V_{r+1}(\vec\pi_j^\tau)$ is non-increasing as well. The latter can be proven using the arguments used in the proof of Theorem~\ref{prop:threshold}. We therefore skip the calculations here.

After proving $V_r(\cdot)$ to be non-increasing and since $\lim_{r\to\infty}V_r(\cdot)=V(\cdot)$ (which holds after verification of mild assumptions), $V_r$ being non-increasing implies $V$ being non-increasing. This concludes the proof.
\end{proof}
We have proven that an stationary solution for the average reward model exists and that the Bellman optimality equation has a threshold type of solution. Therefore, we concentrate on the average reward model to obtain Whittle's index policy.

\subsection{Indexability and Whittle's index}\label{sec:index}
In this section we prove the problem to be indexable. Indexability is the property that ensures Whittle's index to exist. It establishes that as the Lagrange multiplier $W$ increases, the set of states in which the optimal action is the passive action increases. In the following we formally define this property.
\begin{definition}\label{def1} Let $\vec\Gamma(W)$ be an optimal threshold policy for a fixed subsidy $W$. We define the set
$\mathcal{L}(W):=\{\vec\pi_j^{\tau}\in\Pi, \tau>0, \hbox{ and } j\in\{1,\ldots,K\}:\tau\leq\Gamma_j(W)\}$, i.e., the set of all belief states in which passive action is prescribed by policy $\vec\Gamma(W)$.
\end{definition} 
\begin{definition}
Let $\mathcal{L}(W)\subseteq\Pi$ be as defined in Definition~\ref{def1}. Then a bandit is said to be {\it indexable} if $\mathcal{L}(W)\subseteq\mathcal{L}(W')$ for all $W<W'$, i.e., the set of belief states in which passive action is prescribed by an optimal policy of the relaxed problem increases as $W$ increases. A RBP is indexable if all bandits are indexable.
\end{definition}
Although indexability seems a natural property not all problems satisfy this condition; a few examples are given in Hodge et al.~\cite{hodge2011dynamic} and Whittle~\cite{Whi88}.
Next we prove the indexability property.

\begin{proposition}\label{prop:indexability}
All users are indexable.
\end{proposition}
\begin{proof}
To prove indexability, i.e., $\mathcal{L}(W)\subseteq\mathcal{L}(W')$ for all $W<W'$, one needs to show that  $\vec\Gamma(W)\leq\vec\Gamma(W')$ for all $W<W'$ (where $\leq$ stands for $\Gamma_i(W)\leq\Gamma_i(W')$ for all $i\in\{1,\ldots,K\}$). The latter equivalence is implied by the fact that an optimal solution of problem~\eqref{eq:average_optimality} is of threshold type (Theorem~\ref{prop:threshold_average}). 

Let $\alpha^{\vec\Gamma(W)}(\vec\pi_j^\tau)$ be the steady-state probability of being in state $\vec\pi_j^\tau$ under threshold policy $\vec\Gamma(W)$. Having proven threshold type of policies to be an optimal solution, for a user to be indexable it suffices to show that 
$$\sum_{j=1}^K\sum_{r=1}^{\Gamma_j(W)}\alpha^{\vec\Gamma(W)}(\vec\pi_j^r)\leq \sum_{j=1}^K\sum_{r=1}^{\Gamma_j(W')}\alpha^{\vec\Gamma(W')}(\vec\pi_j^r),$$
if $\vec\Gamma(W)\leq\vec\Gamma(W')$. That is, the probability of being in passive mode is greater as the threshold increases. Note that under threshold policy $\vec\Gamma(W)$ $\alpha^{\vec\Gamma(W)}(\vec\pi_j^r)=\frac{\omega_j}{\sum_{k=1}^K(\Gamma_k(W)+1)\omega_k}$ for all $r\in\{1,\ldots,\Gamma_j(W)+1\}$, where $\omega_j$ is computed in Appendix~\ref{app:omegas}, and therefore 
\begin{align*}
&\sum_{j=1}^K\sum_{r=1}^{\Gamma_j(W)}\alpha^{\vec\Gamma(W)}(\vec\pi_j^r)=\frac{\sum_{j=1}^K\Gamma_j(W)\omega_j}{\sum_{k=1}^K(\Gamma_k(W)+1)\omega_k}     \leq\frac{\sum_{j=1}^K\Gamma_j(W')\omega_j}{\sum_{k=1}^K(\Gamma_k(W')+1)\omega_k}=\sum_{j=1}^K\sum_{r=1}^{\Gamma_j(W')}\alpha^{\vec\Gamma(W')}(\vec\pi_j^r),
\end{align*}
since $\vec\Gamma(W)\leq\vec\Gamma(W')$. Therefore users are indexable.
\end{proof}
Having proven indexability Whittle's index can be defined as follows.
\begin{definition} Whittle's index in state $\pi_j^\tau$ is defined as the smallest value of $W$ such that an optimal policy of the single-arm POMDP is indifferent of the action taken in $\pi_j^\tau$.
\end{definition}

We can now proceed to solve Whittle's index. Let us define $\mathcal{T}(\vec\Gamma)=\{\vec\Gamma'=(\Gamma_1',\ldots,\Gamma_K') \text{ with } \Gamma_i'\in\mathbb{N}\cup\{0\} \text{ for all } i: \vec\Gamma'>\vec\Gamma\}$, that is, the set of all threshold policies that are {\it greater}  than $\vec\Gamma$ ({\it i.e.,} $\vec\Gamma'>\vec\Gamma \Leftrightarrow \Gamma_j'\geq\Gamma_j$ for all $j$ and $\Gamma\neq\Gamma'$). In particular, we denote $\mathcal{T}(0)=\{\vec\Gamma'=(\Gamma_1',\ldots,\Gamma_K') \text{ with } \Gamma_i'\in\mathbb{N}\cup\{0\} \text{ for all } i: \vec\Gamma'>(0,\ldots,0)\}$.  Let $\alpha^{\vec\Gamma}(\vec\pi_j^\tau)$ be the steady-state probability of being in state $\vec\pi^\tau_j$ under policy $\vec\Gamma$, and let $b^{\vec\Gamma}$ the steady-state belief state under policy $\vec\Gamma$. It then can be shown that 
\begin{align*}
&\lim_{\beta\to1}(1-\beta)V_\beta^{app}(\cdot)=g^{\vec\Gamma}(W)=\mathbb{E}(R( b^{\vec\Gamma},a^{\vec\Gamma}(b^{\vec\Gamma})))+W\sum_{k=1}^K\sum_{i=1}^{\Gamma_k}\alpha^{\vec\Gamma}(\vec\pi_k^i),
\end{align*}
 where $g^{\vec\Gamma}(W)$ is the average reward under policy $\vec\Gamma$ when the subsidy for passivity equals $W$. Whittle's index for the average reward problem can then be computed as explained in the next theorem. The proof can be found in Appendix~\ref{app:Whittles index}. 
\begin{theorem}\label{prop:whittle}
Assume that an optimal solution of the single-arm POMDP is of threshold type and that $\sum_{k=1}^K\sum_{r=1}^{\Gamma_k}\alpha^{\vec\Gamma}(\vec\pi_k^r)$ is non-decreasing in $\vec\Gamma$. Then the problem is indexable and Whittle's index for user $n$ is computed as follows (we omit the dependence on $n$ from the notation):

 {\it Step i:} Compute
\small{\begin{equation*}
W_{i}=\inf_{\vec\Gamma\in \mathcal{T}(\vec\Gamma^{i-1})}\frac{\mathbb{E}(R(b^{\vec\Gamma^{i-1}},a^{\vec\Gamma^{i-1}}(b^{\vec\Gamma^{i-1}})))-\mathbb{E}(R(b^{\vec\Gamma},a^{\vec\Gamma}(b^{\vec\Gamma})))}{\sum_{j=1}^{K}\left(\sum_{r=1}^{\Gamma_j}\alpha^{\vec\Gamma}(\vec\pi^r_j)-\sum_{r=1}^{\Gamma_j^{i-1}}\alpha^{\vec\Gamma^{i-1}}(\vec\pi^r_j)\right)},
\end{equation*}}
for all $i\geq0$, where $\vec\Gamma^{-1}=\vec0$. Denote by $\vec\Gamma^{i}$ the largest minimizer for all $i>0$. We define $W(\vec\pi_j^\tau):=W_{i}$ for each $j$, such that $\Gamma^{i-1}_j<\tau\leq\Gamma^{i}_j$. If $\vec\Gamma^{i}_j=\infty$ for all $j$ then stop, otherwise go to Step $i+1$. When the algorithm stops the Whittle index for all $\pi_j^\tau$ has been obtained and is given by $W(\pi_j^\tau)$.
\end{theorem}
In the following lemma and corollary we derive an explicit expression for Whittle's index. The proof of the lemma can be found in Appendix~\ref{app:explicit_Whittle}.
\begin{lemma}\label{prop:expressionwhittle}
 If  in Step i of Theorem~\ref{prop:whittle} for an i$>0$, the minimizer $\vec\Gamma^i$ is such that $\sum_{j=1}^K\Gamma_j^i=(\sum_{j=1}^K\Gamma_j^{i-1})+1$ and $\Gamma_j^i\geq\Gamma_j^{i-1}$ for all $j\in\{1,\ldots,K\}$, then
$$
W_i=R^1+\sum_{k=1}^K\sum_{j=1}^{\Gamma_k^{i-1}}R(\vec\pi_k^j,0)\omega_k-R(\vec\pi_u^{\Gamma_u^i},0)\sum_{k=1}^K(\Gamma_k^{i-1}+1)\omega_k,
$$
with $u$ such that $\Gamma_u^i=\Gamma_u^{i-1}+1$.
\end{lemma}

In the next corollary,  we prove that Whittle's index can be easily computed and is non-decreasing in $\tau$.
\begin{corollary}\label{cor:whittle}
Let us define $u^0=\argmax_{u\in\{1,\ldots,K\}}R(\vec\pi_u^1,0)$, and $\vec\Gamma^0=\vec e_{u^0}$, with $\vec e_{u^0}$ the vector with all entries 0 except the $u^0$th element which equals 1. Define
\begin{align}\label{eq:index_definition}
u^i&=\argmax_{u\in\{1,\ldots,K\}} R(\pi_u^{\Gamma^{i-1}_u+1},0),  \hbox{ and, }\nonumber\\
\vec\Gamma^i&=\left\{\sum_{r=0}^i\mathbf{1}_{\{u^r=1\}},\ldots, \sum_{r=0}^i\mathbf{1}_{\{u^r=K\}}\right\},\hbox{ for all } i>0,
\end{align}
where $\mathbf{1}$ refers to the indicator function.
Then
\begin{align*}
W(\vec\pi_{u^j}^{\Gamma^j_{u^j}})=&R^1+\sum_{k=1}^K\sum_{r=1}^{\Gamma_k^{j-1}}R(\vec\pi_k^r,0)\omega_k    -R(\vec\pi_{u_j}^{\Gamma^j_{u^j}},0)\sum_{k=1}^K(\Gamma_k^{j-1}+1)\omega_k, \hbox{ for all  }j\geq0.
\end{align*}
 Whittle's index, $W(\vec\pi_k^\tau)$, is non-decreasing in $\tau$ for all $k$.
\end{corollary} 
\begin{proof}
Let $u^i$ and $\vec\Gamma^i$ be defined as in Equation~\eqref{eq:index_definition}, and let $W_i$ be
$$
R^1+\sum_{k=1}^K\sum_{j=1}^{\Gamma_k^{i-1}}R(\vec\pi_k^j,0)\omega_k-R(\vec\pi_{u^i}^{\Gamma_{u^i}^{i-1}+1},0)\sum_{k=1}^K(\Gamma_k^{i-1}+1)\omega_k.
$$
 We aim at proving that 
\begin{align}\label{eqref:just_an_equation}
W_i\leq\frac{\mathbb{E}(R(b^{\vec\Gamma^{i-1}},a^{\vec\Gamma^{i-1}}(b^{\vec\Gamma^{i-1}})))-\mathbb{E}(R(b^{\vec\Gamma},a^{\vec\Gamma}(b^{\vec\Gamma})))}{\sum_{j=1}^{K}\left(\sum_{r=1}^{\Gamma_j}\alpha^{\vec\Gamma}(\vec\pi^r_j)-\sum_{r=1}^{\Gamma_j^{i-1}}\alpha^{\vec\Gamma^{i-1}}(\vec\pi^r_j)\right)},
\end{align}
for all $\vec\Gamma$ for which $\sum_{j=1}^K\Gamma_j>\sum_{j=1}^K\Gamma_j^{i-1}$ and $\Gamma_j\geq\Gamma_j^{i-1}$ for all $j$. Using the same arguments as those used in proof of Lemma~\ref{prop:expressionwhittle} the RHS in~\eqref{eqref:just_an_equation} simplifies to
\begin{align}\label{eazet}
\bigg(&\sum_{k=1}^K\sum_{j=1}^{\Gamma_k^{i-1}}R(\vec\pi_k^i,0)\omega_k\sum_{r=1}^Kv_r\omega_r
-\sum_{k=1}^K\sum_{j=\Gamma_k^{i-1}+1}^{\Gamma_k^{i-1}+u_k}R(\vec\pi_k^j,0)\omega_k\sum_{r=1}^K(\Gamma_r^{i-1}+1)\omega_r
+R^1\sum_{k=1}^K\omega_k\sum_{r=1}^Kv_r\omega_r\bigg)\cdot\bigg(\sum_{k=1}^Kv_k\omega_k\bigg)^{-1},
\end{align}
where we defined $\Gamma_j:=\Gamma_j^{i-1}+v_j$ with $v_j\geq0$ and $\sum_ {j=1}^Kv_j>~0$. 
 We have that
\begin{align*}
&\hbox{RHS of } \eqref{eqref:just_an_equation}\geq\eqref{eazet}\\
&\geq \sum_{k=1}^K\sum_{j=1}^{\Gamma_k^{i-1}}R(\vec\pi^i_k,0)\omega_k+R^1   -\frac{\sum_{k=1}^KR(\vec\pi_k^{\Gamma_k^{i-1}+1},0)v_k\omega_k\sum_{r=1}^K(\Gamma_r^{i-1}+1)\omega_r}{\sum_{k=1}^Kv_k\omega_k}\\
&\geq\sum_{k=1}^K\sum_{j=1}^{\Gamma_k^{i-1}}R(\vec\pi^i_k,0)\omega_k+R^1-R(\vec\pi_{u^i}^{\Gamma_{u^i}^{i-1}+1},0)\sum_{r=1}^K(\Gamma_r^{i-1}+1)\omega_r\\
&=W_i,
\end{align*}
where recall that $u^i=\argmax_{k}\{R(\vec\pi_{k}^{\Gamma_{k}^{i-1}+1},0)\}$. The second inequality follows from Assumption~A2 and the third inequality is due to the definition of $u^i$.
We have therefore proven~\eqref{eqref:just_an_equation}, which implies that  $\Gamma_j^i=\Gamma_j^{i-1}$ for all $j\neq u^i$ and $\Gamma_{u^i}^i=\Gamma_{u^i}^{i-1}+1$. By Theorem~\ref{prop:whittle}, $W(\vec\pi_j^\tau)=W_i$ for all $\Gamma_j^{i-1}<\tau\leq\Gamma_j^i$, and we have proven that if $j=u^i$ then $\Gamma_{u^i}^{i-1}<\tau\leq\Gamma_{u^i}^i=\Gamma_{u^i}^{i_1}+1$, hence $W(\vec\pi_{u^i}^{\Gamma_{u^i}^{i-1}+1})=W_i$ for all $i$, which concludes the proof.

\end{proof}

Whittle's index being non-decreasing in $\tau$ implies that, the longer a user has not been selected for channel sensing the more attractive it becomes to select him/her. The exploration vs. exploitation trade-off is therefore captured by this property of the index. 

We illustrate how Whittle's index is obtained in Figure~\ref{fig:Whittlesindex} for a particular example with $K=3$. Observe that $g^{OPT}(W)=\max_{\vec\Gamma}\{g^{\vec\Gamma}(W)\}$ is the upper envelope of  affine increasing functions in $W$. Whittle's index is therefore computed by the intersecting points of the affine functions that determine the envelope. By the indexability property we have that, for all $W<W_0$ always being active is prescribed, and for all $W>W_{I}$ always being passive is prescribed (with $I$ the iteration at which the algorithm in Theorem~3 has stopped).

\begin{remark}\label{remark:omega}
We highlight that, although in the present work we have focused on the approximation~\eqref{equation:app} (see Section~\ref{sec:relaxation}), the explicit expression of Whittle's index, as computed in Corollary~\ref{cor:whittle}, could have been obtained  using any of these following approximations. Assume $q^0(\cdot,\cdot)$ to be as in the original model and let \begin{equation}\label{eq:approxi}
q^1(\vec\pi^\tau_i,\vec\pi^{\tau'}_j)=
\begin{cases}
p^{(m)}_{ij} &\hbox{if } \tau'=1,\hbox{ and}, m \hbox{ independent of } \tau,\\
0, & \hbox{otherwise}.
\end{cases}  
\end{equation}
The expression of $\omega_j$ for all $j$ in Corollary~\ref{cor:whittle},  is the solution of the global balance equation for the Markov Chain of the approximation in Equation~\eqref{equation:app}. We note that any approximation in Equation~\eqref{eq:approxi}, shares the same solution as that of approximation~\eqref{equation:app}. Hence, Whittle's index is the same.

This latter statement does not hold for the original model though, since the transition probabilities from one channel to another are policy dependent.
\end{remark}
\begin{figure}\centering
\includegraphics[scale=0.5]{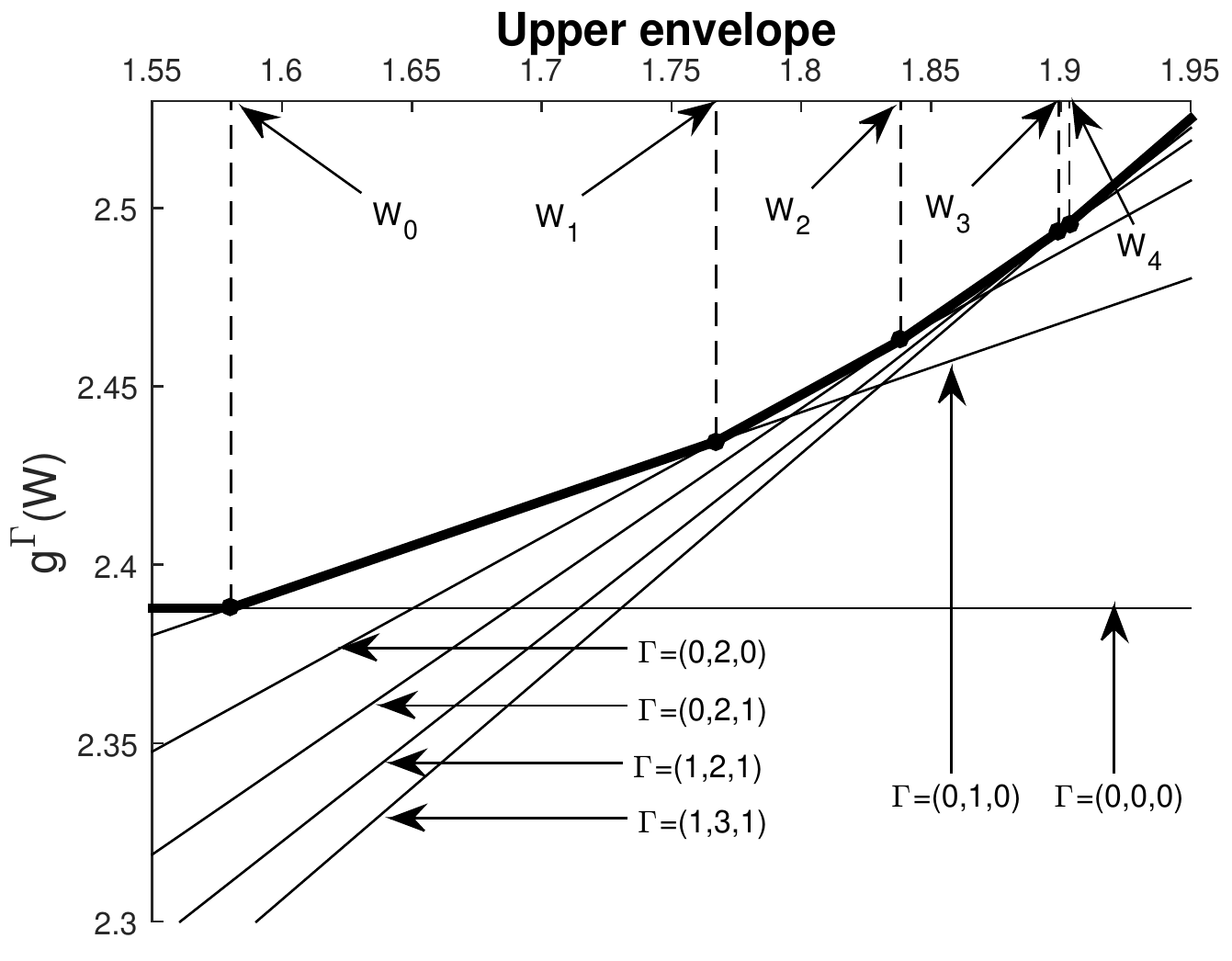}
\caption{Upper envelope, i.e., $\max_{\vec\Gamma}\{g^{\vec\Gamma}(W)\}$, for a particular example with $K=3$, doubly stochastic transition matrix, and $R(\vec\pi_j^\tau,0)=\frac{\rho_j}{3}\sum_{k=1}^3\log_2(1+SNR)$, with $\rho_j=\max_{r}\{p_{jr}^{(\tau)}\}$. Note $W(\vec\pi_2^1)=W_0, W(\vec\pi_2^2)=W_1, W(\vec\pi_3^1)=W_2, W(\vec\pi_1^1)=W_3$ and $W(\vec\pi_2^3)=W_4$. The rest of values can be obtained computing further intersection points in the upper envelope. }\label{fig:Whittlesindex}
\end{figure}
\subsection{Whittle's index policy}\label{sec:WIP}
In this section we explain how the Whittle index can be used in order to define a heuristic for the original unrelaxed problem, as in Equation~\eqref{eq:obj_func_average}. 
\begin{definition}
Assume the state of user $n$ at time $t$ to be $\vec\pi_{j_n}^{\tau_n}$. The Whittle index policy prescribes to allocate a pilot to the $M$ users with the highest $W_n(\vec\pi_{j_n}^{\tau_n})$.
\end{definition}
Whittle's index policy ($WIP$) is an optimal solution for the relaxed POMDP. It has been proven to be optimal in several asymptotic regimes. For instance, it was proven to be optimal in the many-users setting in Verloop \cite{Ver13}, Ouyang et al.~\cite{OuyangEryilShroff12}, and Weber et al. \cite{WW90}. Moreover, the  asymptotic optimality of Whitte's index in this  regime was conjectured by Whittle in the paper in which Whittle's index was first proposed \cite{Whi88}.
\section{Error estimation}\label{sec:error}
In this section we estimate the error introduced by the approximation that has been considered throughout the paper. Recall that this approximation has been adopted in order to obtain structural results of the optimal policy. The latter is due to the optimality equation of the original problem being extremely difficult to solve. In order to characterize the absolute error explicitly we first define $V_\beta^{max}$ and $V_\beta^{min}$. Let $V_\beta^{max}(\cdot)$ be the value function that satisfies the following Bellman equation
\begin{align}\label{eq:bellman_approximations}
V_\beta^{max}(\vec\pi_j^\tau)=\max\{&R(\vec\pi_j^\tau,0)+W+\beta V_\beta^{max}(\vec\pi_j^{\tau+1});\nonumber\\
&R(\vec\pi_j^\tau,1) +\beta\max_i\{V_{\beta}^{max}(\vec\pi_i^1)\}\},
\end{align}
for all $\tau$. And let $V_\beta^{min}(\cdot)$ be the value function that satisfies the following Bellman equation
\begin{align}\label{eq:bellman_approximations_min}
V_\beta^{min}(\vec\pi_j^\tau)=\max\{&R(\vec\pi_j^\tau,0)+W+\beta V_\beta^{min}(\vec\pi_j^{\tau+1});\nonumber\\
&R(\vec\pi_j^\tau,1) +\beta\min_i\{V_{\beta}^{min}(\vec\pi_i^1)\}\},
\end{align}
for all $\tau$. Let $V_\beta$ be the value function of the original discounted reward single-arm POMDP. Then the following lemma holds. The proof can be found in Appendix~\ref{app:proofvaluefunctions}.

\begin{lemma}\label{lemma:}\label{lemma:valuefunctions} Let $V_\beta^{max}(\cdot)$ be defined as in Equation~\eqref{eq:bellman_approximations} and $V_\beta^{min}(\cdot)$ as defined in Equation~\eqref{eq:bellman_approximations_min}. Then 
\begin{align*}
V_\beta^{max}(\cdot)\geq V_{\beta}(\cdot), \text{ and }
V_\beta^{min}(\cdot)\leq V_{\beta}(\cdot).
\end{align*}
\end{lemma}
We define $g^{max}(W)=\lim_{\beta\to1}(1-\beta)V_\beta^{max}(\cdot)$ and $g^{min}(W)=\lim_{\beta\to1}(1-\beta)V_\beta^{min}(\cdot)$. Then the following proposition holds. The proof can be found in Appendix~\ref{prop:proof_performancebounds}.
\begin{proposition}\label{prop:performance_bounds} Let $g(W)$ be the optimal average reward for the  relaxed POMDP and $g^{app}(W)$ be the optimal average reward for the approximation in Equation~\eqref{equation:app}.  Then the relative error of the approximation is bounded as follows
\begin{align*}
\bigg|1-\frac{g^{app}(W)}{g(W)}\bigg|\leq D(W),
\end{align*}
where 
$$
D(W):=\max\left\{1-\frac{g^{app}(W)}{g^{max}(W)},\frac{g^{app}(W)}{g^{min}(W)}-1\right\}.
$$ 
The expression of $D(W)$ can be found in Appendix~\ref{appendix:explicit_D}.
\end{proposition}
Proposition~\ref{prop:performance_bounds} provides an error measure to estimate how good the approximation that has been considered is. Through extensive numerical experiments it has been observed that the error incurred by the approximation is extremely small, see Section~\ref{sec:approxnumerics} for some case studies.




\begin{remark} We note that the approximation introduced in Section~\ref{sec:relaxation} differs from the original model only when the active action is considered. In the case in which the transition probabilities are the steady-state probabilities the error provided by the approximation is zero. The latter suggests that the closer the transition probabilities are from the steady-state probabilities the smaller the error will be.
\end{remark}
\section{Asymptotic optimality in the many users setting}\label{asymptotic}

In this section we prove that the Whittle index policy  is asymptotically optimal in the \emph{many users} setting. We define the many users setting as follows. We assume a downlink scheduling problem with a population of $N$ users and we aim at obtaining a policy $\phi\in\mathcal{U}$ such that
\begin{align}\label{eq:objective_manyusers}
R^{N,\phi}:=\liminf_{T\to\infty}\frac{1}{T}\mathbb{E}\left(\sum_{t=1}^T\sum_{n=1}^NR_n(X_n(t)\vec b_n^\phi(t),a_n^\phi(t))\right),
\end{align}
is maximized subject to
\begin{align}\label{constraint:manyusers}
\sum_{n=1}^Na_n^\phi(t)\leq \lambda N,
\end{align}
for each time slot, where $\mathcal{U}$ is the set of policies that satisfy constraint~\eqref{constraint:manyusers} and $0\leq\lambda\leq 1$. That is, the greater the population of users in the system is, the greater the available number of pilots is (i.e., greater number of users can be selected for channel sensing).  We now introduce the relaxed version of problem \eqref{eq:objective_manyusers}-\eqref{constraint:manyusers}, namely, find $\phi\in\mathcal{U}^{REL}$ that maximizes 
\begin{align}\label{eq:objective_manyusers_relaxed}
\liminf_{T\to\infty}\frac{1}{T}\mathbb{E}\left(\sum_{t=1}^T\sum_{n=1}^NR_n(X_n(t)\vec b_n^\phi(t),a_n^\phi(t))\right),
\end{align}
 subject to
\begin{align}\label{constraint:manyusers_relaxed}
\liminf_{T\to\infty}\frac{1}{T}\mathbb{E}\left(\sum_{t=1}^T\sum_{n=1}^Na_n^\phi(t)\right)\leq \lambda N,
\end{align}
where $\mathcal{U}^{REL}$ is the set of policies that satisfy constraint~\eqref{constraint:manyusers_relaxed}. In particular we have $\mathcal{U}\subset\mathcal{U}^{REL}$.

Next we characterize the optimal relaxed policy.
 
\noindent
 {\it Optimal relaxed policy (REL):} There exist $W^*$ and $\rho\in(0,1]$ such that, the policy that prescribes to allocate a pilot to all users $n$ having $W_n(\vec\pi_{n,j}^\tau)>W^*$, and to all users $n$ having $W_n(\vec\pi_{n,j}^\tau)=W^*$ with probability $\rho$ is optimal
 for problem \eqref{eq:objective_manyusers_relaxed}-\eqref{constraint:manyusers_relaxed}. Moreover, constraint \eqref{constraint:manyusers_relaxed} is satisfied with equality. We refer to this policy by $REL$.

Recall that the policy $WIP$, is such that the $\lambda N$ users with the largest Whittle's index are allocated with a pilot. We therefore have
 \begin{align}\label{asympt_inequality}
R^{N,WIP}\leq R^{N,OPT}\leq R^{N,REL},
 \end{align}
 with $R^{N,OPT}:=\max_{\phi\in\mathcal{U}}R^{N,\phi}$. 

In this section, we aim at establishing that as $N$ tends to infinity the optimal solution of the relaxed problem \eqref{eq:objective_manyusers_relaxed}-\eqref{constraint:manyusers_relaxed}, i.e., $R^{N,REL}$, is asymptotically equivalent to the optimal solution of problem \eqref{eq:objective_manyusers}-\eqref{constraint:manyusers}, i.e., $R^{N,OPT}$. We further prove that, under some assumption, $R^{N,WIP}$ as $N\to\infty$ converges to the optimal solution of the relaxed problem, and is hence an asymptotically optimal solution for problem \eqref{eq:objective_manyusers}-\eqref{constraint:manyusers}. 
 
The asymptotic optimality result obtained below, which considers the pilot allocation problem with $K$-state Markov Chain channels, is a generalization of the result obtained in Ouyang et al.~\cite{OuyangEryilShroff12} for the Gilbert-Elliot model (two-state Markov Chain model). In this paper we follow the same line of arguments that has been used there. We prove the intermediate results (required to show Propositions~1 and~2 in Ouyang et al.~\cite{OuyangEryilShroff12}) that fail to easily extend to our scenario, and we refer to \cite{OuyangEryilShroff12} for the proofs of the lemmas that extend to our case without much effort.

Note that, due to Inequality~\eqref{asympt_inequality}, to prove asymptotic optimality of $WIP$ it suffices to show that as $N$ tends to~$\infty$ $R^{N,REL}$ and $R^{N,WIP}$ are asymptotically equivalent. We will therefore focus on proving the latter. 

The idea for the proof is as follows. Firstly, we define the state of the system to be the proportion of users in all possible channel belief states. We define a \emph{fluid approximation} of this system under $WIP$, by characterizing the evolution of it through a set of linear differential equations. We prove the fluid system to have a single fixed point solution (the equilibrium distribution under $REL$).
Secondly, we establish a \emph{local optimality} result, which states that as $N\to\infty$, $R^{N,WIP}$ and $R^{N,REL}$ are asymptotically equivalent if the initial state (i.e., initial configuration of users) is in the neighborhood of the equilibrium distribution under $REL$.  Finally, we prove \emph{global convergence}, by showing that, under an assumption that can be numerically verified, as $N\to\infty$, $R^{N,WIP}$ and $R^{N,REL}$ are  asymptotically equivalent for any possible initial state.

\subsection{Fluid approximation under $WIP$}

In this section we characterize the fluid system under Whittle's index policy.  For sake of clarity, two technical assumptions are made next.
\begin{itemize}
\item  We assume that there are two different classes of users. Moreover, we denote the channel transition matrix of users that belong to class~1 by $P^1=(p_{ij}^1)_{i,j\in\{1,\ldots,K\}}$ and that of the users that belong to class~2 by $P^2=(p_{ij}^2)_{i,j\in\{1,\ldots,K\}}$. 

Due to the latter assumption, belief state vectors will be denoted as $\vec\pi^{\tau,c}_j$ and Wittle's index in $\vec\pi^{\tau,c}_j$ as $W(\vec\pi^{\tau,c}_j)$ for class-$c$ users, with $c\in\{1,2\}$. Namely, we replace the user dependency (e.g., $W_n(\cdot)$ or $\vec\pi^{\tau}_{n,j}$) by class dependency in the notation.
\item We assume a  truncated belief state space, i.e., we define the state space as follows:
\begin{align*}
\overline\Pi_c=&\{\vec\pi_{j}^{\tau,c}:\vec\pi_{j}^{\tau,c}=\vec e_j(P^{c})^{\tau},0<\tau\leq \overline\tau, 
 j\in\{1,\ldots,K\}\} \cup \{\vec\pi^{s,c}\}
\end{align*}
for all $c\in\{1,2\}$. If the truncation parameter $\overline\tau$ is large enough, then $\vec\pi_{j}^{\overline\tau,c}$, the belief vector for a class-$c$ user, is very close to the steady-state belief vector $\vec\pi^{s,c}=(p_1^{s,c},\ldots,p_K^{s,c})$. Motivated by the latter, we assume that in the truncated system, the passive transition probability from belief state $\vec\pi_{j}^{\overline\tau,c}$ to $\vec\pi^{s,c}$ for a class-$c$ user equals 1, i.e., $q^{0,\overline\tau}(\vec\pi_{j}^{\overline\tau,c},\vec\pi^{s,c})=1$ for all $j$.
\end{itemize}
 
 Now we define the state space over which the optimality result will be established. Let us define $\mathbf{Y}^N$ the proportion of users in each belief value, that is,
 $\mathbf{Y}^N=[\mathbf{Y}^{1,N},\mathbf{Y}^{2,N}]$, where
\begin{align*}
\mathbf{Y}^{c,N}=[Y_{1,1}^{c,N},\ldots,Y_{1,\overline\tau}^{c,N},\ldots,Y_{K,1}^{c,N},\ldots,Y_{K,\overline\tau}^{c,N},Y_{s}^{c,N}],
\end{align*}
for $c\in\{1,2\}$. To this extent, $Y_{i,j}^{c,N}$ represents the proportion of class-$c$ users in belief state $\vec\pi_{i}^{j,c}$, and $Y_s^{c,N}$ represents the proportion of class-$c$ users in the steady-state belief vector, i.e., $\vec\pi^{s,c}$. Let $\delta_c$ denote the fraction of users that belong to class~$c$, then the state space of this system is defined as 
$$
\mathcal{Y}=\{\mathbf{Y}^N: Y^{c,N}_s+\sum_{i=1}^K\sum_{j=1}^{\overline\tau} Y_{i,j}^{c,N}=\delta_c, c\in\{1,2\}\}.
$$ 
To avoid analyzing  well understood scenarios we will make the following assumption. 
\begin{assumption}\label{assumptionWhittle}
We assume that $W(\vec\pi^{s,1}), W(\vec\pi^{s,2}) \geq W^*$ for all class-$1$ users and all  class-$2$ users.
\end{assumption}  
Due to Whittle's index being non-decreasing, we note that if $W(\vec\pi^{s,1})\leq W$ and $W(\vec\pi^{s,2})\leq W$, then $REL$ reduces to not allocating any pilot to any user, that is $\lambda N=0$. Since $WIP$ prescribes to allocate pilots to $\lambda N$ users with the greatest Whittle's index, and $\lambda N=0$, $WIP$ reduces to $REL$ and is hence optimal. Moreover, if $W(\vec\pi^{s,c})\leq W \leq W(\vec\pi^{s,c'})$ for $c\neq c'\in\{1,2\}$, then the system reduces to a single class problem, since one of the classes will never be allocated with a pilot. We therefore focus on the case in which $W(\vec\pi^{s,1}), W(\vec\pi^{s,2}) \geq W^*$ (Assumption~\ref{assumptionWhittle}).

We are now in position to define the fluid system. We adopt the following notation. Let $b_i$ represent the belief value that corresponds to the $i^{th}$ entry in $\mathbf{Y}^N(t)$, and $W_i$ refer to the Whittle's index in belief state $b_i$, e.g., $b_1$ corresponds to $\vec\pi_{1}^{1,1}$ and $W_1$ to $W(\vec\pi_{1}^{1,1})$. Let us denote by $q_{ij}(\mathbf{y})$ the probability that the belief value of the channel  jumps from belief value $b_i$ to $b_j$ given that the systems state is $\mathbf{y}\in\mathcal{Y}$. Then
\begin{align}\label{eq:probs}
q_{ij}(\mathbf{y})=g_i(\mathbf{y})q_{ij}^1+(1-g_i(\mathbf{y}))q_{ij}^0,
\end{align}
where $g_i(\mathbf{y})$ corresponds to the fraction of users in belief value $b_i$ that are activated by $WIP$ and $q_{ij}^a$ for $a=0,1$, is the probability that the belief value transits from $b_i$ to $b_j$ under action $a$, i.e., $q^a(b_i,b_j)$. The explicit expressions of $g_i(\mathbf{y})$ and $q_{ij}^a$ for $a\in\{0,1\}$ are given in Table~I.  In the case in which $y_i\neq0$, only a fraction of the users in belief value $b_i$ will be activated, exactly the amount that is required for constraint~\eqref{constraint:manyusers} to be binding.
\begin{table*}
\caption{Transition probabilities from belief value $b_i$ to $b_j$}
\centering
\begin{minipage}{0.85\textwidth}
\begin{align*}
g_i(\mathbf{y})&=
\begin{cases}
\min\left\{\left[\frac{\lambda-\sum_{j:W_j>W_i}y_j}{y_i}\right]^+,1\right\}, & \hbox{if } y_i\neq 0,\\
1, &\hbox{if } y_i=0, \hbox{ and } \lambda>\sum_{j:W_j>W_i}y_j,\\
0, &\hbox{if } y_i=0, \hbox{ and } \lambda\leq\sum_{j:W_j>W_i}y_j,
\end{cases}\\
\medskip
q_{ij}^1&= 
\begin{cases}
p_{r}^{s,1}, &\hbox{if } j=(r-1)\overline\tau+1, \hbox{ and } (r-1)\overline\tau+1\leq i\leq r\overline\tau , r=1,\ldots,K, \hbox{ or } i=K\overline\tau+1\\
p_{r}^{s,2}, &\hbox{if } j=(K+r-1)\overline\tau+2, \hbox{ and } (K+r-1)\overline\tau+2\leq i\leq (K+r)\overline\tau+1 , r=1,\ldots,K, \hbox{ or } i=2K\overline\tau+2,\\
0, &\hbox{otherwise}, 
\end{cases}\\
\medskip
q_{ij}^0&= 
\begin{cases}
1, &\hbox{if } j=i+1, \hbox{ and } i\neq \overline\tau, 2\overline\tau, \ldots, (K-1)\overline\tau, K\overline\tau+1,  (K+1)\overline\tau+1,\ldots, (2K-1)\overline\tau+1,2K\overline\tau+2,\\
1, & j=K\overline\tau+1, \hbox{ and } i=\overline\tau,\ldots,(K-1)\overline\tau,K\overline\tau+1,\\
1, &\hbox{if } j=2K\overline\tau+2, \hbox{ and } i=(K+1)\overline\tau+1,\ldots,(2K-1)\overline\tau+1,2K\overline\tau+2,\\
0, &\hbox{otherwise}. 
\end{cases}
\end{align*}
\hrule
\end{minipage}
\end{table*}

We next define the expected drift of $\mathbf{Y}^N(t)$ to be
\begin{align*}
D\mathbf{Y}^N(t):=\mathbb{E}(\mathbf{Y}^N(t+1)-\mathbf{Y}^N(t)|\mathbf{Y}^N(t)),
\end{align*}
hence
\begin{align}\label{theexprecteddrift}
D\mathbf{Y}^N(t)\bigg|_{\mathbf{Y}^N(t)=\mathbf{y}}=\sum_{i=}\sum_{j=}q_{ij}(\mathbf{y})y_i\cdot\vec e_{ij}=Q(\mathbf{y})\mathbf{y},
\end{align}
where $\vec e_{ij}=(0,\ldots,0,\overbrace{-1}^\text{$i^{th}$},0,\ldots,0, \overbrace{1}^\text{$j^{th}$},0,\ldots,0)$, that is, it is the $2(K\overline \tau+1)$ dimensional vector that has $-1$ in its $i^{th}$ entry and $1$ in its $j^{th}$ entry, also we define $\vec e_{ii}=(0,\ldots,0)$. Moreover,
\begin{align*}
Q_{i,j}(\mathbf{y})=
\begin{cases}
-\sum_{j\neq i}q_{ij}(\mathbf{y}), & \hbox{if } i=j,\\
q_{ji}(\mathbf{y}), &\hbox{if } i\neq j.
\end{cases}
\end{align*}
The latter equation allows the system to be interpreted as a \emph{fluid system}, only taking the expected direction of the system into account, note that~\eqref{theexprecteddrift} is also defined for $\mathbf{y}\notin\mathcal{Y}$, and $Q(\mathbf{y}(t))\mathbf{y}(t)$ does not depend on $N$. Therefore we represent the expected change of a \emph{fluid system in discrete time} as follows
\begin{align}\label{fluid_sys}
\mathbf{y}(t+1)-\mathbf{y}(t)=Q(\mathbf{y}(t))\mathbf{y}(t).
\end{align}

Let $\overline{Y}_{W^*}=\{\mathbf{y}\in\mathcal{Y}:\sum_{j:W_j>W^*}y_j<\lambda, \sum_{j:W_j\geq W^*}y_j\geq\lambda\}$, that is, the set of states in which all users with Whittle's index higher than $W^*$ are activated, users with Whittle's index smaller than $W^*$ are passive, and users for which Whittle's index equals $W^*$ are activated with randomization parameter $\rho$. In the next lemma we show that the fluid system in Equation~\eqref{fluid_sys} under $WIP$ is linear in $\mathbf{y}(t)\in\overline{Y}_{W^*}$. The proof can be found in Appendix~\ref{proof_linear_fluid}.
\begin{lemma} \label{eq:linear_fluid_sys} 
For all $\mathbf{y}(t)\in\overline{Y}_{W^*}$, the fluid system~\eqref{fluid_sys} is linear. That is,
there exist $\overline Q$ and $\overline d$ such that
\begin{align}\label{eq_fluid_lin}
\mathbf{y}(t+1)-\mathbf{y}(t)=\overline Q\cdot \mathbf{y}(t)+\overline d, 
\end{align}
for all $\mathbf{y}(t)\in\overline Y_{W^*}$.
\end{lemma}  

In Lemma~\ref{lema:hola}, we characterize the unique fix point solution of the linear fluid system of Lemma~\ref{eq:linear_fluid_sys}, the proof can be found in Appendix~\ref{append_fixpoint}. To do so we first introduce the following definition.
\begin{definition}\label{def..} Let $\theta_{\delta,\lambda}:=\mathbb{E}[\mathbf{Y}^{N,\infty}]$, where $\mathbf{Y}^{N,\infty}$ is such that, under the $REL$ policy, the system state $\mathbf{Y}^N(t)$ converges in distribution to $\mathbf{Y}^{N,\infty}$.
\end{definition}
\begin{lemma}\label{lema:hola}  The linear fluid system given by Equation~\eqref{eq_fluid_lin} equals 0, i.e., $\overline Q\cdot \mathbf{y}(t)+\overline d=0$, if and only if $\mathbf{y}(t)=\theta_{\delta,\lambda}$, where  $\theta_{\delta,\lambda}$ is as defined in Definition~\ref{def..}. Furthermore, $\theta_{\delta,\lambda}$ is independent of $N$.
\end{lemma}
Having established the linearity of the fluid system and the uniqueness of its fixed point, the local asymptotic optimality result can be obtained. We do so in the next section.

\subsection{Local asymptotic optimality}

The intuition behind the local asymptotic optimality result is that, if the average reward accrued by the $WIP$ policy falls in the neighborhood of $\mathbf{\theta}_{\delta,\lambda}$, then this reward is close to that accrued under the $REL$ policy. We define the neighborhood of $\mathbf{\theta}_{\delta,\lambda}$ as follows
$$
\mathcal{N}_\epsilon(\mathbf{\theta}_{\delta,\lambda})=\{\mathbf{y}\in\mathcal{Y}: \|\mathbf{y}-\mathbf{\theta}_{\delta,\lambda}\|\leq \epsilon\},
$$
and we denote by $R^{N,WIP}_T(\mathbf{y})$ the throughput obtained under $WIP$ policy in the time interval $[0,T]$ given that the initial state of the system is $\mathbf{y}$, i.e.,
\begin{align*}
&R^{N,WIP}_T(\mathbf{y})    =\frac{1}{T}\mathbb{E}\left(\sum_{t=1}^T\sum_{n=1}^NR(X_n(t),\vec b_n^{WIP}(t),a_n^{WIP}(t))\bigg|\mathbf{Y}^N(0)=\mathbf{y}\right).
\end{align*}
Moreover, it can be easily proven that the reward obtained by $REL$, i.e., $R^{N,REL}$, is independent of $N$. The latter can be obtained by exploiting the idea that users under the $REL$ policy are activated independently from each other, see Lemma~3 in~\cite{OuyangEryilShroff16}. Therefore, $R^{REL}:=R^{N,REL}$, is determined by a user configuration $\delta$ and a given $\lambda$ and not the population size $N$.

The local convergence of the reward under $WIP$ to $R^{REL}$ is proven in the next proposition.
\begin{proposition}\label{prop:local_optimality_WIP}
For any given $(\delta,\lambda)$, there exist $\epsilon$ and $\mathcal{N}_{\epsilon}(\mathbf{\theta}_{\delta,\lambda})$ such that
\begin{align*}
\lim_{T\to\infty}\lim_{r\to\infty}\frac{R^{N_r,WIP}_T(\mathbf{y})}{N_r}=R^{REL},
\end{align*}
if $\mathbf{y}\in\mathcal{N}_{\epsilon}(\mathbf{\theta}_{\delta,\lambda})$, for all $(N_r)_{r}$ increasing sequence of positive integers such that $N_r,\delta_cN_r\in\mathbb{Z}$.
\end{proposition}
The proof of the proposition can be found in Appendix~\ref{proof_local}.

\subsection{Global asymptotic optimality}
In this section we establish the global asymptotic optimality of $WIP$ in the many users setting. 
 In order to do so, we are first going to prove that the system state $\mathbf{Y}^N(t)$ has a particular structure, see lemma below.
\begin{lemma}\label{lemma:recurrentclass}
For fixed values of $\delta$ and $\lambda$, and letting $N$ be large enough, we have that
\begin{enumerate}
\item $\mathbf{Y}^N(t)$ with $t\geq0$ is an aperiodic Markov chain with a single recurrent class.
\item For each $\epsilon>0$ there exists a recurrent state within $\mathcal{N}_\epsilon({\mathbf{\theta}_{\delta,\lambda}})$.
\end{enumerate}
\end{lemma}
\begin{proof}
The proof can be found in Appendix~\ref{proof_lemma_recurrent}, and follows the arguments used in~\cite[Lemma 5]{OuyangEryilShroff16}.
\end{proof}
Having proven that there exists a recurrent state in any $\epsilon$ neighborhood of $\mathbf{\theta}_{\delta,\lambda}$ allows to establish the global optimality result. However, one needs to ensure that the time the process $\mathbf{Y}^N(t)$ under $WIP$ policy needs to enter the neighborhood $\mathcal{N}_{\varepsilon}(\mathbf{\theta}_{\delta,\alpha})$ does not grow as $N$ increases. To avoid this from happening one can verify certain conditions to be satisfied, such as that given in~\cite[Assumption in Th. 2]{WW90} or that given in~\cite[Assumption $\Psi$]{OuyangEryilShroff16}. This latter states that the expected time of reaching any $\epsilon$ neighborhood of $\mathbf{\theta}_{\delta,\lambda}$ is bounded by an $\epsilon$ dependent constant. 
We can now state the global optimality result.
\begin{proposition} Let Assumption $\Psi$ in~\cite{OuyangEryilShroff16} be satisfied. Then for any initial state $\mathcal{Y}^N(0)=\mathbf{y}$ the following holds
\begin{align*}
\lim_{r\to\infty}\frac{R^{N_r,WIP}(\mathbf{y})}{N_r}=R^{REL},
\end{align*}
with $R^{N,WIP}(\mathbf{y})=\lim_{T\to\infty}R^{N,WIP}_T(\mathbf{y})$.
\end{proposition}
\begin{proof} The proof follows from proof of Proposition~2 in~\cite{OuyangEryilShroff16}, and relies in the proof of our Lemma~\ref{lemma:recurrentclass}.
\end{proof}

\section{Numerical analysis}\label{sec:numerical}

We provide in this section some numerical results to assess the performance of the Whittle's index policy. Firstly, in Section~\ref{sec:approxnumerics} we study various scenarios to evaluate the accuracy of the approximation introduced in Section~\ref{sec:relaxation}. In Section~\ref{sec:num1} we compare the structure of $WIP$ w.r.t. the optimal solution. Finally, in Section~\ref{sec:num2} we perform extensive numerical experiments to compute the relative suboptimality gap of $WIP$ w.r.t. the optimal solution.  All the results have been obtained through the value iteration algorithm~\cite[Chap. 8.5.1]{Puterman2005}. 

\begin{figure*}
\begin{minipage}[b]{0.48\linewidth}\centering
\includegraphics[scale=0.45]{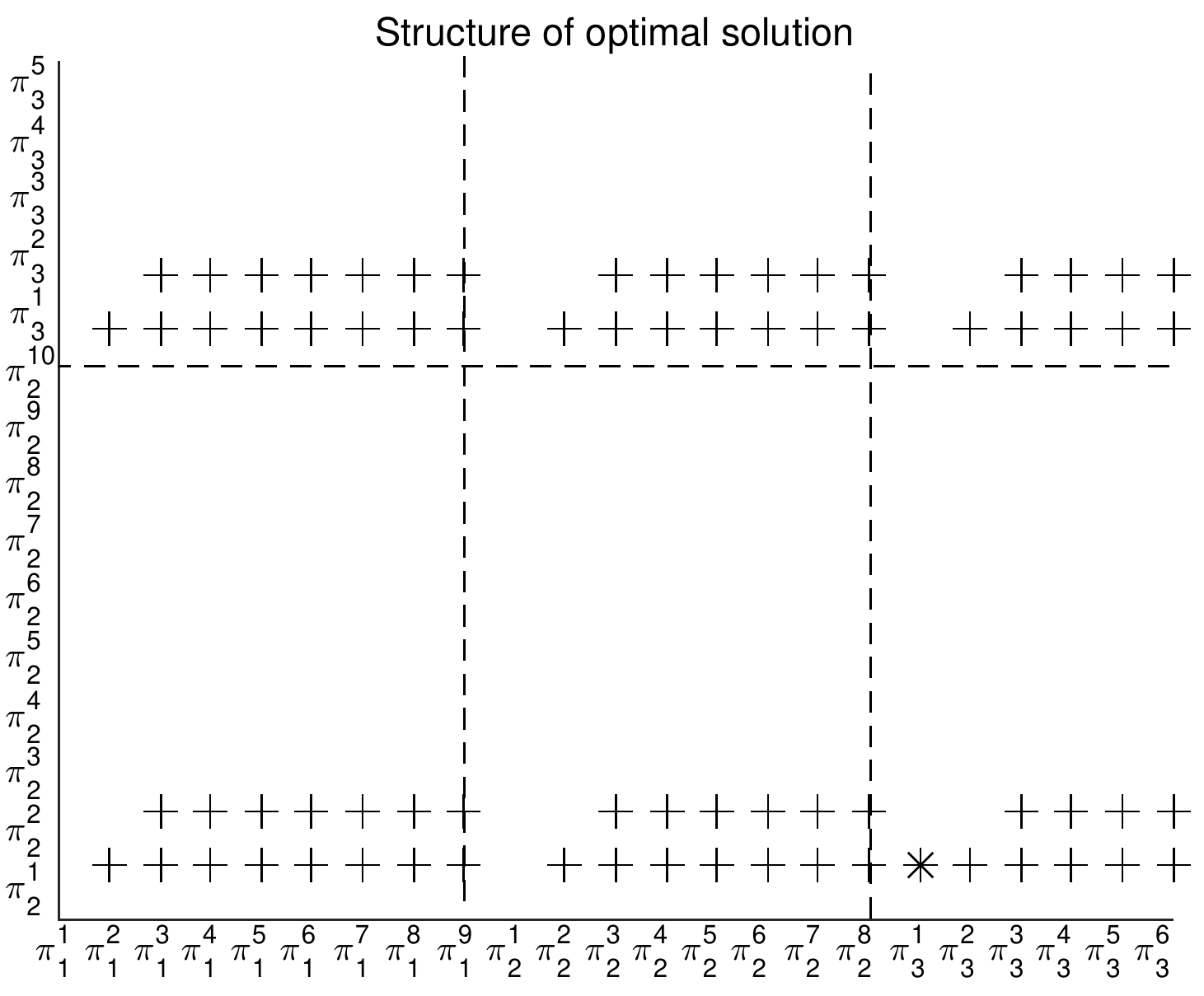}
\end{minipage}
\begin{minipage}[b]{0.48\linewidth}\centering
\includegraphics[scale=0.45]{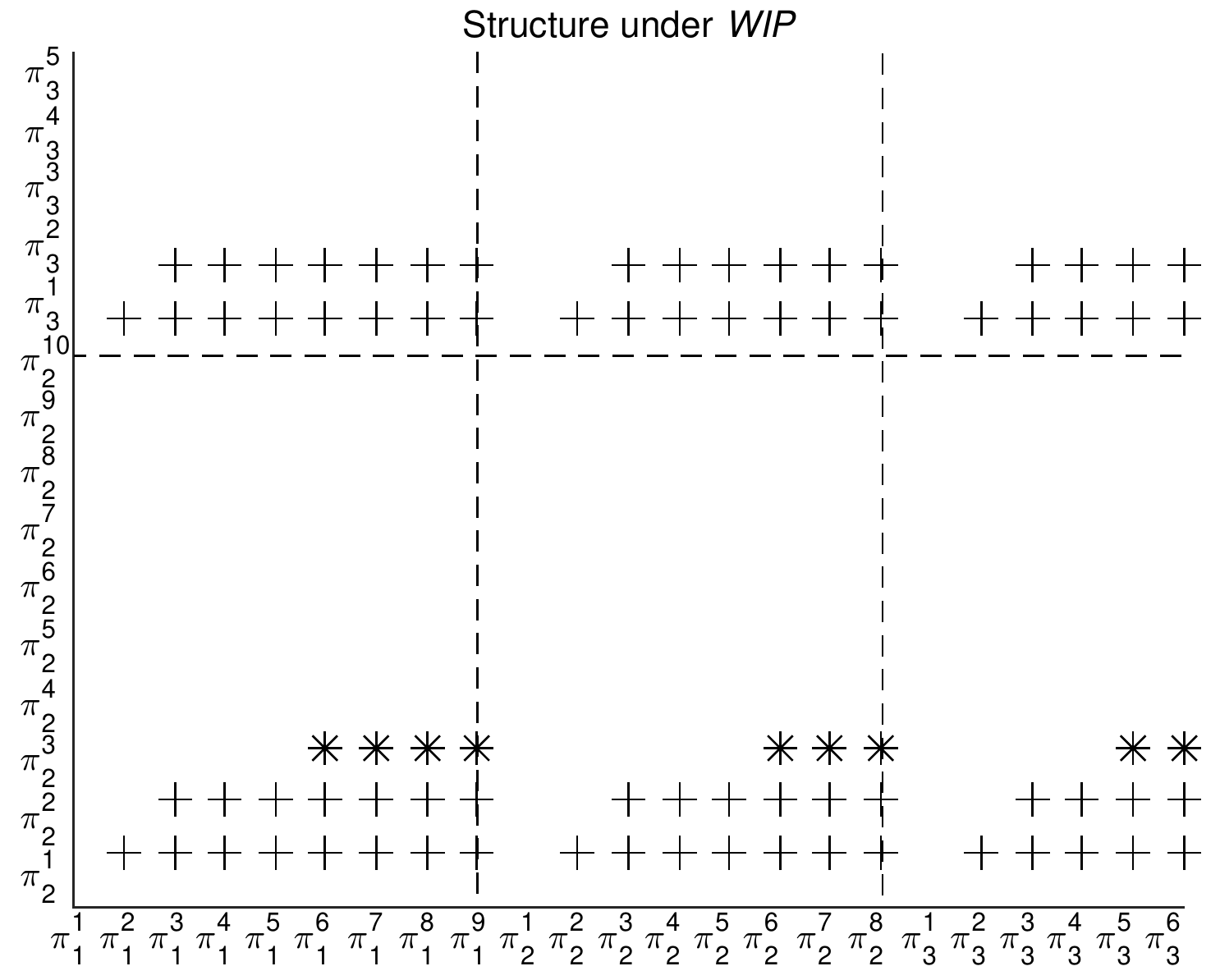}
\end{minipage}
\caption{Left: Structure of optimal solution. Right: Structure of Whittle's index policy. In the area with ``+'' or ``*'' user~1 is allocated with a pilot, and in the blank area user~2 receives the pilot. The sign ``*''  illustrates the states in which the optimal structure and the structure under $WIP$ do not match. The state vector $\pi_i^j$ in the horizontal axis refers to the belief state for user~1, and $\pi_i^j$ in the vertical axis refers to user~2. All states $\pi_1^j$ for user~2 are omitted since both policies prescribe to allocate the pilot to user~2.}
\label{fig:sce0}
\end{figure*}
\subsection{Accuracy of the approximation}\label{sec:approxnumerics}

In Section~\ref{sec:error} 
an upper bound on the error incurred by the approximation has been characterized, i.e., $D(W)$, for the per-user average reward. In this section we illustrate that this approximation shows an extremely small relative error in the $N$-dimensional problem, that is, problem~\eqref{eq:obj_func_average}. In order to perform this analysis we compute the optimal solution for the approximation and the optimal solution for the original model and we compare the corresponding average rewards.

\noindent
{\bf Example:} Let us assume a system with a BS and four users. We assume users to be in three possible channel states $h_{n1}, h_{n2}, h_{n3}$. 
Let the transition matrices to be doubly stochastic and to be different for all four users. The steady-state belief state for all four users is $(1/3,1/3,1/3)$. Therefore, the immediate average reward for user $i$ if a pilot has been allocated to it is assumed to be $R^1_i=\frac{1}{3}\sum_{k=1}^3\log_2(1+SNR),$
$i\in\{1,\ldots,K\}$. If user $i$ has not been selected the average immediate reward is considered to be $R_i(\vec\pi_j^\tau,0)=\rho_i\frac{1}{3}\sum_{k=1}^3\log_2(1+SNR),$
where $\rho_i=\max_r\{p_{jr}^{(\tau)}\}$, that is, the highest probability channel state for user $i$, when its belief state  is  $\vec\pi_j^\tau$, and $\hat h_i=h_{i\sigma}$ where $\sigma=\argmax_r\{p_{jr}^{(\tau)}\}$. We first assume that a single pilot is available to the system, and later on we assume that three pilots are available. The relative  error of the approximation w.r.t. the original problem can be found in Table~\ref{table_example} for three different examples (three different channel vectors and probability transition matrices). We can observe in Table~\ref{table_example} that the error in all the examples is extremely small. 
\begin{table}[!t]
\renewcommand{\arraystretch}{1.3}
\caption{Relative  (\%) suboptimality gap }
\label{table_example}
\centering
\begin{tabular}{c|c|c}
\hline
&  App.  1 pilot   &  App.  3 pilots   \\
\hline
Rel. err. ex. 1& 0.0798 & 0.0527 \\
\hline
Rel. err. ex. 2& 0.0149 & 0.0393 \\
\hline
Rel. err. ex. 3 & 0.0217 & 0.0403 \\
\hline
\end{tabular}
\end{table}

\subsection{Structure of Whittle's index}\label{sec:num1}

We have shown in Corollary~\ref{cor:whittle} that Whittle's index is non-decreasing in $\tau$. Recall that this is due to Assumption A1. The latter implies that  if serving user $1$ is prescribed by $WIP$ in state $\vec\pi_j^\tau$ then also in $\vec\pi_j^{\tau+1}$ (independent of the number of users in the system). This structure is illustrated in the next example.

\noindent
{\bf Example:} We consider a system with two users, one pilot and three channel states, where the transition probability matrices for both users are 
\begin{align*}
P_1=\begin{bmatrix}
0.3& 0.4&0.3\\
0.2&0.2&0.6\\
0.5&0.4&0.1
\end{bmatrix}, P_2=\begin{bmatrix}
0.35& 0.35&0.3\\
0.3&0.15&0.55\\
0.35&0.5&0.15
\end{bmatrix}, 
\end{align*}
 and the channel vectors are $\mathbf{h}^1=(0.512+0.9671i,-1.694-1.892i,0.0503+0.0621i)$ for user~1, and $\mathbf{h}^2=(0.6386-0.1388i,-0.8789+0.2781i,-2.7781+0.6188)$ for user~2. The structure for this particular examples under $WIP$ and the optimal structure are illustrated in Figure~\ref{fig:sce0}. Both have been computed exploiting a value iteration algorithm. We see that $WIP$ captures the optimal strategy in a large area of the state-space.

\subsection{Performance of Whittle's index policy }\label{sec:num2}

\begin{figure*}
\begin{minipage}[b]{0.48\linewidth}\centering
\includegraphics[scale=0.5]{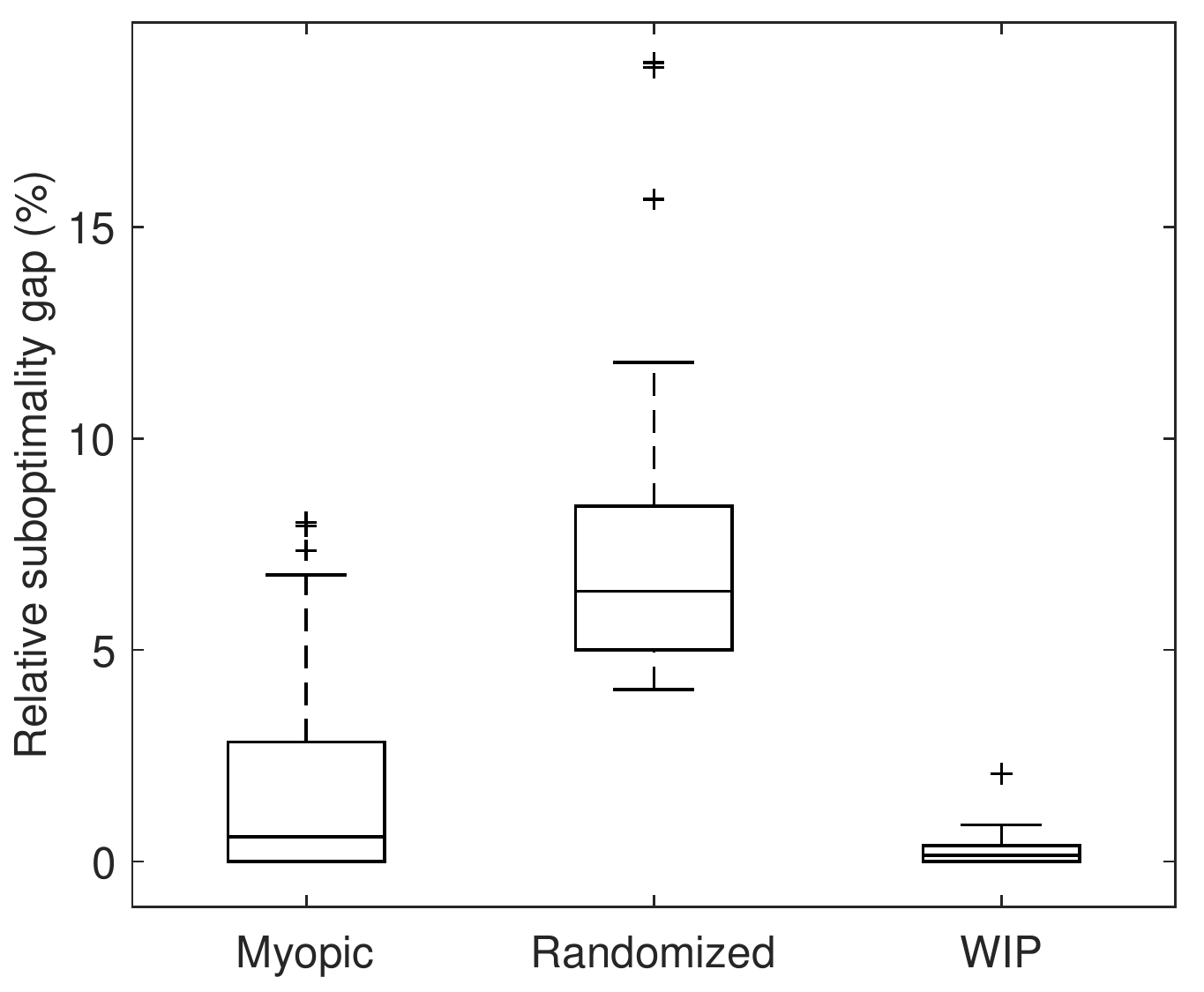}
\end{minipage}
\begin{minipage}[b]{0.48\linewidth}\centering
\includegraphics[scale=0.48]{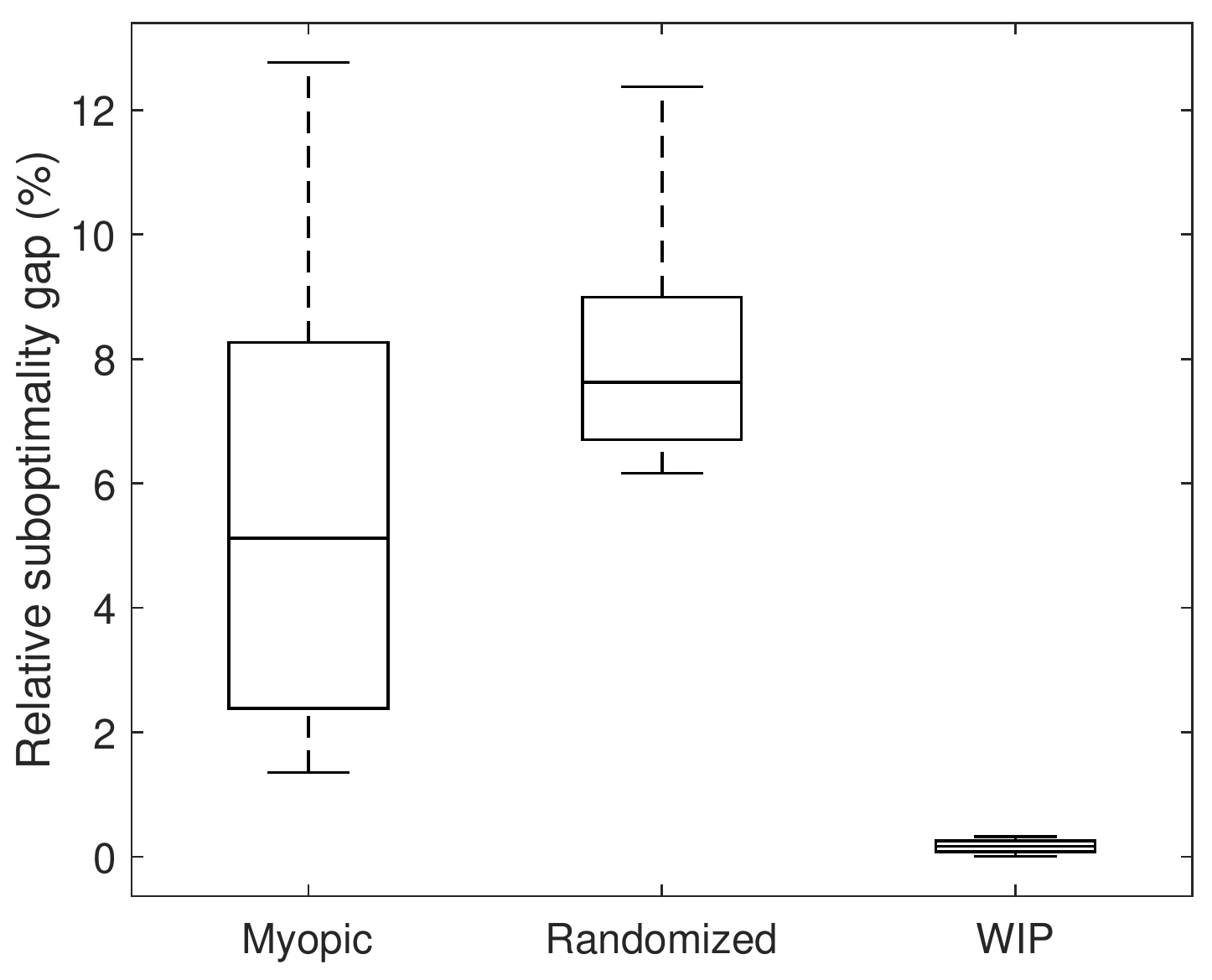}
\end{minipage}
\caption{Left: Suboptimality gap (\%) of the myopic policy, the randomized policy and Whittle's index policy ($WIP$), for 40 randomly generated examples with two users. Right:Suboptimality gap (\%) of the myopic policy, the randomized policy and Whittle's index policy ($WIP$), for 20 randomly generated examples with three users.}
\label{fig:sce1}
\end{figure*}

In this section we evaluate the performance of Whittle's index policy ($WIP$) using a value teration algorithm. In Example~1 we consider a system with two users an one pilot, and in Example~2 a system with three users and one pilot. Note that the value iteration algorithm is computationally very expensive and evaluating systems with a large number of users is out of reach. 
We are going to compare three different policies: (1) a myopic policy, which allocates the pilot to the user with highest average immediate reward, (2) a randomized policy, which allocates the pilot randomly to the users, and (3)  Whittle's index policy as defined in Corollary~\ref{cor:whittle}.

In order to use this algorithm, we need to truncate the belief state space with parameter $\tau>0$ large. We make sure $\tau$ to be large enough so that the structure of the optimal solution is not altered by the truncation.

\noindent
{\bf Example~1:} We generate 40 examples with randomly generated doubly stochastic transition probability matrices. We generate the channel vectors for each user randomly from a  zero-mean complex Gaussian  distribution. The throughput obtained by each user under both passive (no pilot has been allocated) and active actions (pilot has been allocated) are considered to be as in Section~\ref{sec:approxnumerics}. We have computed the suboptimality gap of all 40 examples (suboptimality gap$=\frac{g^{OPT}-g^{\phi}}{g^{OPT}}\cdot100$), for $\phi=WIP,\text{ randomized, and myopic}$. The results can be found in Figure~\ref{fig:sce1} (Left), where the horizontal line inside the box refers to the average suboptimality gap, the upper and lower edges of the box are the 25th and 75th percentiles and the crosses are the outliers.
We observe that the relative error of Whittle's index policy is remarkably small in all 40 examples, whereas choosing a user to allocate a pilot at random can give a relative error of up to 20\%.  $WIP$ being remarkably simple to apply, captures very closely the optimal exploration vs. exploitation trade-off. 

\noindent
{\bf Example~2:} We generate 20 examples with one pilot, three users, and randomly generated doubly stochastic transition probability matrices for each user. We generate the channel vectors for each user randomly from a zero-mean complex Gaussian distribution. The reward function is again considered to be 
$$R_i(\vec\pi_j^\tau,0)=\rho_i\frac{1}{3}\sum_{k=1}^3\log_2(1+SNR),$$
where $\rho_i=\max_r\{p_{jr}^{(\tau)}\}$. The suboptimality gap for all three policies, myopic, randomized and $WIP$, is illustrated in Figure~\ref{fig:sce1} (right). We note that $WIP$ is again a remarkably  good policy. Moreover, although the performance of the myopic policy was good in the example with two users, in this case (with three users) this does not hold anymore. This suggests that the more users there are in the system, the better the performance of $WIP$ is w.r.t. the performance of the myopic and the randomized policies. 
\begin{remark}
The optimality of the myopic policy for the two users setting has been proven in Zhao et al.~\cite{zhao2008myopic}, for a similar model to the one considered in this paper. It is therefore not surprising that the myopic policy behaves well.
\end{remark}
\section{Conclusions}
We investigate the challenging problem of pilot allocation  in wireless networks over Markovian fading channels where typically, there are less available pilots than users. At each time, the BS can know the current CSI of users to whom a pilot has been assigned. A channel belief state is estimated for other users.  The problem can be cast as a restless multi-armed bandit  problem for which obtaining an optimal solution is out of reach.  We have proposed an approximation that yields, applying the Lagrangian relaxation approach, a low-complexity policy (Whittle's index policy).  The latter has shown to perform remarkably well. Future work include  deriving Whittle's index policy for the original problem. However, this would imply 
 deriving conditions under which threshold type of policies are optimal in the original POMDP with $K>2$, an extremely difficult task.


\appendix
\subsection{Proof of Theorem~\ref{prop:threshold}}\label{app:proofthreshold}
 For ease of notation we drop the superscript $app$.
Let us define $$\nu(\vec\pi_j^\tau)=\max\left(x\in\argmax_{a\in\{0,1\}}f_\beta(\vec\pi_j^\tau,a)\right),$$ where
\begin{align*}
&f_\beta(\vec\pi_j^\tau,0):= R(\vec\pi_{j}^\tau,0)+ W +\beta V_\beta(\vec\pi_{j}^{\tau+1}),\\
&f_\beta(\vec\pi_j^\tau,1):= R^1+ \beta \sum_{k=1}^Kp_{k}^{s}V_\beta(\vec\pi_{k}^{1}),
\end{align*}
and $j\in\{1,\ldots,K\}$.
We want to prove that $\nu(\vec\pi_{j}^\tau)\leq\nu(\vec\pi_{j}^{\tau+1})$ for all $j\in\{1,\ldots,K\}$ and $\tau>0$. Since the latter implies that if it is optimal to select the user  in state $\vec \pi_j^{\tau}$ then it is also optimal to select the user in state $\vec \pi_j^{\tau+1}$.
Let $j\in\{1,\ldots,K\}$ and let $a\leq\nu(\vec\pi_j^{\tau})$ (where $a\in\{0,1\}$) then by definition
\begin{align}\label{eq:proof_threshold1}
f_\beta(\vec\pi_j^{\tau},\nu(\vec\pi_j^{\tau}))-f_\beta(\vec\pi_j^{\tau},a)\geq 0.
\end{align} 
Next we will prove
\begin{align}\label{eq:proof_threshold2}
&f_\beta(\vec\pi_j^{\tau},\nu(\vec\pi_j^{\tau}))+f_\beta(\vec\pi_j^{\tau+1},a)    \leq f_\beta(\vec\pi_j^{\tau},a)+f_\beta(\vec\pi_j^{\tau+1},\nu(\vec\pi_j^{\tau})), 
\end{align}
for all $\tau>0$, that is the supermodularity of $V_\beta(\cdot)$. The latter together with~\eqref{eq:proof_threshold1} imply 
\begin{align*}
f_\beta(\vec\pi_j^{\tau+1},a)&\leq -f_\beta(\vec\pi_j^{\tau},\nu(\vec\pi_j^{\tau}))+f_\beta(\vec\pi_j^{\tau},a)+f_\beta(\vec\pi_j^{\tau+1},\nu(\vec\pi_j^\tau))\\
&\leq f_\beta(\vec\pi_j^{\tau+1},\nu(\vec\pi_j^{\tau})), 
\end{align*}
that is, $\nu(\vec\pi_j^{\tau+1})\geq\nu(\vec\pi_j^{\tau})$, which concludes the proof. We are therefore left to prove \eqref{eq:proof_threshold2}
for which it suffices to show
\begin{align}\label{eq:proof_threshold3}
&f_\beta(\vec\pi_j^{\tau},1)+f_\beta(\vec\pi_j^{\tau+1},0)\leq f_\beta(\vec\pi_j^{\tau},0)+f_\beta(\vec\pi_j^{\tau+1},1). 
\end{align}
We substitute the expression of $f_\beta(\cdot,\cdot)$ in~\eqref{eq:proof_threshold3} and we obtain
\begin{align}\label{eq:proof_threshold4}
&\beta(p_{1}^{s}V_{\beta}(\vec\pi_1^1)+\ldots+p_{K}^{s}V_{\beta}(\vec\pi_K^1))+R(\vec\pi_j^{\tau+1},0)+\beta V_\beta(\vec\pi_j^{\tau +2})\nonumber\\
&\leq \beta(p_{1}^{s}V_{\beta}(\vec\pi_1^1)+\ldots+p_{K}^{s}V_{\beta}(\vec\pi_K^1))+R(\vec\pi_j^{\tau},0)+\beta V_\beta(\vec\pi_j^{\tau+1}).
\end{align}
By assumption $R(\vec\pi_j^{\tau},0)$ is non-increasing in $\tau$ and therefore in order to prove~\eqref{eq:proof_threshold4} it suffices to prove
\begin{align}\label{eq:proof_threshold5}
V_\beta(\vec\pi_j^{\tau +2})\leq V_\beta(\vec\pi_j^{\tau+1}),
\end{align}
i.e., $V_\beta(\cdot)$ being non-increasing. In order to prove~\eqref{eq:proof_threshold5} we will use the value iteration approach Puterman~\cite[Chap. 8]{Puterman2005}. Define $V_{\beta,0}(\vec\pi_j^\tau)=0$ for all $j\in\{1,\ldots,K\}$ and $\tau>0$ and 
\begin{align*}
V_{\beta,t+1}(\vec\pi_j^\tau)=\max\{&R(\vec\pi_j^\tau,0)+W+\beta V_{\beta,t}(\vec\pi_j^{\tau+1}),   R^1+\beta\sum_{k=1}^Kp_{k}^{s}V_{\beta,t}(\vec\pi_k^1)\}. 
\end{align*}
Observe that $V_{\beta,0}(\vec\pi_j^\tau)=0$ satisfies Inequality~\eqref{eq:proof_threshold5} (since $V_{\beta,0}(\vec\pi_j^\tau)=0$). We assume that $V_{\beta,t}(\vec\pi_j^\tau)$ satisfies~\eqref{eq:proof_threshold5} for all $j\in\{1,\ldots,K\}$ and all $\tau>0$, and we prove that $V_{\beta,t+1}(\vec\pi_j^\tau)$ satisfies the inequality as well. In order to prove the latter we need to show 
\begin{align}\label{eq:proof_threshold7}
&\max\{R(\vec\pi_j^\tau,0)+W+\beta V_{\beta,t}(\vec\pi_j^{\tau+1}),R^1+\beta\sum_{k=1}^Kp_{k}^{s}V_{\beta,t}(\vec\pi_k^1)\}\nonumber\\
&\geq\max\{R(\vec\pi_j^{\tau+1},0)+W+\beta V_{\beta,t}(\vec\pi_j^{\tau+2});   R^1+\beta\sum_{k=1}^Kp_{k}^{s}V_{\beta,t}(\vec\pi_k^1)\}.
\end{align}
Define $a(\vec\pi_j^{\tau})\in\{0,1\}$ as the action that is prescribed in state  $\vec\pi_j^{\tau}$. Since $V_{\beta,t}(\cdot)$ satisfies~\eqref{eq:proof_threshold5} we can argue on the monotonicity of the solution for $V_{\beta,t}(\cdot)$, i.e., $(a(\vec\pi_j^{\tau}),a(\vec\pi_j^{\tau+1}))\in\{(0,0),(0,1),(1,1)\}$. Therefore, it suffices to show Inequality~\eqref{eq:proof_threshold7} for the latter three options. Let us first assume $(a(\vec\pi_j^{\tau}),a(\vec\pi_j^{\tau+1}))=(0,0)$. Then~\eqref{eq:proof_threshold7} reduces to
\begin{align*}
&R(\vec\pi_j^\tau,0)+\beta V_{\beta,t}(\vec\pi_j^{\tau+1})\geq R(\vec\pi_j^{\tau+1},0)+\beta V_{\beta,t}(\vec\pi_j^{\tau+2}).
\end{align*}
The latter is satisfied due to the assumption that $R(\cdot,0)$ is non-increasing (A2) and the induction assumption that states that $V_{\beta,t}(\cdot)$ is non-increasing. We now assume $(a(\vec\pi_j^{\tau}),a(\vec\pi_j^{\tau+1}))=(1,1)$ and then~\eqref{eq:proof_threshold7} writes
\begin{align}\label{eq:proof_threshold8}
 &R^1+\beta\sum_{k=1}^Kp_{k}^{s}V_{\beta,t}(\vec\pi_k^1)\geq R^1+ \beta\sum_{k=1}^Kp_{k}^{s}V_{\beta,t}(\vec\pi_k^1),
\end{align}
which is obviously true. The last case, that is, $(a(\vec\pi_j^{\tau}),a(\vec\pi_j^{\tau+1}))=(0,1)$ follows from the $(1,1)$ case.

\subsection{Verification of conditions 8.10.1- 8.10.4' in Puterman~\cite{Puterman2005}}\label{append:discountedLimit}
We prove here that the conditions 8.10.1-8.10.4 and 8.10.4' in Puterman~\cite{Puterman2005} are satisfied. They imply that the relaxed long-run expected average reward, has a limit and can be obtained either  letting the discount factor $\beta\to1$ in the expected discounted reward model, or solving the average optimality equation that corresponds to the average reward model (Equation (8.10.9) in \cite{Puterman2005}). 
\begin{itemize}[noitemsep]
\item {\it Condition 8.10.1 in~\cite{Puterman2005}:} For all $\vec\pi_j^\tau\in\Pi$ $-\infty<R(\vec\pi_j^\tau,a(\vec\pi_j^\tau))<C$, for a constant $C<\infty$. The latter is obvious from the assumption that $0\leq R(\vec\pi_j^\tau,a(\vec\pi_j^\tau))<R^1<\infty$.
\item {\it Condition 8.10.2 in~\cite{Puterman2005}:} For all $\vec\pi_j^\tau\in\Pi$ and $0\leq \beta < 1$, $V_\beta(\vec\pi_j^\tau)>-\infty$, where 
\begin{align*}
V_{\beta}(\vec\pi_j^\tau)=\max\{&R(\vec\pi_j^\tau,1)+\beta\sum_{\overline \pi\in\Pi}q^1(\vec\pi_j^\tau,\overline\pi)V_\beta(\overline\pi);\\
& R(\vec\pi_j^\tau,0)+W+\beta \sum_{\overline \pi\in\Pi} q^0(\vec\pi_j^\tau,\overline\pi)V_\beta(\overline\pi)\}.
\end{align*}
The function $ R(\vec\pi_j^\tau,a(\vec\pi_j^\tau))$ being greater than or equal to 0  implies $V_{\beta}(\vec\pi_j^\tau)\geq0$, therefore condition 8.10.2 is satisfied.
\item {\it Condition 8.10.3 in~\cite{Puterman2005}:} There exists $0<C<\infty$ such that for all $\vec\pi_j^\tau,\vec\pi_i^{\tau'}\in\Pi$, $|V_{\beta}(\vec\pi_j^\tau)-V_{\beta}(\vec\pi_i^{\tau'})|\leq C$. We have shown that $V_{\beta}(\cdot)\geq 0$ and that $V_\beta(\cdot)$ is a non-increasing function (done in Lemma~\ref{lemma_monotonicity_V}, below). W.l.o.g. assume $V_\beta(\vec\pi_1^1)=\max_i\{V_\beta(\vec\pi_i^1)\}$. It therefore suffices to show that $\max_j\{V_{\beta}(\vec\pi_j^1)\}<\infty$, since in that case the inequality that be want to prove would be satisfied taking $C=V_\beta(\vec\pi_1^1)$. This is  proven in Lemma~\ref{lemma1}, see below.
\item {\it Condition 8.10.4 in~\cite{Puterman2005}:} There exists a non-negative function $F(\vec\pi_j^\tau)$ such that 
\begin{enumerate}
\item $F(\vec\pi_j^\tau)<\infty$ for all $\vec\pi_j^\tau\in\Pi$,
\item for all $\vec\pi_j^\tau\in\Pi$, and all $0\leq \beta<1$, $V_\beta(\vec\pi_j^\tau)-V_\beta(\vec\pi_1^{1})\geq-F(\vec\pi_j^\tau)$ and,
\item there exists $a\in \{0,1\}$ s.t 
$$
\sum_{\overline\pi\in\Pi}q^a(\vec\pi_1^1,\overline\pi)F(\overline\pi_j^\tau)<\infty.
$$
\end{enumerate}
It suffices to take $F(\cdot)= C$, and all three items above are satisfied. In order to prove condition 8.10.4' it suffices to extend the result in item 3) above to all $a\in\{0,1\}$ and all $\vec\pi_j^1\in\Pi$.
\end{itemize}

\begin{lemma}\label{lemma_monotonicity_V}
Let $V_{\beta}^{app}(\vec\pi_j^\tau)$ be the value function that corresponds to Approximation~\eqref{equation:app}, in state $\vec\pi^\tau_j$. Then,  $V_{\beta}^{app}(\vec\pi_j^\tau)$ is non-increasing in $\tau$ for all $j\in\{1,\ldots,K\}$.
\end{lemma}
\begin{proof}
We want to prove that $V_\beta^{app}(\vec\pi_j^\tau)\geq V_\beta^{app}(\vec\pi_j^{\tau+1})$ for all $\tau>0$ and all $j\in\{1,\ldots,K\}$. 

We drop the superscript $app$ from the notation of $V_\beta(\cdot)$ throughout the proof. We will prove the monotonicity of $V_\beta(\cdot)$ using the Value Iteration algorithm. Let us define $V_{\beta,0}(\vec\pi_j^\tau)=0$ for all $j\in\{1,\ldots,K\}$ and $\tau>0$, and 
\begin{align}\label{another_equation}
V_{\beta,t+1}(\vec\pi_j^\tau)=\max\{&R(\vec\pi_j^\tau,0)+W+\beta V_{\beta,t}(\vec\pi_j^{\tau+1});     R^1+\beta\sum_{k=1}^Kp_{k}^sV_\beta(\vec\pi^1_k)\}.
\end{align}
We now prove that $V_{\beta,t}(\vec\pi_j^\tau)\geq V_{ \beta,t}(\vec\pi_j^{\tau+1})$ for all $t\geq 0$ using an induction argument. Note that the latter is obvious for $t=0$ since by definition $V_{\beta,0}(\vec\pi_j^\tau)=0$ for all $j\in\{1,\ldots,K\}$ and $\tau>0$. We assume $V_{\beta,t}(\cdot)$ to be non-increasing and we prove $V_{\beta,t+1}(\cdot)$ to be non-increasing.  To prove $V_{\beta,t+1}(\vec\pi^\tau_j)\geq V_{\beta,t+1}(\vec\pi^{\tau+1}_j)$, by definition of $V_{\beta,t+1}(\cdot)$ in~\eqref{another_equation}, we have to show that 
\begin{align}\label{eq_inequality}
&\max\{R(\vec\pi_j^\tau,0)+W+\beta V_{\beta,t}(\vec\pi_j^{\tau+1});R^1+\beta\sum_{k=1}^Kp_{k}^sV_\beta(\vec\pi^1_k)\}\nonumber\\
&\geq \max\{R(\vec\pi_j^{\tau+1},0)+W+\beta V_{\beta,t}(\vec\pi_j^{\tau+2}); R^1+\beta\sum_{k=1}^Kp_{k}^sV_\beta(\vec\pi^1_k)\}.
\end{align}
Arguing on the monotonicity of $V_{\beta,t}(\cdot)$ (induction assumption), we have that $(a(\vec\pi_j^\tau),a(\vec\pi_j^{\tau+1}))\in\{(0,0),(0,1),(1,1)\}$, where $a(\vec\pi_j^\tau)$ represents the optimal action in state $\vec\pi_j^\tau$. Therefore, to show that~\eqref{eq_inequality} is satisfied, it suffices to show inequality~\eqref{eq_inequality} for $(a(\vec\pi_j^\tau),a(\vec\pi_j^{\tau+1}))\in\{(0,0),(0,1),(1,1)\}$. Let us first assume $(a(\vec\pi_j^\tau),a(\vec\pi_j^{\tau+1}))=(1,1)$, then inequality~\eqref{eq_inequality} is obvious since both the RHS and the LHS are identical. If $(a(\vec\pi_j^\tau),a(\vec\pi_j^{\tau+1}))=(0,1)$, then from the definition of $V_{\beta,t+1}(\vec\pi_j^\tau)$, $a(\vec\pi_j^\tau)=0$ implies
\begin{align*}
R(\vec\pi_j^\tau,0)+W+\beta V_{\beta,t}(\vec\pi_j^{\tau+1})\geq R^1+\beta\sum_{k=1}^Kp_{k}^sV_\beta(\vec\pi^1_k),
\end{align*}
and the latter implies inequality~\eqref{eq_inequality} to be satisfied for $(a(\vec\pi_j^\tau),a(\vec\pi_j^{\tau+1}))=(0,1)$. We are left with the case $(a(\vec\pi_j^\tau),a(\vec\pi_j^{\tau+1}))=(0,0)$, in order for~\eqref{eq_inequality} to be satisfied, we need to show that
\begin{align*}
&R(\vec\pi_j^\tau,0)+W+\beta V_{\beta,t}(\vec\pi_j^{\tau+1})   \geq R(\vec\pi_j^{\tau+1},0)+W+\beta V_{\beta,t}(\vec\pi_j^{\tau+2}),
\end{align*}
which is true due to A1 and the induction assumption, i.e., $V_{\beta,t}(\cdot)$ to be non-increasing. This concludes the proof.
\end{proof}

\begin{lemma}\label{lemma1} Let $V_{\beta}(\cdot)$ denote the value function that corresponds to Approximation  in Equation~\eqref{equation:app}, with $0\leq\beta<1$ the discounted factor. Let $\vec\Gamma=(\Gamma_1(W),\ldots,\Gamma_K(W))$ be the optimal threshold policy for a fixed $W<\infty$. Then  $V_{\beta}(\vec\pi_j^{\tau})<\infty$ for all $j\in\{1,\ldots,K\}$ and $\tau>0$.
\end{lemma}

\begin{proof} For ease of notation, we will denote by $R^0(\vec\pi_j^\tau):=R(\vec\pi_j^\tau,0)$, i.e., the average immediate reward under action passive, throughout the proof.

We have proven in Theorem~\ref{prop:threshold} that an optimal solution is of threshold type. Let $\vec\Gamma(W)=(\Gamma_1(W),\ldots,\Gamma_K(W))$ be the optimal threshold for a given $W$. Then it can be shown that
\begin{align}\label{eq:valuefunction}
V_\beta(\vec\pi_j^1)= &\sum_{i=1}^{\Gamma_j(W)}\beta^{i-1}(R^0(\vec\pi_j^i)+W)    +\beta^{\Gamma_j(W)}(R^1+\beta\sum_{k=1}^Kp^s_kV_\beta(\vec\pi_k^1)),
\end{align}
for all $j\in\{1,\ldots,K\}$. From the $j=1$ case we obtain
\begin{align}\label{eq:VF_j=1}
&\sum_{k=1}^Kp^s_kV_\beta(\vec\pi_k^1)    = -\frac{R^1}{\beta}+\frac{V_\beta(\vec\pi_1^1)-\sum_{i=1}^{\Gamma_1(W)}\beta^{i-1}(R^0(\vec\pi_1^i)+W)}{\beta^{\Gamma_1(W)+1}}.
\end{align}
Substituting the latter in Equation~\eqref{eq:valuefunction} for the $j>1$ case, we obtain
\begin{align}\label{eq:withrespectto1}
V_\beta(\vec\pi_j^1)=&\sum_{i=1}^{\Gamma_j(W)}\beta^{i-1}(R^0(\vec\pi_j^i)+W)    +\frac{\beta^{\Gamma_j(W)+1}}{\beta^{\Gamma_1(W)+1}}\left(V_\beta(\vec\pi_1^1)-\sum_{i=1}^{\Gamma_1(W)}\beta^{i-1}(R^0(\vec\pi_1^i)+W)\right),
\end{align}
for all $j\neq1$.
We now substitute the latter in Equation~\eqref{eq:VF_j=1} and solve for $V_\beta(\vec\pi_1^1)$. We obtain
\begin{align}\label{eq:V_1_1}
V_\beta(\vec\pi_1^1)=\bigg[&\sum_{i=1}^{\Gamma_1(W)}\beta^{i-1}(R^0(\vec\pi_1^i)+W)+\beta^{\Gamma_1(W)}R^1\nonumber\\
&+\beta^{\Gamma_1(W)+1}\sum_{k=2}^Kp_k^s\sum_{i=1}^{\Gamma_k(W)}\beta^{i-1}(R^0(\vec\pi_k^i)+W)\nonumber\\
&-\sum_{k=2}^Kp_k^s\beta^{\Gamma_k(W)+1}\sum_{i=1}^{\Gamma_1(W)}\beta^{i-1}(R^0(\vec\pi_1^i)+W)\bigg]\cdot \left[1-\sum_{k=1}^{K}p_k^s\beta^{\Gamma_k(W)+1}\right]^{-1}.
\end{align}
If we assume that $\vec\pi\neq e_j$ for any $j\in\{1,\ldots,K\}$ then $V_\beta(\vec\pi_1^1)<\infty$. The latter together with  Equation~\eqref{eq:withrespectto1} imply $V_\beta(\vec\pi_j^1)<\infty$ for all $j\in\{1,\ldots,K\}$. This concludes the proof.
\end{proof}

\subsection{Explicit expression of $\omega_i$}\label{app:omegas}
We aim at solving the balance equations for the Approximation in Equation~\eqref{equation:app}. Note that $\alpha^{\vec\Gamma}(\vec\pi_i^\tau)=\alpha^{\vec\Gamma}(\vec\pi_i^{\tau'})$ for all $\tau,\tau'\leq\Gamma_i+1$, that is, the probability of being in state $\vec\pi_i^\tau$ equals that of state $\vec\pi_i^{\tau'}$ if passive action is prescribed in them or, if $\tau=\Gamma_i+1$. Hence, $\omega_j$ is the solution of 
\begin{align*}
&\omega_j(1-p_j^s)=\sum_{i=1}^{j-1}p_i^s\omega_i+\sum_{i=j+1}^Kp_i^s\omega_i, \hbox{ for all } j\in\{1,\ldots,K\},
\end{align*} 
and $\sum_{k=1}^K\omega_k=1$. Hence,   $\omega_j=p_j^s$.

\subsection{Proof of Theorem~\ref{prop:whittle}}\label{app:Whittles index}
 The following definition will be exploited throughout the proof:
\begin{align*}
g^{\vec\Gamma}(W)=\mathbb{E}(R(b^{\vec\Gamma},a^{\vec\Gamma}(b^{\vec\Gamma})))+W\sum_{k=1}^K\sum_{j=1}^{\Gamma_k}\alpha^{\vec\Gamma}(\vec\pi_k^{j}).
\end{align*}
Note that $g^{\vec\Gamma}(W)$ refers to the average reward obtained under threshold policy $\vec\Gamma$ and subsidy for passivity $W$.

We will assume $I\in\mathbb{N}\cup\{0,\infty\}$ to be the number of steps until the algorithm stops. Therefore $\Gamma_j^I=\infty$ for all $j\in\{1,\ldots,K\}$. We set $W_i:=W_I$ for all $i\geq I$. We will prove that $W_0<W_1<\ldots<W_\infty$. By definition we have that $\Gamma^i$ is increasing in $i$, that is, $\Gamma_j^i\geq\Gamma_j^{i-1}$ for all $j$ and $i>0$.
Let us first prove that $W_i<W_{i+1}$. By the definition of $W_i$ we have that
\begin{align*}
&\frac{\mathbb{E}(R(b^{\vec\Gamma^{i-1}},a^{\vec\Gamma^{i-1}}(b^{\vec\Gamma^{i-1}})))-\mathbb{E}(R(b^{\vec\Gamma^i},a^{\vec\Gamma^i}(b^{\vec\Gamma^i})))}{\sum_{j=1}^{K}\left(\sum_{r=1}^{\Gamma_j^i}\alpha^{\vec\Gamma^i}(\vec\pi^r_j)-\sum_{r=1}^{\Gamma_j^{i-1}}\alpha^{\vec\Gamma^{i-1}}(\vec\pi^r_j)\right)}\nonumber\\
<&\frac{\mathbb{E}(R(b^{\vec\Gamma^{i-1}},a^{\vec\Gamma^{i-1}}(b^{\vec\Gamma^{i-1}})))-\mathbb{E}(R(b^{\vec\Gamma^{i+1}},a^{\vec\Gamma^{i+1}}(b^{\vec\Gamma^{i+1}})))}{\sum_{j=1}^{K}\left(\sum_{r=1}^{\Gamma_j^{i+1}}\alpha^{\vec\Gamma^{i+1}}(\vec\pi^r_j)-\sum_{r=1}^{\Gamma_j^{i-1}}\alpha^{\vec\Gamma^{i-1}}(\vec\pi^r_j)\right)},
\end{align*}
since $\sum_{j=1}^{K}\sum_{r=1}^{\Gamma_j^{i}}\alpha^{\vec\Gamma^{i}}(\vec\pi^r_j)$ is non-decreasing in $i$ we have
\begin{align*}
&[\mathbb{E}(R(b^{\vec\Gamma^{i-1}},a^{\vec\Gamma^{i-1}}(b^{\vec\Gamma^{i-1}})))-\mathbb{E}(R(b^{\vec\Gamma^i},a^{\vec\Gamma^i}(b^{\vec\Gamma^i})))]\nonumber\\
&\cdot\left[\sum_{j=1}^{K}\left(\sum_{r=1}^{\Gamma_j^{i+1}}\alpha^{\vec\Gamma^{i+1}}(\vec\pi^r_j)-\sum_{r=1}^{\Gamma_j^{i-1}}\alpha^{\vec\Gamma^{i-1}}(\vec\pi^r_j)\right)\right]\nonumber\\
&< [\mathbb{E}(R(b^{\vec\Gamma^{i-1}},a^{\vec\Gamma^{i-1}}(b^{\vec\Gamma^{i-1}})))-\mathbb{E}(R(b^{\vec\Gamma^{i+1}},a^{\vec\Gamma^{i+1}}(b^{\vec\Gamma^{i+1}})))]\nonumber\\
&\quad\cdot\left[\sum_{j=1}^{K}\left(\sum_{r=1}^{\Gamma_j^i}\alpha^{\vec\Gamma^i}(\vec\pi^r_j)-\sum_{r=1}^{\Gamma_j^{i-1}}\alpha^{\vec\Gamma^{i-1}}(\vec\pi^r_j)\right)\right].
\end{align*}
Adding the term 
$$\mathbb{E}(R(b^{\vec\Gamma^i},a^{\vec\Gamma^i}(b^{\vec\Gamma^i})))\sum_{j=1}^{K}\left(\sum_{r=1}^{\Gamma_j^{i-1}}\alpha^{\vec\Gamma^{i-1}}(\vec\pi^r_j)-\sum_{r=1}^{\Gamma_j^{i}}\alpha^{\vec\Gamma^{i}}(\vec\pi^r_j)\right),$$
 on both sides of the latter inequality, and after some algebra we obtain $W_i<W_{i+1}$. We now prove that indeed $W_i$ for all $i$ defines Whittle's index. To show that we need to prove:
\begin{enumerate}
\item Threshold policy $\vec\Gamma^{-1}=(0,\ldots,0)$ is optimal for the single-arm average reward POMDP problem for all $W$ such that $W<W_0$.
\item Threshold policy $\vec\Gamma^i$ is optimal for all $W_i<W<W_{i+1}$.
\item Threshold policy $\infty$ is optimal for all $W$ such that~$W>W_I$.
\end{enumerate}
Let us first prove 1) . From the definition of $W_0$ we have that, for all $W<W_0$
\begin{align*}
&W\sum_{k=1}^K\sum_{j=1}^{\Gamma_k}\alpha^{\vec\Gamma}(\vec\pi_k^j)\leq\mathbb{E}(R(b^{\vec\Gamma^{-1}},a^{\vec\Gamma^{-1}}(b^{\vec\Gamma^{-1}})))-\mathbb{E}(R(b^{\vec\Gamma},a^{\vec\Gamma}(b^{\vec\Gamma})))\nonumber\\
&\Longrightarrow\mathbb{E}(R(b^{\vec\Gamma},a^{\vec\Gamma}(b^{\vec\Gamma})))+W\sum_{k=1}^K\sum_{j=1}^{\Gamma_k}\alpha^{\vec\Gamma}(\vec\pi_k^j)\nonumber\\
&\leq \mathbb{E}(R(b^{\vec\Gamma^{-1}},a^{\vec\Gamma^{-1}}(b^{\vec\Gamma^{-1}})))=g^{\vec\Gamma^{-1}}(W).
\end{align*}
That is, $g^{\vec\Gamma^{-1}}(W)\leq g^{\vec\Gamma}(W)$ for all $\vec\Gamma\geq(0,\ldots,0)$. Threshold policy $\vec\Gamma^{-1}$ is therefore optimal for all $W<W_0$.

We will establish 2) using an inductive argument. From the definition of $\vec\Gamma^0$ it can be seen that
\begin{align}\label{eq:ineq_whittle}
&\mathbb{E}(R(b^{\vec\Gamma^{0}},a^{\vec\Gamma^{0}}(b^{\vec\Gamma^{0}})))+W_0\sum_{k=1}^K\sum_{j=1}^{\Gamma_k^0}\alpha^{\vec\Gamma^0}(\vec\pi^{j}_k)   \geq\mathbb{E}(R(b^{\vec\Gamma},a^{\vec\Gamma}(b^{\vec\Gamma})))+W_0\sum_{k=1}^K\sum_{j=1}^{\Gamma_k}\alpha^{\vec\Gamma}(\vec\pi^{j}_k),
\end{align}
for all $\vec\Gamma$, that is, $g^{\vec\Gamma^0}(W_0)\geq g^{\vec\Gamma}(W_0)$. By the assumption that $\sum_{k=1}^K\sum_{j=1}^{\Gamma_k}\alpha^{\vec\Gamma}(\vec\pi^{j}_k)$ strictly increases in $\vec\Gamma$ and inequality~\eqref{eq:ineq_whittle} we obtain for all $\vec\Gamma\leq\vec\Gamma^0$
\begin{align*}
&\mathbb{E}(R(b^{\vec\Gamma^{0}},a^{\vec\Gamma^{0}}(b^{\vec\Gamma^{0}})))+W\sum_{k=1}^K\sum_{j=1}^{\Gamma_k^0}\alpha^{\vec\Gamma^0}(\vec\pi^{j}_k)   \geq\mathbb{E}(R(b^{\vec\Gamma},a^{\vec\Gamma}(b^{\vec\Gamma})))+W\sum_{k=1}^K\sum_{j=1}^{\Gamma_k}\alpha^{\vec\Gamma}(\vec\pi^{j}_k),
\end{align*}
that is, $g^{\vec\Gamma^0}(W)\geq g^{\vec\Gamma}(W)$ for all $\vec\Gamma\leq\vec\Gamma^0$ and $W_0<W$, in particular for all $W_0<W<W_1$. Using similar type of arguments and the definition of $W_1$ it can be seen that $g^{\vec\Gamma^0}(W_1)\geq g^{\vec\Gamma}(W_1)$ and again by monotonicity of $\sum_{k=1}^K\sum_{j=1}^{\Gamma_k}\alpha^{\vec\Gamma}(\vec\pi^{j}_k)$ we obtain $g^{\vec\Gamma^0}(W)\geq g^{\vec\Gamma}(W)$ for all $\vec\Gamma\geq\vec\Gamma^0$ and $W_0<W<W_1$. Hence, threshold policy $\vec\Gamma^0$ is optimal for $W_0<W<W_1$.
We now assume that $\vec\Gamma^{i-1}$ is the optimal threshold policy when $W_{i-1}<W<W_i$, {\it i.e.,} $g^{\vec\Gamma^i}(W)\geq g^{\vec\Gamma}(W)$ and we prove that $\vec\Gamma^{i}$ is optimal for $W_i<W<W_{i+1}$. From the definition of $W_i$ and the assumption that $\vec\Gamma^{i-1}$ is optimal for all $W_{i-1}<W<W_i$ we obtain
$
g^{\vec\Gamma^i}(W_i)=g^{\vec\Gamma^{i-1}}(W_i)\geq g^{\vec\Gamma}(W_i),
$
for all $\vec\Gamma$. Since $\sum_{k=1}^K\sum_{j=1}^{\Gamma_k}\alpha^{\vec\Gamma}(\vec\pi^{j}_k)$ is strictly increasing in $\vec\Gamma$ we obtain $g^{\vec\Gamma^i}(W)\geq g^{\vec\Gamma}(W)$ for all $\vec\Gamma\leq\vec\Gamma^i$ and $W_i<W<W_{i+1}$. Moreover, from the definition of $W_{i+1}$ we have $g^{\vec\Gamma^i}(W)\geq g^{\vec\Gamma}(W)$ for all $\vec\Gamma\geq\vec\Gamma^i$ and $W_i<W<W_{i+1}$. Therefore, $\vec\Gamma^i$ is the optimal threshold policy for all $W_i<W<W_{i+1}$.

Item 3) can now easily be proven using the same argument in each iteration step. This concludes the proof.

\subsection{Proof of Lemma~\ref{prop:expressionwhittle}}\label{app:explicit_Whittle}
Let us assume that in Step i, $\vec\Gamma^i$ is such that $\sum_{j=1}^K\Gamma_j^i=(\sum_{j=1}^K\Gamma_j^{i-1})+1$ and $\Gamma_j^i\geq\Gamma_j^{i-1}$ for all $j\in\{1,\ldots,K\}$, then there exists $u\in\{1,\ldots,K\}$ such that $\Gamma_u^i=\Gamma_u^{i-1}+1$ and $\Gamma_j^i=\Gamma_j^{i-1}$ for all $j\neq u$. 
By Proposition~\ref{prop:whittle} we have
\begin{align}\label{eq:whittle_index}
W_i=&\frac{\mathbb{E}(R(X^{\vec\Gamma^{i-1}},a(X^{\vec\Gamma^{i-1}})))-\mathbb{E}(R(X^{\vec\Gamma^{i}},a(X^{\vec\Gamma^{i}})))}{\sum_{j=1}^{K}\left(\sum_{r=1}^{\Gamma_j^i}\alpha^{\vec\Gamma^i}(\vec\pi^r_j)-\sum_{r=1}^{\Gamma_j^{i-1}}\alpha^{\vec\Gamma^{i-1}}(\vec\pi^r_j)\right)}.
\end{align}
The numerator in Equation~\eqref{eq:whittle_index}, after substitution of $\mathbb{E}(R(X^{\vec\Gamma},a(X^{\vec\Gamma})))=\sum_{k=1}^K\sum_{j=1}^{\Gamma_k}R(\vec\pi_k^{j},0)\alpha^{\vec\Gamma}(\vec\pi_k^{j})+R^1\sum_{k=1}^K\alpha^{\vec\Gamma}(\vec\pi_k^{\Gamma_k+1})$, reads
\begin{align}\label{eqeqeq}
&\sum_{k=1}^K\sum_{j=1}^{\Gamma_k^{i-1}}R(\vec\pi_k^{j},0)\left(\alpha^{\vec\Gamma^{i-1}}(\vec\pi_k^{j})-\alpha^{\vec\Gamma^{i}}(\vec\pi_k^j)\right)    -R(\vec\pi_u^{\Gamma_u^{i}},0)\alpha^{\vec\Gamma^i}(\vec\pi_u^{\Gamma_u^{i}})\nonumber\\
&+R^1\sum_{k=1}^K\left(\alpha^{\vec\Gamma^{i-1}}(\vec\pi_k^{\Gamma^{i-1}_k+1})
-\alpha^{\vec\Gamma^{i}}(\vec\pi_k^{\Gamma^{i}_k+1})\right).\end{align}
Since $\alpha^{\vec\Gamma}(\vec\pi^i_j)=\frac{\omega_j}{\sum_{r=1}^K(\Gamma_r+1)\omega_r}$, Equation~\eqref{eqeqeq} simplifies to
\begin{align}\label{eqeq}
&\omega_u\frac{\sum_{k=1}^K\sum_{j=1}^{\Gamma_k^{i-1}}R(\vec\pi_k^{j},0)\omega_k-R(\vec\pi_u^{\Gamma_u^i},0)\sum_{k=1}^K(\Gamma_k^{i-1}+1)\omega_k}{(\sum_{r=1}^K(\Gamma_r^i+1)\omega_r)\cdot(\sum_{r=1}^K(\Gamma_r^{i-1}+1)\omega_r)}\nonumber\\
&+\omega_u\frac{R^1\sum_{k=1}^K\omega_k}{(\sum_{r=1}^K(\Gamma_r^i+1)\omega_r)\cdot(\sum_{r=1}^K(\Gamma_r^{i-1}+1)\omega_r)}
\end{align}
Substituting the value of $\alpha^{\vec\Gamma}(\cdot)$ in the denominator of Equation~\eqref{eq:whittle_index}, the denominator reduces to
\begin{align}\label{eq:beste}
&-\sum_{k=1}^K\sum_{r=1}^{\Gamma^{i-1}_k}\omega_k\frac{\omega_u }{(\sum_{r=1}^K(\Gamma_r^i+1)\omega_r)\cdot(\sum_{r=1}^K(\Gamma_r^{i-1}+1)\omega_r)}\nonumber\\
& + \frac{\omega_u\sum_{r=1}^K(\Gamma_r^{i-1}+1)\omega_r}{(\sum_{r=1}^K(\Gamma_r^i+1)\omega_r)\cdot(\sum_{r=1}^K(\Gamma_r^{i-1}+1)\omega_r)}.
\end{align}
To obtain the explicit expression of Equation~\eqref{eq:whittle_index} it now suffices to divide the expression of the numerator as given by Equation~\eqref{eqeq} with the expression of the denominator as given by Equation~\eqref{eq:beste}, that is,
\begin{align*}
R^1+ \frac{\sum_{k=1}^K\sum_{j=1}^{\Gamma_k^{i-1}}R(\vec\pi_k^{j},0)\omega_k-R(\vec\pi_u^{\Gamma_u^i},0)\sum_{k=1}^K(\Gamma_k^{i-1}+1)\omega_k}{\sum_{k=1}^K\omega_k}.
\end{align*}
Since $\sum_{k=1}^K\omega_k=1$ and $\Gamma^{i}_u=\Gamma^{i-1}_u+1$ the explicit expression of $W_i$ is given by
\begin{align*}
W_i=&R^1+\sum_{k=1}^K\sum_{j=1}^{\Gamma_k^{i-1}}R(\vec\pi_k^{j},0)\omega_k   -R(\vec\pi_u^{\Gamma_u^{i-1}+1},0)\sum_{k=1}^K(\Gamma_k^{i-1}+1)\omega_k,
\end{align*}
which concludes the proof.

\subsection{Proof of Lemma~\ref{lemma:valuefunctions}}\label{app:proofvaluefunctions}
We will proof the inequality $V_\beta^{max}(\cdot)\geq V_{\beta}(\cdot)$. The  inequality that  corresponds to $V_\beta^{min}$ can be proved similarly. Let us use the Value Iteration. Define  $V_{\beta,0}^{max}(\cdot)= V_{\beta,0}(\cdot)\equiv 0$, 
\begin{align*}
V_{\beta,t+1}(\vec\pi_j^\tau)=\max\{&R(\vec\pi_j^\tau,0)+W+\beta V_{\beta,t}(\vec\pi_j^{\tau+1});\nonumber\\
&R(\vec\pi_j^\tau,1) +\beta\sum_{i=1}^Kp_{ji}^{(\tau)}V_{\beta,t}(\vec\pi_i^1)\}, \text { and }\nonumber\\
V_{\beta,t+1}^{max}(\vec\pi_j^\tau)=\max\{&R(\vec\pi_j^\tau,0)+W+\beta V_{\beta,t}^{max}(\vec\pi_j^{\tau+1});\nonumber\\
&R(\vec\pi_j^\tau,1) +\beta\max_i\{V_{\beta,t}(\vec\pi_i^1)\}\}.
\end{align*}
Note that $V_{\beta,0}^{max}(\cdot)\geq V_{\beta,0}(\cdot)$. We will now prove the result by induction.
We assume $V_{\beta,t}^{max}(\cdot)\geq V_{\beta,t}(\cdot)$ and we prove $V_{\beta,t+1}^{max}(\cdot)\geq V_{\beta,t+1}(\cdot)$. To prove the latter it suffices to show 
\begin{align}\label{ineq:lemma2}
\max\{&R(\vec\pi_j^\tau,0)+W+\beta V_{\beta,t}(\vec\pi_j^{\tau+1});\nonumber\\
&R(\vec\pi_j^\tau,1) +\beta\sum_{i=1}^Kp_{ji}^{(\tau)}V_{\beta,t}(\vec\pi_i^1)\}, \nonumber\\
\leq \max\{&R(\vec\pi_j^\tau,0)+W+\beta V_{\beta,t}^{max}(\vec\pi_j^{\tau+1});\nonumber\\
&R(\vec\pi_j^\tau,1) +\beta\max_i\{V_{\beta,t}^{max}(\vec\pi_i^1)\}\}.
\end{align}
We first assume that the maximizer in both sides of Inequality~\eqref{ineq:lemma2} is the passive action. Then it suffices to show 
$$V_{\beta,t}(\vec\pi_j^{\tau+1})\leq V_{\beta,t}^{max}(\vec\pi_j^{\tau+1}),$$
which is satisfied due to the induction assumption. Let us now assume that the maximizer in both sides of  Inequality~\eqref{ineq:lemma2} is the  active action. Then to prove Inequality~\eqref{ineq:lemma2} we need to show that
\begin{align}\label{oneequation}
\sum_{i=1}^Kp_{ji}^{(\tau)}V_{\beta,t}(\vec\pi_i^1)\leq \max_i\{V_{\beta,t}^{max}(\vec\pi_i^1)\}.
\end{align}
We have
\begin{align*}
\sum_{i=1}^Kp_{ji}^{(\tau)}V_{\beta,t}(\vec\pi_i^1)&\leq\sum_{i=1}^Kp_{ji}^{(\tau)}V_{\beta,t}^{max}(\vec\pi_i^1)\nonumber\\ 
&\leq\max_i\{V_{\beta,t}^{max}(\vec\pi_i^1)\},
\end{align*}
which proves~\eqref{oneequation}. In the latter we have used the induction assumption in the first inequality and the fact that $p_{ji}^{(\tau)}$ is a probability distribution for all $\tau$ in the second inequality. The cases in which the maximizers are active and passive actions, and passive and active actions follow from the previous two cases. We have therefore proved that $V_{\beta,t}(\cdot)\leq V_{\beta,t}^{max}(\cdot)$ for all $t$. Since $\lim_{t\to\infty}V_{\beta,t}=V_\beta$ (and similarly for $V_\beta^{max}$) then $V_{\beta}(\cdot)\leq V_{\beta}^{max}(\cdot)$. This concludes the proof.

\subsection{Proof of Proposition~\ref{prop:performance_bounds}}\label{prop:proof_performancebounds}
 In Lemma~\ref{lemma:} we have proven that
\begin{align*}
&V_\beta^{max}(\vec\pi_j^\tau)\geq V_\beta(\vec\pi_j^\tau), \text{ and } V_\beta^{min}(\vec\pi_j^\tau)\leq V_\beta(\vec\pi_j^\tau),
\end{align*}
for all $\vec\pi_j^\tau\in\Pi$. From the latter we obtain
\begin{align*}
&V_\beta^{max}(\vec\pi_j^\tau)-V_\beta^{app}(\vec\pi_j^\tau)\geq V_\beta(\vec\pi_j^\tau)-V_\beta^{app}(\vec\pi_j^\tau),\nonumber\\ &V_\beta^{min}(\vec\pi_j^\tau)-V_\beta^{app}(\vec\pi_j^\tau)\leq V_\beta(\vec\pi_j^\tau)-V_\beta^{app}(\vec\pi_j^\tau),
\end{align*}
for all $\vec\pi_j^\tau\in\Pi$.
By~\cite[Theorem 8.10.7]{Puterman2005} we have that $g(W)=\lim_{\beta\to1}(1-\beta)V_\beta(\vec\pi_j^\tau)$ (similarly for $V_\beta^{max}, V_\beta^{min}$ and $V_\beta^{app}$). Therefore,
\begin{align*}
&\lim_{\beta\to1}(1-\beta)\left(V_\beta^{max}(\vec\pi_j^\tau)-V_\beta^{app}(\vec\pi_j^\tau)\right)\nonumber\\
\geq&\lim_{\beta\to1}(1-\beta)\left( V_\beta(\vec\pi_j^\tau)-V_\beta^{app}(\vec\pi_j^\tau)\right),\nonumber\\ 
&\lim_{\beta\to1}(1-\beta)\left(V_\beta^{min}(\vec\pi_j^\tau)-V_\beta^{app}(\vec\pi_j^\tau)\right)\nonumber\\
\leq& \lim_{\beta\to1}(1-\beta)\left(V_\beta(\vec\pi_j^\tau)-V_\beta^{app}(\vec\pi_j^\tau)\right),
\end{align*}
for all $\vec\pi_j^\tau\in\Pi$, that is,
\begin{align*}
g^{max}(W)-g^{app}(W)&\geq g(W)-g^{app}(W)\nonumber\\
&\geq g^{min}(W)-g^{app}(W). \nonumber\\ 
\end{align*}
Define $D(W):= \max\{1-\frac{g^{app}(W)}{g^{max}(W)},\frac{g^{app}(W)}{g^{min}(W)}-1\}$. The explicit expression of $D(W)$ can be found in Appendix~\ref{appendix:explicit_D}. Hence,
\begin{align*}
&\bigg| 1-\frac{g^{app}(W)}{g(W)}\bigg|\leq D(W).\nonumber\\ 
\end{align*}

\subsection{Explicit expression of $D(W)$}\label{appendix:explicit_D}
To derive the explicit expression of $D(W)$, we need to obtain the expressions of $g^{min}(W), g^{max}(W)$ and $g^{app}(W)$. From the proof of Lemma~\ref{lemma1} and the results in Appendix~\ref{append:discountedLimit}, we have that
$$
g^{app}(W)=\lim_ {\beta\to1}(1-\beta)V_\beta^{app}(\vec\pi_1^1),
$$
where $V_\beta^{app}(\vec\pi_1^1)$ is as given in Equation~\eqref{eq:V_1_1} (after adding the superscript $app$). Note that when computing the limit as $\beta\to1$ we encounter a $0/0$ indetermination. After applying L'Hopital's rule it can easily be seen that
$$
g^{app}(W)=\frac{R^1+\sum_{k=1}^Kp_k^s\sum_{i=1}^{\tau_k(W)}(R(\vec\pi_k^i,0)+W)}{\sum_{k=1}^K(\tau_k(W)+1)p_k^s}.
$$
To obtain the closed-form expressions of $g^{max}(W)$ and $g^{min}(W)$ we need to follow the same steps as those used in the derivation of $g^{app}(W)$. That is, we need to (i) show that an optimal solution of Equations~\eqref{eq:bellman_approximations} and~\eqref{eq:bellman_approximations_min} is a threshold type of policy, (ii) obtain the explicit expressions of $V_\beta^{max}(\cdot)$ and $V_\beta^{min}(\cdot)$, (iii) prove conditions 8.10.1-8.10.4' in Puterman~\cite{Puterman2005} to be satisfied, and finally, (iv) compute $g^{min}(W)$ by taking the limit of $(1-\beta)V_\beta^{min}(\cdot)$ as $\beta\to1$ (similarly for $g^{max}(W)$).
The first three steps can easily be done using the same arguments that have been used for Approximation~1. Step (i) is similar to the proof of Theorem~\ref{prop:threshold}, step (ii) can be done using the arguments in the proof of Lemma~\ref{lemma1}, and step (iii) can be proven through the ideas exploited in Appendix~\ref{append:discountedLimit}. After showing the first three steps one obtains
\begin{align*}
g^{max}(W)&=\frac{R^1+\sum_{i=1}^{\overline\tau_{\sigma_{max}}(W)}(R(\vec\pi_{\sigma_{max}}^i,0)+W)}{\overline\tau_{\sigma_{max}}(W)+1},\\
g^{min}(W)&=\frac{R^1+\sum_{i=1}^{\underline\tau_{\sigma_{min}}(W)}(R(\vec\pi_{\sigma_{min}}^i,0)+W)}{\underline\tau_{\sigma_{min}}(W)+1},
\end{align*}
where $\sigma_{max}=\argmax_j\{(R^1+\sum_{i=1}^{\overline\tau_j(W)}(R(\vec\pi_{j}^i,0)+W))/(\overline\tau_j(W)+1)\}$,  similarly, $\sigma_{min}=\argmin_j\{(R^1+\sum_{i=1}^{\underline\tau_j(W)}(R(\vec\pi_{j}^i,0)+W))/(\underline\tau_j(W)+1)\}$, and $\overline\tau_i(W)$ and $\underline\tau_i(W)$ refer to the optimal threshold policies of problems ~\eqref{eq:bellman_approximations} and~\eqref{eq:bellman_approximations_min}, respectively. Note that the optimal threshold policies $\tau_i(W),\overline\tau_i(W)$ and $\underline\tau_i(W)$, can be computed from the Bellman equations by equating the value obtained from passive action and the value obtained from active action. Having said that, we obtain
\begin{align}
D(W)=\max\bigg\{&1-\frac{\overline\tau_{\sigma_{max}}(W)+1}{\sum_{k=1}^K(\tau_k(W)+1)p_k^s}   \cdot\frac{R^1+\sum_{k=1}^Kp_k^s\sum_{i=1}^{\tau_k(W)}(R(\vec\pi_k^i,0)+W)}{R^1+\sum_{i=1}^{\overline\tau_{\sigma_{max}}(W)}(R(\vec\pi_{\sigma_{max}}^i,0)+W)};\nonumber\\
&\frac{\overline\tau_{\sigma_{min}}(W)+1}{\sum_{k=1}^K(\tau_k(W)+1)p_k^s}  \cdot\frac{R^1+\sum_{k=1}^Kp_k^s\sum_{i=1}^{\tau_k(W)}(R(\vec\pi_k^i,0)+W)}{R^1+\sum_{i=1}^{\overline\tau_{\sigma_{min}}(W)}(R(\vec\pi_{\sigma_{min}}^i,0)+W)}-1
\bigg\}.
\end{align}

\subsection{Proof of Lemma~\ref{eq:linear_fluid_sys}}\label{proof_linear_fluid}
Throughout the proof we will assume for sake of clarity, $W(\vec\pi_{1}^{\ell_1^{*},1})=W^*$, $W(\vec\pi_{j}^{\ell_j^{*}-1,1})<W^*<W(\vec\pi_{j}^{\ell_j^{*},1})$ for all $j\in\{2,\ldots,K\}$ and  $W(\vec\pi_{j}^{m_j^{*}-1,2})<W^*<W(\vec\pi_{j}^{m_j^{*},2})$ for all $j\in\{1,\ldots,K\}$. That is, for all $j=2,\ldots,K$  there exists $\ell_j^*\in\{(j-1)\overline\tau+1, \ldots,j\overline\tau\}$ such that $REL$ prescribes to activate all states  $\vec\pi_j^{i,1}$ for which $i\geq \ell_j^*-(j-1)\overline\tau$, and for all $j=1,\ldots,K$ there exists $m_j^*\in\{(K+j-1)\overline\tau+2,\ldots,(K+j)\overline\tau+1\}$ such that $REL$ prescribes to activate all states  $\vec\pi_j^{i,2}$ for which $i\geq m_j^*-(j-1)\overline\tau$. In state $\vec\pi_{1}^{\ell_1^*,1}$ the policy $REL$ prescribes to activate the users in that state with probability $\rho\in(0,1)$.
\begin{remark} Observe that we exclude the possibility $\rho=1$. It can be seen that a non-randomized policy, which corresponds to $\rho=1$, is optimal only for a finite number of $\lambda$s, Weber et al. \cite{WW90}. 
\end{remark}

We have that
\begin{align}\label{thedifferences}
&\mathbf{y}(t+1)-\mathbf{y}(t)\bigg|_{\mathbf{y}(t)=\mathbf{y}}=\sum_{i=1}^{2(K\overline\tau+1)}\sum_{j=1}^{2(K\overline\tau+1)}q_{ij}(\mathbf{y})\vec e_{ij}y_i\nonumber\\
&=\sum_{i=1}^{\ell_1^*-1}\sum_{j=1}^{2(K\overline\tau+1)}q_{ij}(\mathbf{y})\vec e_{ij}y_i    +\sum_{i=\ell_1^*+1}^{2(K\overline\tau+1)}\sum_{j=1}^{2(K\overline\tau+1)}q_{ij}(\mathbf{y})\vec e_{ij}y_i+y_{\ell_1^*}\sum_{j=1}^{2(K\overline\tau+1)}q_{\ell_1^*j}(\mathbf{y})\vec e_{\ell_1^*j}\nonumber\\
&=\sum_{i\neq\ell_1^*}\sum_{j=1}^{2(K\overline\tau+1)}q_{ij}(\mathbf{y})\vec e_{ij}     +y_{\ell_1^*}\sum_{j=1}^{2(K\overline\tau+1)}[g_{\ell_1^*}(\mathbf{y})q_{\ell_1^*j}^1+(1-g_{\ell_1^*}(\mathbf{y}))q_{\ell_1^*j}^0)]\vec e_{\ell_1^*j}\nonumber\\
&=\sum_{i\neq \ell_1^*}\sum_{j=1}^{2(K\overline\tau+1)}q_{ij}(\mathbf{y})\vec e_{ij}y_i+y_{\ell_1^*}\sum_{j=1}^{2(K\overline\tau+1)}q^0_{\ell_1^*j}\vec e_{\ell_1^*j}    + g_{\ell_1^*}(\mathbf{y})y_{\ell_1^*}\sum_{j=1}^{2(K\overline\tau+1)}[q_{\ell_1^*j}^1-q_{\ell_1^*j}^0]\vec e_{\ell_1^*j}.
\end{align}
The second inequality in the latter equation follows from the definition of $q_{ij}(\mathbf{y})$ in Equation~\eqref{eq:probs}. Note that by the definitions of $\ell_1^*$ (defined in the beginning of this section) and $g_{\ell_1^*}(\mathbf{y})$ imply
\begin{align*}
g_{\ell_1^*}(\mathbf{y})y_{\ell_1^*}=\lambda-\sum_{i:W_i>W^*}y_i.
\end{align*}
Substituting the latter in Equation~\eqref{thedifferences} we obtain
\begin{align*}
&\mathbf{y}(t+1)-\mathbf{y}(t)\bigg|_{\mathbf{y}(t)=\mathbf{y}}=\sum_{i\neq \ell_1^*}\sum_{j=1}^{2(K\overline\tau+1)}q_{ij}(\mathbf{y})\vec e_{ij}y_i+y_{\ell_1^*}\sum_{j=1}^{2(K\overline\tau+1)}q^0_{\ell_1^*j}\vec e_{\ell_1^*j} + (\lambda-\sum_{i:W_i>W^*}y_i)\sum_{j=1}^{2(K\overline\tau+1)}[q_{\ell_1^*j}^1-q_{\ell_1^*j}^0]\vec e_{\ell_1^*j}.
\end{align*}
In the latter equation $q_{ij}(\mathbf{y})$ for all $i\neq\ell_1^*$ stays constant for all $\mathbf{y}\in\overline Y_{W^*}$, since $g_{i}(\mathbf{y})$ for all $i\neq\ell_1^*$ is either 0 or 1 and therefore independent of $\mathbf{y}$.

\subsection{Proof of Lemma~\ref{lema:hola}}\label{append_fixpoint}
We want to show that $\mathbf{\theta}_{\delta,\lambda}$ is the unique zero of $\overline Q\mathbf{y}+\overline d=0$.

It is clear that $\overline Q\mathbf{\theta}_{\delta,\lambda}+ \overline d=0$, since $\mathbf{\theta}_{\delta,\lambda}$ is the mean of the random vector to which the system $\mathbf{Y}^{N}(t)$  under $REL$ converges, and the fluid system is defined by the mean drift of the system $\mathbf{Y}^{N}(t)$. We assume there exists $\overline y\neq \mathbf{\theta}_{\delta,\lambda}$ such that $\overline Q \overline y+\overline d=0$, then there exists a policy characterized by $\overline W$ and $\overline \rho$ (i.e., allocate a pilot to all users with $W_i>\overline W$, idle if $W_i<\overline W$ and randomize with probability $\overline \rho$ if $W_i=\overline W$) for which the steady-state vector is given by $\overline y$ and the average fraction of activated users equals $\lambda$. This is however in contradiction with the indexability property which implies that a unique $\overline  W$ and $\overline\rho$ exist for each $\lambda$ (Lemma 1 in \cite{WW90}).

To conclude the proof, we mention that $\mathbf{\theta}_{\delta,\lambda}$ is independent of $N$, the proof follows from Lemma~4 in~\cite{OuyangEryilShroff16}.

\subsection{Proof of Proposition~\ref{prop:local_optimality_WIP}}\label{proof_local}

The local asymptotic optimality can be obtained in two steps.  

\noindent
{\it Step~1:} We prove that for an initial state $\mathbf{y}(0)\in\mathcal{N}(\mathbf{\theta}_{\delta,\lambda})$ the fluid system converges to $\mathbf{\theta}_{\delta,\lambda}$.  

\noindent
{\it Step~2:} We show that the system $\mathbf{Y}^N(t)$ can be made arbitrarily close to the fluid system $\mathbf{y}(t)$ as $N\to\infty$. 

\subsubsection{Step~1}

To prove {\it Step~1} we are going to (i) obtain the explicit expression of the linear fluid system, (ii) prove the eigenvalues of this system, i.e., $\iota$, to satisfy $|\iota+1|<1$, and (iii) we will prove that $\mathbf{y}(t)\to\mathbf{\theta}_{\delta,\lambda}$.

We are now going to write the explicit expression of the difference $\mathbf{y}(t+1)-\mathbf{y}(t)$. For simplicity, we reduce the dimension of vector $\mathbf{y}(t)$ by one. This reduction can be done due to  the fact that $
\sum_{i=1}^{K\overline\tau+1}y_i=\delta_1,
$
for all $\mathbf{y}\in\mathcal{Y}$ and the fact that if $\mathbf{y}(0)\in\mathcal{Y}$ then $\mathbf{y}(t)\in\mathcal{Y}$. For all $\mathbf{y}\in\mathcal{Y}$ we define $\mathbf{\hat y}=(y_1,\ldots,y_{\ell_1^*-1},y_{\ell_1^*+1},\ldots,y_{2(K\overline\tau+1)})$. With a bit of abuse of notation, we let $\vec e_{ij}$ be the vector of dimension $2K\overline\tau +1$ with  all entries 0s except the $i^{\text{th}}$ term which equals -1 and the $j^{\text{th}}$ which equals 1, and we let $q_{ij}(\mathbf{\hat y})$ be defined as in Equation~\eqref{eq:probs} for vectors of dimension $2K\overline\tau +1$. Therefore, we have
\begin{align*}
&\mathbf{\hat y}(t+1)-\mathbf{\hat y}(t)\bigg|_{\mathbf{\hat y}(t)=\mathbf{\hat y}}=\sum_{i\neq \ell_1^*}\sum_{j\neq \ell_1^*}q_{ij}(\mathbf{\hat y})\vec e_{ij} y_i\\
&\quad+\left(\delta_1-\sum_{i=1}^{\ell_1^*-1}y_i-\sum_{i=\ell_1^*+1}^{2(K\overline\tau+1)}y_i\right)\sum_{j\neq\ell_1^*}q^0_{\ell_1^*j}\vec e_{\ell_1^*j}\\
&\quad + (\lambda-\sum_{i:W_i>W^*}y_i)\sum_{j\neq\ell_1^*}[q_{\ell_1^*j}^1-q_{\ell_1^*j}^0]\vec e_{\ell_1^*j}\\
&=\sum_{i:W_i<W^*}\sum_{j\neq\ell_1^*}[q_{ij}(\mathbf{\hat y})\vec e_{ij}-q_{\ell_1^*j}^0\vec e_{\ell_1^*j}]y_i\\
&\quad+\sum_{i:W_i>W^*}\sum_{j\neq\ell_1^*}[q_{ij}(\mathbf{\hat y})\vec e_{ij}-q_{\ell_1^*j}^1\vec e_{\ell_1^*j}]y_i\\
&\quad+\delta_1 \sum_{j\neq\ell_1^*}q^0_{\ell_1^*j}\vec e_{\ell_1^*j} + \lambda \sum_{j\neq\ell_1^*}[q_{\ell_1^*j}^1-q_{\ell_1^*j}^0]\vec e_{\ell_1^*j}.
\end{align*}
Where we used Equation~\eqref{thedifferences},  $
\sum_{i=1}^{K\overline\tau+1}y_i=\delta_1
$, and $g_{\ell_1^*}(\mathbf{y})y_{\ell_1^*}=\lambda-\sum_{j:W_j>W^*}y_i$.
One can then derive the expression 
\begin{align}\label{low_dimension_fluid}
\mathbf{\hat y}(t+1)-\mathbf{\hat y}(t)=\hat Q\mathbf{\hat y}+\hat d,
\end{align}
where $\hat d = \delta_1 \sum_{j\neq\ell_1^*}q^0_{\ell_1^*j}\vec e_{\ell_1^*j} + \lambda \sum_{j\neq\ell_1^*}[q_{\ell_1^*j}^1-q_{\ell_1^*j}^0]\vec e_{\ell_1^*j}$, and 
$$
\hat Q=\begin{bmatrix}
Q_1^1& \ldots &Q_K^1&Q_1^2 &\ldots&Q_K^2\\
\vec 0 & \ldots &\vec 0&O_1^2&\ldots&O_K^2 
\end{bmatrix}
$$
The explicit expressions of $ Q_k^c$ for all $k\in\{1,\ldots,K\}$ and all $c\in\{1,2\}$, can be found in~\eqref{matrix_Q}. In order to simplify the expression in~\eqref{matrix_Q} we have used the following notation, $0_{n\times m}$ represents the matrix of size $n\times m$ whose entries are all $0$ and $-I_n$ refers to the negative identity matrix of size $n\times n$. 

\begin{table*}
\begin{minipage}{0.95\textwidth}\centering
\begin{align}\label{matrix_Q}
&Q_1^1=\begin{bmatrix}
A_{\ell_1^*-1}&0_{(\ell_1^*-1)\times(\tau-\ell_1^*)}\\
B_{(\tau-\ell_1^*)\times (\ell_1^*-1)}&-I_{\tau-\ell_1^*}\\
0_{((K-1)\tau+1)\times\ell_1^*-1}&0_{((K-1)\tau+1)\times(\tau-\ell_1^*)}\\
\end{bmatrix},\hbox{ where }
A_{m}=\overbrace{\begin{bmatrix}
-1&0&\ldots&0&0\\
1&-1&\ldots&0&0\\
0&1&\ldots&0&0\\
\vdots&\vdots&\ddots&\vdots&\vdots\\
0&0&\ldots&1&-1
\end{bmatrix}}^{m},\quad
B_{n\times m}=\overbrace{\begin{bmatrix}
-1&\ldots&-1&0\\
0&\ldots&0&0\\
\vdots&\ddots&\vdots&\vdots\\
0&\ldots&0&0
\end{bmatrix}}^{m},\nonumber\\
&Q_i^1=\begin{bmatrix}
0_{(\ell_1^*-1)\times\ell_i^*} & 0_{(\ell_1^*-1)\times(\tau-\ell_i^*)}\\
B_{(\tau-\ell_1^*)\times\ell_i^*}&0_{(\tau-\ell_1^*)\times (\tau-\ell_i^*)}\\
0_{(i-2)\tau\times\ell_i^*}&0_{(i-2)\tau\times(\tau-\ell_i^*)}\\
A_{\ell_i^*}&0_{\ell_i^*\times(\tau-\ell_i^*)}\\
0_{(\tau-\ell_i^*)\times\ell_i^*}& -I_{\tau-\ell_i^*}\\
0_{(K\tau+1-i\tau)\times\ell_i^*}&0_{(K\tau+1-i\tau)\times(\tau-\ell_i^*)}
\end{bmatrix},\forall\,\, i\in\{2,\ldots,K-1\},
Q_K^1=\begin{bmatrix}
0_{(\ell_1^*-1)\times\ell_K^*} & 0_{(\ell_1^*-1)\times(\tau+1-\ell_K^*)}\\
B_{(\tau-\ell_1^*)\times\ell_K^*}&0_{(\tau-\ell_1^*)\times (\tau+1-\ell_K^*)}\\
0_{(K-2)\tau\times\ell_K^*}&0_{(K-2)\tau\times(\tau+1-\ell_K^*)}\\
A_{\ell_K^*}&0_{\ell_K^*\times(\tau+1-\ell_K^*)}\\
0_{(\tau+1-\ell_K^*)\times\ell_K^*}& -I_{\tau+1-\ell_K^*}
\end{bmatrix}.\nonumber\\
&Q_i^2=\begin{bmatrix}
0_{(\ell_1^*-1)\times m_i^*} & 0_{(\ell_1^*-1)\times(\tau-m_i^*)}\\
B_{(\tau-\ell_1^*)\times m_i^*}&0_{(\tau-\ell_1^*)\times (\tau-m_i^*)}\\
0_{((K-1)\tau+1)\times m_i^*}&0_{((K-1)\tau+1)\times(\tau-m_i^*)}
\end{bmatrix},\forall\,\, i\in\{1,\ldots,K-1\},
Q_K^2=\begin{bmatrix}
0_{(\ell_1^*-1)\times m_K^*} & 0_{(\ell_1^*-1)\times(\tau+1-m_K^*)}\\
B_{(\tau-\ell_1^*)\times m_K^*}&0_{(\tau-\ell_1^*)\times (\tau+1-m_K^*)}\\
0_{((K-1)\tau+1)\times m_K^*}&0_{((K-1)\tau+1)\times(\tau+1-m_K^*)}\\
\end{bmatrix},\nonumber\\
&O_i^2=\begin{bmatrix}
0_{(i-1)\tau\times m_i^*}&0_{(i-1)\tau\times(\tau-m_i^*)}\\
A_{m_i^*}&0_{m_i^*\times(\tau-m_i^*)}\\
0_{(\tau-m_i^*)\times m_i^*}& -I_{\tau-m_i^*}\\
0_{(K\tau+1-i\tau)\times m_i^*}&0_{(K\tau+1-i\tau)\times(\tau-m_i^*)}
\end{bmatrix},\forall\,\, i\in\{1,\ldots,K-1\},
O_K^2=\begin{bmatrix}
0_{(K-1)\tau\times m_K^*}&0_{(K-1)\tau\times(\tau+1-m_K^*)}\\
A_{m_K^*}&0_{m_K^*\times(\tau+1-m_K^*)}\\
0_{(\tau+1-m_K^*)\times m_K^*}& -I_{\tau+1-m_K^*}
\end{bmatrix},
\end{align}
\hrule
\end{minipage}
\end{table*}

In the next lemma we prove that the eigenvalues of $\hat Q$ satisfy $|\iota+1|<1$.
\begin{lemma}\label{eigen}
The eigenvalues of $\hat Q$, i.e., $\iota$, satisfy $|\iota+1|<1$.
\end{lemma}
\begin{proof}
 We compute the eigenvalues of $\hat Q$, that is,
compute $\iota$ the solution of 
\begin{align*}
\hbox{det}(\hat Q-\iota I_{2K\overline\tau+1})=&\hbox{det}([Q_{1}^1,\ldots,Q_K^1]-\iota I_{K\overline\tau})\\
&\cdot\hbox{det}([O_{1}^2,\ldots,O_K^2]-\iota I_{K\overline\tau+1})=0,
\end{align*}
due to the property of block matrices. Note that matrices $[Q_{1}^1,\ldots,Q_K^1]$ and $[O_1^2,\ldots,O_K^2]$ are square matrices. Analyzing the structures of $Q_i^1$ and $O_i^2$ for all $i$ we obtain that
  \begin{align}\label{eq:determinant}
&\hbox{det}(\hat Q-\iota I_{2K\overline\tau+1})\nonumber\\
 &=\hbox{det}(A_{\ell_1^*-1}-\iota I_{\ell_1^*1})\hbox{det}(A_{\ell_2^*}-\iota I_{\ell_2^*})\cdot\ldots\nonumber\\
&\quad\cdot\hbox{det}(A_{\ell_K^*}-\iota I_{\ell_K^*})\cdot \hbox{det}(-I_{\overline\tau-\ell_1^*}-\iota I_{\overline\tau-\ell_1^*})\cdot\ldots\nonumber\\
&\quad\cdot\hbox{det}(-I_{\overline\tau-\ell_{K-1}^*}-\iota I_{\overline\tau-\ell_{K-1}^*})\hbox{det}(-I_{\overline\tau+1-\ell_{K}^*}-\iota I_{\overline\tau+1-\ell_{K}^*})\nonumber\\
 &\quad\cdot\hbox{det}(A_{m_1^*-1}-\iota I_{m_1^*})\hbox{det}(A_{m_2^*}-\iota I_{m_2^*})\cdot\ldots\nonumber\\
&\quad\cdot\hbox{det}(A_{m_K^*}-\iota I_{m_K^*})\cdot \hbox{det}(-I_{\overline\tau-m_1^*}-\iota I_{\overline\tau-m_1^*})\cdot\ldots\nonumber\\
&\quad\cdot\hbox{det}(-I_{\overline\tau-m_{K-1}^*}-\iota I_{\overline\tau-m_{K-1}^*})\nonumber\\
&\quad\cdot\hbox{det}(-I_{\overline\tau+1-m_{K}^*}-\iota I_{\overline\tau+1-m_{K}^*})=0.
\end{align}
The latter is obtained exploiting the properties of block matrices. 
It is easy to see that Equation~\eqref{eq:determinant} reduces to 
\begin{align*}
\hbox{det}(\hat Q-\iota I_{2K\overline\tau+1})=(-1-\iota)^{2K\overline\tau+1}=0,
\end{align*}
therefore, all eigenvalues equal $-1$, and consequently $|\iota+1|<1$. This concludes the proof.

\end{proof}

Having proven that for all eigenvalues of the system $|\iota+1|<1$ we prove the following.

\begin{lemma}\label{a_lemma}
Let $y(0)=\mathbf{y}$ and assume there exists $\varepsilon>0$ such that, if $\mathbf{y}(0)\in\mathcal{N}_\varepsilon(\mathbf{\theta}_{\delta,\lambda})\subset \overline Y_{W^*}$, that is, the initial point is in the neighborhood of $\mathbf{\theta}_{\delta,\lambda}$ then 
(1) $y(t)\in\mathcal{Y}_W$ for all $t$, and (2) $y(t)\to \mathbf{\theta}_{\delta,\lambda}$ as $t\to\infty$.   
\end{lemma}
\begin{proof} The proof of this lemma follows from the arguments in Lemma~12 in \cite{OuyangEryilShroff16} and relies in the following results.
\begin{itemize}
\item $\mathbf{\theta}_{\delta,\lambda}\in\overline Y_{W^*}$. 
To prove the latter  it suffices to recall that, from the definition of $g_i(\mathbf{y}(t))$ in Table~I, if $\mathbf{y}(t)=\mathbf{\theta}_{\delta,\lambda}$ then $\sum_{j:W_j\geq W^*}g_i(\mathbf{y}(t))y_i(t)=\lambda$, therefore $\mathbf{\theta}_{\delta,\lambda}\in\overline Y_{W^*}$.
\item  The assumption on $\rho\neq 1$ allows us to ensure $\overline Y_{W^*}\neq\{\mathbf{\theta}_{\delta,\lambda}\}$, that is, there exist state vectors in $\mathcal{Y}$, other than the steady-state, that belong to the set $\overline Y_{W^*}$. Therefore, there exists $\varepsilon_0>\varepsilon$ such that $\mathcal{N}_{\varepsilon_0}(\mathcal{\theta}_{\delta,\lambda})\subset \overline Y_{W^*}$, and  $\mathcal{N}_{\varepsilon_0}(\mathcal{\theta}_{\delta,\lambda})\neq \emptyset$.
\item Equation~\eqref{low_dimension_fluid} which ensure the fluid system to be linear in $\overline Y_{W^*}$.
\item Lemma~\ref{eigen} which implies convergence of $ \mathbf{\hat{y}}(t)\to\mathbf{\theta}_{\delta,\lambda}$ as $t\to\infty$.
\end{itemize}
\end{proof}

\subsubsection{Step~2}
In what follows we are going to state three lemmas and a proposition that will allow to establish the local asymptotic optimality result for Whittle's index policy. The proofs of these lemmas can be obtained by slightly adapting the results obtained in~\cite{OuyangEryilShroff16}.

\begin{lemma}
There exits $\mathcal{N}_{\epsilon}(\mathcal{\theta}_{\delta,\lambda})$, a neighborhood of $\mathbf{\theta}_{\delta,\lambda}$ such that for all $\nu>0$ there exists $\mathcal{N}_{\epsilon}(\vec y_{\delta,\alpha})$ such that for all $\mathbf{y}\in\mathcal{N}_{\epsilon}(\mathbf{\theta}_{\delta,\lambda})$ there exists $f(\cdot)$ independent of $N$ and $\mathbf{y}$ such that
\begin{align*}
&\mathbb{P}(\|\mathbf{Y}^N(t+1)-(I+Q(\mathbf{y}))\mathbf{y}\|\geq \nu|\mathbf{Y}^N(t)=\mathbf{y})\\
&\leq 2K\mathrm{e}^{-N\cdot f(\nu)}.
\end{align*}
\begin{proof}The proof can be obtained following the proof of Lemma~17 in~\cite{OuyangEryilShroff16}.\end{proof}
\end{lemma}
\begin{lemma}\label{and_yet_another_lemma}
Let $\mathbf{Y}^N(0)=\mathbf{y}$. Assume there exists a neighborhood $\mathcal{N}_\psi(\mathbf{\theta}_{\delta,\lambda})$ such that for all $\nu>0$, if $\mathbf{y}\in\mathcal{N}_\psi(\mathbf{\theta}_{\delta,\lambda})$ there exists $\beta_1^t$ and $\beta_2^t$, independent of $N$ and $\mathbf{y}$ for which
\begin{align*}
\mathbb{P}_\mathbf{y}(\|\mathbf{Y}^N(t)-y(t)\|\geq \nu)\leq \beta_1^t\mathrm{e}^{-N\cdot \beta_2^t}, \forall t=1,2,\ldots
\end{align*}
\end{lemma}
\begin{proof}
The proof follows from the proof of Lemma~18 in~\cite{OuyangEryilShroff16}.
\end{proof}
\begin{proposition}\label{prop:last_for_local_optimality}
Let $\mathbf{Y}^N(0)=y(0)=\mathbf{y}$. There exists a neighborhood $\mathcal{N}_{\psi}(\mathbf{\theta}_{\delta,\lambda})$ such that, for all $\mathbf{y}\in\mathcal{N}_{\psi}(\mathbf{\theta}_{\delta,\lambda})$, all $\nu>0$ and all time horizon $T<\infty$, there exists positive constants $C_1$ and $C_2$, independent of $N$ and $\mathbf{y}$, such that
\begin{align*}
\mathbb{P}_\mathbf{y}(\sup_{0\leq t<T}\|\mathbf{Y}^N(t)-y(t)\|\geq \nu)\leq C_1\mathrm{e}^{-N\cdot C_2},
\end{align*}
where $\psi<\varepsilon$ (where $\varepsilon$ has been defined in Lemma~\ref{a_lemma}).
\end{proposition}
\begin{proof}
Note that 
\begin{align*}
&\mathbb{P}_\mathbf{y}(\sup_{0\leq t<T}\|\mathbf{Y}^N(t)-y(t)\|\geq \nu)\\
&\leq \sum_{t=0}^{T-1}\mathbb{P}_{\mathbf{y}}(\|\mathbf{Y}^N(t)-y(t)\|\geq\nu)\\
&\leq \sum_{t=0}^T\beta_1^t\mathrm{e}^{-N\cdot\beta_2^t},
\end{align*}
where the first inequality follows from Boole's inequality and the second inequality from Lemma~\ref{and_yet_another_lemma}. Let us now define $C_2=\min_{0\leq t< T}\{\beta_2^t\}$ then
\begin{align*}
\sum_{t=0}^T\beta_1^t\mathrm{e}^{-N\cdot\beta_2^t}\leq\mathrm{e}^{-N\cdot C_2}\sum_{t=0}^{T-1}\beta_1^t\mathrm{e}^{-N\cdot (\beta_2^t-C_2)}\leq C_1\mathrm{e}^{-N\cdot C_2},
\end{align*}
where $C_1=\sum_{t=0}^{T-1}\beta_1^t$. The second inequality in the latter equation follows from the fact that $\beta_2^t-C_2\geq 0$ for all $t\geq0$. This concludes the proof.
\end{proof}

\begin{lemma}\label{lemmalemma}
Let $\mathbf{Y}^N(0)=\mathbf{y}$. For all $\mathbf{y}\in\mathcal{N}_{\psi}(\mathbf{\theta}_{\delta,\lambda})$ and all $\nu>0$, there exists $T_0$ such that for all $T>T_0$, there exists positive constant $k_1$ and $k_2$ such that 
\begin{align*}
\mathbb{P}_{\mathbf{y}}(\sup_{T_0\leq t<T}\|\mathbf{Y}^N(t)-\mathbf{\theta}_{\delta,\lambda}\|\geq\nu)\leq k_1\mathrm{e}^{-N\cdot k_2}.
\end{align*}
\end{lemma}
\begin{proof}The proof can be found in \cite[Lemma 13]{OuyangEryilShroff16} and it essentially  follows from Proposition~\ref{prop:last_for_local_optimality}.\end{proof}

To conclude the {\it Step 2} of the proof of Proposition~\ref{prop:local_optimality_WIP}, it now suffices to show that 
 there exist $\psi$ and $\mathcal{N}_{\psi}(\mathbf{\theta}_{\delta,\lambda})$ such that
\begin{align*}
\lim_{T\to\infty}\lim_{r\to\infty}\frac{R^{WIP,N_r}_T(\mathbf{y})}{N_r}=R^{REL}.
\end{align*}
To do so we first define $R(\mathbf{y})$ to be the average reward accrued by each user in the systems state $\mathbf{y}\in\mathbf{Y}$. The latter implies $NR(\mathbf{Y}^N(t))$ to be the immediate reward at time $t$. Note that $R^{REL}=R(\mathbf{\theta} y_{\delta,\lambda})$.
 
 Let $\omega>0$ and $\nu>0$ such that for all $\mathbf{y}\in\mathcal{Y}$, 
 \begin{align*}
 |R(\mathbf{y})-R(\mathbf{\theta}_{\delta,\lambda})|<\omega,
 \end{align*}
 if $\|\mathbf{y}-\mathbf{\theta}_{\delta,\lambda}\|<\nu$.
 
 Let $N_r\in\mathbb{Z}$ be a positive sequence of integers such that $\lambda N_r,\delta_cN_r\in\mathbb{Z}$ for all $c\in\{1,2\}$. We then have the following
 \begin{align}\label{oneequation}
 &\bigg|\frac{R^{N_r,WIP}_T(\mathbf{y})}{N_r}-R^{REL}\bigg|=\bigg|\frac{1}{N_rT}\mathbb{E}\left(\sum_{t=0}^{T-1}N_rR(\mathbf{Y}^{N_r}(t))\right)-R^{REL}\bigg|\nonumber\\
 &=\bigg|\frac{1}{T}\sum_{t=0}^{T_0-1}\mathbb{E}(R(\mathbf{Y}^{N_r}(t)))+\frac{1}{T}\sum_{t=T_0}^{T-1}\mathbb{E}(R(\mathbf{Y}^{N_r}(t)))-\frac{T_0+(T-T_0)}{T}R^{REL}\bigg|\nonumber\\
 &\leq\bigg|\frac{1}{T}\sum_{t=0}^{T_0-1}\mathbb{E}(R(\mathbf{Y}^{N_r}(t))-R^{REL})\bigg|+\bigg|\frac{1}{T}\sum_{t=T_0}^{T-1}\mathbb{E}(R(\mathbf{Y}^{N_r}(t))-R^{REL})\bigg|\nonumber\\
 &\leq R^1\frac{T_0}{T}+\bigg|\frac{1}{T}\sum_{t=T_0}^{T-1}\mathbb{E}(R(\mathbf{Y}^{N_r}(t))-R^{REL})\bigg|.
 \end{align}
 The last inequality follows from the fact that the per user average reward cannot exceed $R^1$. Now note that 
 \begin{align}\label{equation_another}
 &\bigg|\frac{1}{T}\sum_{t=T_0}^{T-1}\mathbb{E}(R(\mathbf{Y}^{N_r}(t))-R^{REL})\bigg|\nonumber\\
 &\leq\mathbb{P}_{\mathbf{y}}(\sup_{T_0\leq t\leq T}\|\mathbf{Y}^{N_r}(t)-\mathbf{\theta}_{\delta,\lambda}\|\geq\nu)     \cdot\frac{1}{T}\sum_{t=T_0}^{T-1}\mathbb{E}(|R(\mathbf{Y}^{N_r}(t))-R^{REL}|\bigg|A_{N_r})\nonumber\\
 &\quad+(1-\mathbb{P}_{\mathbf{y}}(\sup_{T_0\leq t\leq T}\|\mathbf{Y}^{N_r}(t)-\mathbf{\theta}_{\delta,\lambda}\|\geq\nu))     \cdot \frac{1}{T}\sum_{t=T_0}^{T-1}\mathbb{E}(|R(\mathbf{Y}^{N_r}(t))-R^{REL}|\bigg|\overline A_{N_r})\nonumber\\
 &\leq R^1\left(\mathbb{P}_{\mathbf{y}}(\sup_{T_0\leq t\leq T}\|\mathbf{Y}^{N_r}(t)-\mathbf{\theta}_{\delta,\lambda}\|\geq\nu)(1-\omega)+\omega\right),
 \end{align}
 where $A_{N_r}$ represents the event that $\sup_{T_0\leq t\leq T}\|\mathbf{Y}^{N_r}(t)-\mathbf{\theta}_{\delta,\lambda}\geq\nu\|)$ and $A_{N_r}$ its complementary. The last inequality follows form the fact that $R(\mathbf{y})\leq R^1$ and the fact that $|R(\mathbf{y})-R(\mathbf{\theta}_{\delta,\lambda})|<\omega$ for all $\|\mathbf{y}-\mathbf{\theta}_{\delta,\lambda}\|<\nu$.
 
 From Lemma~\ref{lemmalemma}, for all $\mathbf{y}\in\mathcal{N}(\mathbf{\theta}_{\delta,\lambda})$ we have 
 \begin{align*}
\lim_{r\to\infty} \mathbb{P}(\sup_{T_0\leq t<T}\|\mathbf{Y}^{N_r}(t)-\mathbf{\theta}_{\delta,\lambda}\|\geq\nu)\leq \lim_{r\to\infty}k_1\mathrm{e}^{-N_r\cdot k_2}=0.
 \end{align*}
 Hence, using the latter in Equations~\eqref{oneequation} and~\eqref{equation_another}, we deduce
 \begin{align*}
 \lim_{r\to\infty}\bigg|\frac{R^{WIP,N_r}_T(\mathbf{y})}{N_r}-R^{REL}\bigg|\leq R^1\frac{T_0}{T}+R^1\omega,
 \end{align*}
 with $\omega$ arbitrarily small. Therefore
 \begin{align*}
 \lim_{T\to\infty}\lim_{r\to\infty}\frac{R^{N_r,WIP}_T(\mathbf{y})}{N_r}=R^{REL}.
 \end{align*}

 \subsection{Proof of Lemma~\ref{lemma:recurrentclass}}\label{proof_lemma_recurrent}
This proof follows the same line of ideas as those in Appendix~E in~\cite{OuyangEryilShroff16}.

{\bf Proof of item 1)  in Lemma~\ref{lemma:recurrentclass}:} First we are going to prove that the Markov chain is aperiodic and has a single recurrent class. Let us define $i_1=\min_i\{\ell_{i}^*\}$ and $i_2=\min_i\{m_i^*\}$ and we assume w.l.o.g. that $W(\vec \pi_{i_1}^{1,1})\geq W(\vec \pi_{i_2}^{1,2})$,  with $\ell_i^*$ and $m_i^*$ for all $i$, as defined in the beginning of  Appendix~\ref{proof_linear_fluid}. We are going to prove that from any initial state $\mathbf{Y}^N(0)=\mathbf{y}$, the following states can be reached:
 \begin{itemize}[noitemsep]
 \item State vector $\mathbf{Y}^N=[\mathbf{Y}^{1,N},\mathbf{Y}^{2,N}]$ with $Y_{i_1,1}^{1,N}=\lambda,$ $Y_{s}^{1,N}=\delta_1-\lambda$ and $\mathbf{Y}_s^{2,N}=\delta_2$, and all other entries $0$, if $\lambda\leq\delta_1$.
 \item State vector $\mathbf{Y}^N=[\mathbf{Y}^{1,N},\mathbf{Y}^{2,N}]$ with $Y_{i_1,1}^{1,N}=\delta_1,$ $\mathbf{Y}_{i_2,1}^{2,N}=\lambda-\delta_1$, $\mathbf{Y}_{s}^{2,N}=1-\lambda$,  and all other entries $0$, if $\lambda>\delta_1$.
 \end{itemize}
 To reach the state introduced in the first item above note that the following can occur. Given an initial state $\mathbf{y}$ for all the users of class 1 that have been allocated with a pilot (out of all the activated $\lambda$ fraction of users under $WIP$), we observe channel state $i_1$. All the class-2 users that have been activated happen to be in channel state $i_2$. After a long enough period $\mathbf{Y}^N=[\mathbf{Y}^{1,N},\mathbf{Y}^{2,N}]$ with $Y_{i_1,1}^{1,N}=\lambda,$ $Y_{s}^{1,N}=\delta_1-\lambda$ and $\mathbf{Y}_s^{2,N}=\delta_2$, and all other entries $0$, will be reached.
 
 If instead $\lambda>\delta_1$ the same event as introduced above can occur. That is, every class-1 user that is allocated with a pilot happens to be in channel state $i_1$ and every class-2 user allocated with a pilot happens to be in channel state $i_2$. Then the state $\mathbf{Y}^N=[\mathbf{Y}^{1,N},\mathbf{Y}^{2,N}]$ with $Y_{i_1,1}^{1,N}=\delta_1,$ $\mathbf{Y}_{i_2,1}^{2,N}=\lambda-\delta_1$, $\mathbf{Y}_{s}^{2,N}=1-\lambda$,  and all other entries $0$, is reached under $WIP$ policy.
 
 We are going to denote this recurrent state by $\mathbf{Y}^{N}_{rec}$.
 
 Aperiodicity of the Markov chain is given, since by the path that we have described above the transition from  $\mathbf{Y}^{N}_{rec}$ to itself is possible.
 
{\bf Proof of item 2)  in Lemma~\ref{lemma:recurrentclass}:}  For notational ease, let us denote the steady-state vector $\mathbf{\theta}_{\delta,\lambda}$ by $\mathbf{\theta}$ throughout the proof. Note that $\mathbf{\theta}=[\theta^1,\theta^2]$ is such that
 \begin{align*}
 &\theta^1_{i,1}=\ldots=\theta_{i,\ell_i^*}^1, \hbox{ for all } i\in\{1,\ldots,K\},\\
 &\theta^1_{1,\ell_1^*+1}=(1-\rho)\theta_{1,\ell_1^*}^1, \hbox{ and } \\
 &\theta^2_{i,1}=\ldots=\theta_{i,m_i^*}^2, \hbox{ for all } i\in\{1,\ldots,K\},
 \end{align*}
and all the other entries equal $0$. The objective is to show that there exists a path that under $WIP$ will bring the system to state $\theta$ having started in state $\mathbf{Y}_{rec}^N$. A remark on the procedure to construct this path is in order.
 \begin{remark}As it has been highlighted in~\cite[Appendix F]{OuyangEryilShroff16} we are going to consider that channels are splittable. We explain next what this property implies. Note that $WIP$ prescribes to activate the fraction of users in belief states for which the Whittle's index is highest. Let us assume that for $\pi_1,\pi_2,\ldots,\pi_L\in\overline\Pi=\Pi_1\cup\Pi_2$, $W(\pi_1)\geq \ldots\geq W(\pi_L)$, $W(\pi)\leq W(\pi_L)$ for all $\pi\in\overline\Pi\backslash\{\pi_1,\ldots,\pi_L\}$, and $\sum_{i=1}^Ly_i>\lambda$ and $\sum_{i=1}^{L-1}y_i<\lambda$ (with $y_i$ the fraction of users in belief state $\pi$). If channels where unsplittable $WIP$ would prescribe to activate all users in belief states $\pi_i, i\in\{1,\ldots,L-1\}$ leading to a fraction of activated users $\sum_{i=1}^{L-1}y_i=\overline \lambda<\lambda$. To avoid this from happening, we assume channels to be splittable and therefore allow $WIP$ to activate only a fraction of users in belief state $\pi_L$, leading to the fraction of activated users to equal $\lambda$. Through this assumption, a path from $\mathbf{Y}_{rec}^N$ to $\theta$ can be constructed (done below). The authors in~\cite{OuyangEryilShroff16} argue that for large enough $N$ a path with unsplittable channels under $WIP$ can be arbitrarily close to the exact path (built exploiting the splittable property of the channels)  that brings the system from $\mathbf{Y}_{rec}^N$ to $\theta$.
 \end{remark}
We construct the path from $\mathbf{Y}_{rec}^N$ to $\theta$ next. We are going to assume  $W(\vec \pi^{s,1})\geq W(\vec\pi^{s,2})$. The other case can be studied similarly. 
Let us define $h_i:=\max\{r:W(\vec\pi_{i}^{\ell_i^*+r,1})\leq W(\vec \pi^{s,2})\}$, and let $h_{max}=\max_{i}\{h_i\}$.

{\bf Step 1:} We want to build a path from $\mathbf{Y}_{rec}^N$ to $\theta$. Let us assume that the permutation $\sigma$ is such that $\ell_{\sigma(1)}^*\geq \ell_{\sigma(2)}^*\geq\ldots \geq\ell_{\sigma(K)}^*$. We are going to assume that \emph{in the first $\ell_{\sigma(1)}^*-\ell_{\sigma(2)}^*$ time slots}, out of the $\lambda$ activated users, a fraction $\theta_{\sigma(1),1}^1$ of class-1 users happen to be in channel $\sigma(1)$. The rest of activated users remain in channel $i_1$ for class-1 users and in $i_2$ for class-2 users. That is, after this first period the path we have constructed brings the system to the state
 \begin{align*}
&\mathbf{Y}_{\sigma(1),r}^{1,N}=\theta_{\sigma(1),1}^1, \hbox{ for all } r\in\{1,\ldots,\ell_{\sigma(1)}^*-\ell_{\sigma(2)}^*\},\\  &\mathbf{Y}_{i_1,1}^{1,N}+\mathbf{Y}_s^{1,N}=\delta_1-(\ell_{\sigma(1)}^*-\ell_{\sigma(2)}^*)\theta_{\sigma(1),1}^1,\\
&\mathbf{Y}_{i_2,1}^{2,N}+\mathbf{Y}_s^{2,N}=\delta_2.
 \end{align*}

Following the same arguments, \emph{in the next $\ell_{\sigma(2)}^*-\ell_{\sigma(3)}^*$ time slots}, we assume that out of the $\lambda$ activated fraction of users, $\theta_{\sigma(1),1}^1$ fraction of class-1 users are in channel $\sigma(1)$, $\theta_{\sigma(2),1}^1$ are in channel state $\sigma(2)$ and all the other activated users are in channel $i_1$ if the users belong to class 1 and in channel $i_2$ if the user belongs to class 2. Therefore, after this period we reach the following state
 \begin{align*}
&\mathbf{Y}_{\sigma(1),r}^{1,N}=\theta_{\sigma(1),1}^1, \hbox{ for all } r\in\{1,\ldots,\ell_{\sigma(1)}^*-\ell_{\sigma(3)}^*\},\\ 
&\mathbf{Y}_{\sigma(2),r}^{1,N}=\theta_{\sigma(2),1}^1, \hbox{ for all } r\in\{1,\ldots,\ell_{\sigma(2)}^*-\ell_{\sigma(3)}^*\},\\ 
 &\mathbf{Y}_{i_1,1}^{1,N}+\mathbf{Y}_s^{1,N}=\delta_1-(\ell_{\sigma(1)}^*-\ell_{\sigma(3)}^*)\theta_{\sigma(1),1}^1\\
&\qquad\qquad\qquad\quad-(\ell_{\sigma(2)}^*-\ell_{\sigma(3)^*}^*)\theta_{\sigma(2),1}^1,\\
&\mathbf{Y}_{i_2,1}^{2,N}+\mathbf{Y}_s^{2,N}=\delta_2.
 \end{align*}
 This process is repeated for other $\ell_{\sigma(3)}^*$ time slots, and at the end of it we obtain
  \begin{align*}
&\mathbf{Y}_{\sigma(i),r}^{1,N}=\theta_{\sigma(i),1}^1, \hbox{ for all } r\in\{1,\ldots,\ell_{\sigma(i)}^*\}, \text{ and } \sigma(i)\neq 1,i_1,\\ 
&\mathbf{Y}_{\sigma(j),r}^{1,N}=\theta_{\sigma(j),1}^1, \hbox{ for all } r\in\{1,\ldots,\ell_{\sigma(j)}^*\},\\
&\mathbf{Y}_{\sigma(j),r}^{1,N}=(1-\rho)\theta_{\sigma(j),1}^1,\hbox{ with } \sigma(j)=1, r=\ell_1^*+1,\\ 
 &\mathbf{Y}_{i_1,1}^{1,N}+\mathbf{Y}_s^{1,N}=\delta_1-\sum_{i=1}^K\ell_{\sigma(i)}^*\theta_{\sigma(i),1}^1-(1-\rho)\theta^1_{1,1},\\
&\mathbf{Y}_{i_2,1}^{2,N}+\mathbf{Y}_s^{2,N}=\delta_2.
 \end{align*}
In the time slot in which the channel $\sigma(j)=1$ for class-1 users receives a fraction of users for the first time, we assume the received fraction of users to equal $\theta_{1,1}^1(1-\rho)$, and not $\theta_{1,1}^1$ as in every other case.
 
{\bf Step 2:} By definition of $i_2$, in the belief  states that correspond to $\mathbf{Y}^{1,N}_{i,\ell_i^*+h_i+1}$ for all $i=1,\ldots,K$,  Whittle's index, i.e., $W(\vec\pi_{\ell_{i}^*+h_i+1}^{1,1})$, satisfies $W(\vec\pi_{\ell_{i}^*+h_i+1}^{1,1})\geq W(\vec\pi_{i_2}^{1,2})$. We will assume that for \emph{$x$ time slots} all fraction of users that occupy the state $\mathbf{Y}^{1,N}_{i,\ell_i^*+h_i+1}$ for all $i$ after activation they happen to be in the same channel state $i$. All the class-2 users that are activated happen to be in state $i_2$. Therefore, at the end of Step~2, if $x=0\mod L$ (where $L$ is the least common multiple of all $\ell_i^*+h_i$) we recover the same state that we had at the end of Step~1. We are however interested in finding $x$ such that $x+\max_i\{m_i^*\}=0 \mod L$ in which 
\begin{align*}
&\sum_{i=1}^K\sum_{r=1}^{\ell_i^*+h_i}\mathbf{Y}_{i,r}^{1,N}+(1-\rho)\theta_{1,1}^1=\delta_1,\\
&\mathbf{Y}_{i_2,1}^{2,N}+\mathbf{Y}_{s}^{2,N}=\delta_2.
\end{align*}
In the latter we have that $\sum_{i=1}^Kh_i$ entries in $\mathbf{Y}_{i,j}^{1,N}$ for all $i$ and all $j\in\{1,\ldots,\ell_i^*+h_i\}$ equal $0$. The position that these $0$s occupy is determined by $x$.

{\bf Step 3:} In this last \emph{period of length $\max_i\{m_i^*\}$ time slots} we mimic the path followed in Step~1 but with respect to class-2 users. That is, we assume the permutation $\vartheta$ to be such that  $m_{\vartheta(1)}^*\geq m_{\vartheta(2)}^*\geq\ldots \geq m_{\vartheta(K)}^*$. We are going to assume that \emph{in the first $m_{\vartheta(1)}^*-m_{\vartheta(2)}^*$ time slots}, out of the $\lambda$ activated users, a fraction $\theta_{\vartheta(1),1}^2$ of class-2 users happen to be in channel $\vartheta(1)$. The rest of activated users remain in channel $i_2$ for class-2 users. The fraction of class-1 users in states $\vec\pi_{i}^{\ell_i^*+h_i+1,1}$ happen to be in channel state $i$ after activation. Hence we obtain
\begin{align*}
&\sum_{i=1}^K\sum_{r=1}^{\ell_i^*+h_i}\mathbf{Y}_{i,r}^{1,N}+(1-\rho)\theta_{1,1}^1=\delta_1,\\
&\mathbf{Y}_{\vartheta(1),r}^{2,N}=\theta_{\vartheta(1),1}^2, \hbox{ for all } r\in\{1,\ldots,m_{\vartheta(1)}^*-m_{\vartheta(2)}^*\},\\ 
&\mathbf{Y}_{i_2,1}^{2,N}+\mathbf{Y}_{s}^{2,N}=\delta_2-(m_{\vartheta(1)}^*-m_{\vartheta(2)}^*)\theta_{\vartheta(1),1}^2.
\end{align*}
 We follow this process as done in Step-1 until we reach the state
 $\mathbf{Y}^{2,N}_{i,r}=\theta_{i,1}^2$ for all $i\in\{1,\ldots,K\}$ and all $r\in\{1,\ldots,m_i^*\}$ for class-2 users. Since we have assumed in the previous step that $x+\max_i\{m_i^*\}=0 \mod L$, we know that in Step 3 of length $\max_i\{m_i^*\}$ we reach the state 
   \begin{align*}
&\mathbf{Y}_{\sigma(i),r}^{1,N}=\theta_{\sigma(i),1}^1, \hbox{ for all } r\in\{1,\ldots,\ell_{\sigma(i)}^*\},\\ 
&\mathbf{Y}_{\sigma(j),r}^{1,N}=\theta_{\sigma(j),1}^1, \hbox{ for all } r\in\{1,\ldots,\ell_{\sigma(j)}^*\},\\
&\mathbf{Y}_{\sigma(j),r}^{1,N}=(1-\rho)\theta_{\sigma(j),1}^1,\hbox{ with } \sigma(j)=1, 
 \end{align*}
for class-1 users. We have therefore reached state $\theta$.
This concludes the proof.

\end{document}